\documentclass[11pt]{amsart}

\usepackage{amscd,amsmath,amssymb,amsfonts,amsthm,ascmac}
\usepackage{enumerate,mathrsfs,stmaryrd,latexsym, comment, mathtools, mathdots} 

\usepackage{amsthm, amssymb}
\usepackage[leqno]{amsmath}
\usepackage{caption}
\usepackage{subcaption}
\usepackage{hyperref}
\usepackage{relsize}
\usepackage{booktabs}

\usepackage[all]{xy}

\usepackage{tikz-cd}
\usepackage{tikz}
\usetikzlibrary{snakes, 
	3d, matrix,decorations.pathreplacing,calc,decorations.pathmorphing}
\usepackage{tikz-3dplot}
	
\numberwithin{equation}{section}
\allowdisplaybreaks[1]

\makeatletter
\newcommand{\leqnomode}{\tagsleft@true\let\veqno\@@leqno}
\newcommand{\reqnomode}{\tagsleft@false\let\veqno\@@eqno}
\makeatother

\newcommand{\defi}[1]{{\textit{#1}}}

\newcommand{\flag}{{\mathcal{F} \ell}}
\newcommand{\al}{{ \alpha}}
\newcommand{\C}{{\mathbb{C}}}

\newcommand{\R}{{\mathbb{R}}}

\newcommand{\Z}{{\mathbb{Z}}}
\newcommand{\Cstar}{{\C^{\ast}}}

\newcommand{\xii}[2]{{\xi^{{(#1)}}_{{#2}}}}
\newcommand{\bolda}[3]{{\mathbf{a}^{{(#1)}}_{{{#2}}, {{#3}}}}}

\DeclareMathOperator{\Pic}{Pic}
\DeclareMathOperator{\Lie}{Lie}

\DeclareMathOperator{\SL}{SL}
\DeclareMathOperator{\GL}{GL}

\newtheorem{theorem}{Theorem}[section]
\newtheorem{lemma}[theorem]{Lemma}
\newtheorem{proposition}[theorem]{Proposition}
\newtheorem{corollary}[theorem]{Corollary}

\theoremstyle{definition}
\newtheorem{example}[theorem]{Example}
\newtheorem{definition}[theorem]{Definition}
\newtheorem{remark}[theorem]{Remark}

\begin{document}

\author{Naoki Fujita}
\address{Graduate School of Mathematical Sciences, The University of Tokyo, 3-8-1 Komaba, Meguro-ku, Tokyo 153-8914, Japan}
\email{nfujita@ms.u-tokyo.ac.jp}

\author{Eunjeong Lee}
\address{Center for Geometry and Physics, Institute for Basic Science (IBS), Pohang 37673, Korea}
\email{eunjeong.lee@ibs.re.kr}

\author{Dong Youp Suh}
\address{Department of Mathematical Sciences
	\\ KAIST \\ 291 Daehak-ro Yuseong-gu \\ Daejeon \\ Republic of Korea 34141}
\email{dysuh@kaist.ac.kr}



\thanks{
Fujita was partially supported by Grant-in-Aid for JSPS Fellows (No. 16J00420).
Lee was partially supported by IBS-R003-D1.
Suh was partially supported by Basic
Science Research Program through the National Research Foundation of Korea (NRF) funded by the Ministry of Science, ICT \& Future Planning (No. 2016R1A2B4010823).
}

\title[Flag Bott--Samelson varieties]{Algebraic and geometric properties of \\
	flag Bott--Samelson varieties and \\
	applications to representations}

\maketitle


\begin{abstract}
We define and study flag Bott--Samelson varieties which generalize both Bott--Samelson varieties and flag varieties.
Using a birational morphism from an appropriate Bott--Samelson variety to a flag Bott--Samelson variety, we compute the Newton--Okounkov bodies of  flag Bott--Samelson varieties as generalized string polytopes, which are applied to give polyhedral expressions for irreducible decompositions of  tensor products of  $G$-modules. Furthermore, we show that flag Bott--Samelson varieties degenerate into  flag Bott manifolds with higher rank torus actions, and we describe the Duistermaat--Heckman measures of the moment map images of  flag Bott--Samelson varieties with torus actions and invariant closed 2-forms.
\end{abstract}

\section{Introduction}

Bott--Samelson varieties provide fruitful connections between representation theory and algebraic geometry. 
They are nonsingular towers of $\C P^1$-fibrations, and studied in \cite{BoSa58}, \cite{De74}, and \cite{Hansen73} to find resolutions of singularities of  Schubert varieties. 
Moreover, the set of global sections of a holomorphic line 
bundle over a Bott--Samelson variety is the dual of a \defi{generalized Demazure module}. This leads to worthwhile connections
between representation theory and algebraic geometry exemplified by
the character formulas of  Demazure modules in~\cite{And85, Ku87}, the standard monomial theory in~\cite{LMS79, LS91, LaRa08, Ses14}, and 
the theory of Newton--Okounkov bodies in~\cite{Fuj17, Kav15}.

On the other hand, for a given Bott--Samelson variety, it is presented by Grossberg and Karshon~\cite{GrKa94} that there is a complex one-parameter family of smooth varieties, which are all diffeomorphic, such that a generic fiber is the Bott--Samelson variety and the special fiber is a nonsingular toric variety, called a \defi{Bott manifold}. 
It should be noted that the Bott manifold is toric, while 
the Bott--Samelson variety is not toric in general.
Using this connection, Grossberg and Karshon~\cite{GrKa94} also introduces a combinatorial object, called \defi{Grossberg--Karshon twisted cubes}, which is used to compute the multiplicities of generalized Demazure modules  (see~\cite[Theorem~3]{GrKa94}).

One of the primary goals of this paper is to generalize the notion of Bott--Samelson varieties and to extend its rich connections with representation theory. Moreover, the generalization also supports the Grossberg--Karshon type degeneration into flag Bott manifolds. 
Indeed, we consider a \defi{flag Bott--Samelson variety} (see Definition~\ref{definition:flag Bott--Samelson variety}) which
is a nonsingular projective  tower of products of full flag manifolds. 
Moreover, under a certain condition, the flag Bott--Samelson variety  is a desingularization
of a Schubert variety. Because of the definition, 
both the full  flag varieties  and Bott--Samelson varieties are flag Bott--Samelson varieties. 
Hence we may regard flag Bott--Samelson varieties as the generalization of 
both flag varieties and Bott--Samelson varieties. 

This notion of flag Bott--Samelson varieties is not new. Actually, in~\cite{Jan07}, flag Bott--Samelson varieties are
treated in a more general setting without naming them. Indeed, flag Bott--Samelson varieties are iterated flag manifold fibrations but Jantzen~\cite{Jan07} considers iterated Schubert varieties fibrations. 
Perrin~\cite{Perrin07small} uses these varieties to obtain small resolutions of Schubert varieties. In fact, they are called \textit{generalized Bott--Samelson varieties} (see~\cite{brion2019minimal, brion2019some} and references therein). Moreover, flag Bott--Samelson varieties are generalized Bott--Samelson varieties (see Remark~\ref{rmk_Perrin_gBS}).

Let $G$ be a simply-connected semisimple algebraic group of rank $n$ over $\C$. 
Bott--Samelson varieties $Z_{\mathbf i}$ are parametrized by words $\mathbf i= (i_1,\dots,i_r)$, where $i_1,\dots,i_r$ are elements in the set $[n] \coloneqq \{1,\dots,n\}$. On the other hand,  flag Bott--Samelson varieties $Z_{\mathcal I}$ are parametrized by sequences $\mathcal I = (I_1,\dots,I_r)$ of subsets of $[n]$. 
Even though the class of  flag Bott--Samelson varieties is much larger than that of Bott--Samelson varieties, for each  flag Bott--Samelson variety, there exists a Bott--Samelson variety such that there is a birational morphism from the Bott--Samelson variety to the flag Bott--Samelson variety (see Proposition~\ref{prop_fBS_and_BS}).

Using the above mentioned birational morphism, we provide Theorem~\ref{thm_isom_between_holom_sections}, which describes the set of holomorphic sections of a holomorphic line bundle over the flag Bott--Samelson variety in terms of generalized Demazure modules. 
The theory of Newton--Okounkov bodies of projective varieties has been used to present a connection between representation theory and algebraic geometry (see \S\ref{subsect_NOBY} for the definition of Newton--Okounkov bodies).
The description of holomorphic sections of flag Bott--Samelson varieties is used to compute their Newton--Okounkov bodies. Indeed, using the result of Newton--Okounkov bodies of Bott--Samelson varieties by Fujita~\cite{Fuj18}, we obtain Theorem~\ref{thm_NOBY_of_fBS}, which shows that the Newton--Okounkov bodies of flag Bott--Samelson varieties with an appropriate valuation agree with generalized string polytopes up to sign.

One of the fundamental questions in group representation theory is to find the multiplicities of irreducible representations in the tensor product of two representations. Berenstein and Zelevinsky~\cite{BeZe01} describes the multiplicities in terms of the numbers of lattice points in some explicit rational convex polytope. 
In Theorem~\ref{c:tensor product multiplicity} we give a different description of the multiplicities using
the integral lattice points of the Newton--Okounkov bodies, hence generalized string polytopes, of flag
Bott--Samelson varieties. We notice that our results give concrete constructions of convex bodies, appearing in~\cite{KaKh12a}, which encode multiplicities of irreducible representations.

As is mentioned before, we degenerate the complex structures of flag Bott--Samelson varieties.
The notion of Bott manifolds is generalized to that of \defi{flag Bott manifolds} in terms of iterated flag manifold fibrations described in~\cite{KKLS, KLSS}. More precisely, a flag Bott manifold is the total space of an iterated flag manifold fibrations which are taken by the full flag fibration of a sum of line bundles (see Definition~\ref{def_fBT}). 
In general, a flag Bott manifold is not toric but admits an action of a certain torus.
For a given flag Bott--Samelson variety, we provide a complex one-parameter family of smooth varieties, which are all diffeomorphic, such that a generic fiber is the flag Bott--Samelson variety and the special fiber is a {flag Bott manifold} (see  Corollary~\ref{cor_fBS_are_fBT_general}). Moreover, when the Levi subgroup $L_{I_k}$ of the parabolic subgroup $P_{I_k}$ is of type $A$, we explicitly describe such flag Bott manifolds in Theorem~\ref{thm_fBS_is_fBT} in terms of the Chern classes of the line bundles used in the construction of the flag Bott manifold.

For a given flag Bott--Samelson variety, there exists a Bott--Samelson variety which is birationally equivalent to the flag Bott--Samelson variety. Moreover, using the result of Grossberg and Karshon~\cite{GrKa94}, and our one-parameter family, we obtain two manifolds: a flag Bott manifold and a Bott manifold, and a map between them. We study a relation between these manifolds.
Actually, considering torus actions on these manifolds, we describe the Duistermaat--Heckman measure of the flag Bott manifold with a certain closed $2$-form using a Grossberg--Karshon twisted cube  in Theorem~\ref{thm_DH_measure_of_fBS}. 

This paper is organized as follows.
In Section~\ref{sec_NOBY_of_fBS}, we provide the definition of flag Bott--Samelson varieties and their properties. In particular, we investigate a relation between flag Bott--Samelson varieties and Bott--Samelson varieties. Moreover, we describe the set of holomorphic sections of a line bundle over  a flag Bott--Samelson variety using generalized Demazure modules in Theorem~\ref{thm_isom_between_holom_sections}. Using this association, we describe the Newton--Okounkov bodies of flag Bott--Samelson varieties in Theorem~\ref{thm_NOBY_of_fBS}.
In Section~\ref{sec_appli_to_RT}, we give an application of Newton--Okounkov bodies of flag Bott--Samelson varieties to representation theory.
Indeed, we provide a way to compute the multiplicities of representations in the tensor product of representations counting certain lattice points in the Newton--Okounkov bodies of flag Bott--Samelson varieties in Theorem~\ref{c:tensor product multiplicity}. 
In Section~\ref{sec_fBS_and_fBT}, we present a Grossberg--Karshon type degeneration of complex structures on flag Bott--Samelson varieties, and explicitly describe the corresponding flag Bott manifold when all Levi subgroups of parabolic subgroups $P_{I_k}$ are of type $A$ in Theorem~\ref{thm_fBS_is_fBT}. 
In Section~\ref{sec_torus_actions_on_fBS}, we study torus actions on flag Bott manifolds which are obtained by the degeneration of flag Bott--Samelson manifolds. Moreover, we describe the Duistermaat--Heckman measure of flag Bott manifolds using Grossberg--Karshon twisted cubes  in Theorem~\ref{thm_GK_DH_measure}. 

\smallskip
\noindent\textbf{Acknowledgments.}
We are grateful to the anonymous referee for a careful reading and helpful comments to improve our initial manuscript.
\section{Newton--Okounkov bodies of flag Bott--Samelson varieties}
\label{sec_NOBY_of_fBS}
\subsection{Definition of flag Bott--Samelson varieties}
\label{subsec_def_of_fBS}
In this subsection we introduce flag Bott--Samelson varieties
which generalize both Bott--Samelson varieties and flag varieties, 
and study their properties. 
We notice that the notion of flag Bott--Samelson varieties is already
considered in Jantzen's book~\cite[II.13]{Jan07}
without naming it. 

Let $G$ be a simply-connected semisimple algebraic group of rank $n$ over $\mathbb C$. Choose a Cartan subgroup $H$, and let $\mathfrak{g} = \mathfrak{h} 
\oplus \sum_{\alpha} \mathfrak{g}_{\alpha}$ be the decomposition of the Lie algebra $\mathfrak{g}$ of $G$ into root spaces, where $\mathfrak{h}$ is the Lie algebra of $H$.
Let $\Delta \subset \mathfrak{h}^{\ast}$ denote the roots of $G$. Choose a set of 
positive roots $\Delta^+ \subset \Delta$, and let $\Sigma = \{\alpha_1,\dots,\alpha_n\} \subset \Delta^+$ denote the simple roots. Let $\Delta^- \coloneqq - \Delta^+$ be the set of negative roots. 
Let $B$ be the Borel subgroup whose Lie algebra is 
$\mathfrak{b} = \mathfrak{h} \oplus \sum_{\alpha \in \Delta^+} \mathfrak{g}_{\alpha}$.
Let
$\{\alpha_1^{\vee},\dots,\alpha_n^{\vee}\}$ denote the coroots, and 
$\{\varpi_1,\dots,\varpi_n\}$ the fundamental weights which are characterized by the relation $\langle \varpi_i, \alpha_j^{\vee} \rangle = \delta_{ij}$. Here, $\delta_{ij}$ denotes the Kronecker symbol. Let $s_i \in W$ denote the simple reflection in the Weyl group $W$ of $G$ corresponding to the simple root $\alpha_i$. 

For a subset $I$ of $[n]\coloneqq \{1,\dots,n\}$, define the subtorus $H_I \subset H$ as
\begin{equation}\label{eq_def_of_HI}
H_I \coloneqq \{h \in H \mid \alpha_i(h) = 1 \text{ for all } i \in I\}^0.
\end{equation}
Here, for a group $G$, $G^0$ is the connected component which contains the identity element of $G$.
Then the centralizer $C_{G}(H_I)
= \{g \in G \mid g h = h g \text{ for all } h \in H_I\}$ of $H_I$
is a connected reductive subgroup of $G$ whose
Weyl group is isomorphic to $W_I \coloneqq \langle s_{i} \mid i \in I \rangle$. 
We set $L_I \coloneqq C_{G}(H_I)$ for simplicity.
Then the Borel subgroup $B_I$ of $L_I$ is $B \cap L_I$ (see~\cite[\S 8.4.1]{Springer09}).
Let $\Delta_I$ be the subset $\Delta \cap \text{span}_{\mathbb Z}\{\alpha_i \mid i \in I \}$ of $\Delta$. 
The set of roots $\Delta^+ \setminus \Delta_I$ defines 
the unipotent subgroup $U_I$ of $G$ satisfying the condition
\[
\Lie(U_I) = \bigoplus_{\alpha \in \Delta^+ \setminus \Delta_I} \mathfrak{g}_{\alpha}.
\]
The \defi{parabolic subgroup} $P_I$ of $G$ corresponding to $I$ is
defined to be $P_I \coloneqq L_I U_I$. The subgroup $L_I$ is called 
a \defi{Levi subgroup} of $P_I$. 

Note that the Lie algebra of the parabolic subgroup $P_I$ is 
\[
\Lie (P_I) =  \mathfrak{b}\oplus \bigoplus_{\alpha \in \Delta^- \cap \Delta_I} 
\mathfrak{g}_{\alpha}.
\]
Moreover the parabolic subgroup $P_I$ can be described that
\[
P_I = \bigcup_{w \in W_I} BwB = \overline{B w_I B} \subset G,
\]
where $w_I$ be the longest element in $W_I$ (see~\cite[Theorem~8.4.3]{Springer09}).

We now define the flag Bott--Samelson variety using a sequence of parabolic subgroups.
Let $\mathcal{I} = (I_1,\dots,I_r)$ be a sequence of 
subsets of $[n]$, and let
$\mathbf{P}_{\mathcal I} = P_{I_1} \times \cdots \times P_{I_r}$.
Define a right action $\Theta$ of $B^r = \underbrace{B \times \cdots \times B}_r$ on $\mathbf P_{\mathcal I}$ as
\begin{equation}\label{eq_Bm-action_defining_fBS}
\Theta((p_1,\dots,p_r) , (b_1,\dots,b_r))
= (p_1b_1, b_1^{-1}p_2b_2,\dots, b_{r-1}^{-1}p_rb_r)
\end{equation}
for $(p_1,\dots,p_r) \in \mathbf{P}_{\mathcal I}$ and
$(b_1,\dots,b_r) \in B^r$.
\begin{definition}\label{definition:flag Bott--Samelson variety}
	Let $\mathcal{I} = (I_1,\dots,I_r)$ be a sequence of 
	subsets of $[n]$.
	The \defi{flag Bott--Samelson variety} $Z_{\mathcal I}$ is defined 
	to be the 	orbit space 
	\[
	Z_{\mathcal I} \coloneqq \mathbf{P}_{\mathcal I}/\Theta.
	\]
\end{definition}
For instance, suppose that $\mathcal I = ([n])$. Then we have 
$\mathbf P_{\mathcal I} = G$ and the action $\Theta$
is the right multiplication of $B$. Therefore
the flag Bott--Samelson variety $Z_{\mathcal I}$
is the flag variety $G/B$.  Moreover, for the case when $|I_k|=1$ for all $k$, the flag Bott--Samelson 
variety is a \defi{Bott--Samelson variety},
see~\cite{BoSa58} for the definition of Bott--Samelson variety.
In this case we use a sequence $(i_1,\dots,i_r)$ of elements of $[n]$
rather than $(\{i_1\}, \dots, \{i_r\})$, and we write 
$Z_{(i_1,\dots,i_r)}$ for
the corresponding Bott--Samelson variety. 

For the subsequence $\mathcal I' = (I_1,\dots,I_{r-1})$ of $\mathcal I$, 
there is a fibration structure
on the flag Bott--Samelson variety $Z_{\mathcal I}$:
\begin{equation}\label{eq_bdle_str_on_fBS}
P_{I_r}/B \hookrightarrow 
Z_{\mathcal I} \stackrel{\pi}{\longrightarrow} Z_{\mathcal I'},
\end{equation}
where the projection map 
$\pi \colon Z_{\mathcal I} \to Z_{\mathcal I'}$ is defined as 
\[
\pi([p_1,\dots,p_{r-1},p_r]) = [p_1,\dots,p_{r-1}].
\]
On the other hand, we have that
\[
Z_{\mathcal I} = P_{I_1} \times^B Z_{(I_2,\dots,I_r)}.
\]

\begin{remark}\label{rmk_Perrin_gBS}
	For a finite sequence $\widehat{w} = (w_1,\dots,w_r)$ of elements of $W$, 
	Perrin~\cite{Perrin07small} considers a tower $\widehat{X}(\widehat{w})$ of Schubert varieties $X(w_1),\dots,X(w_r)$ fibrations and call it a \defi{generalized Bott--Samelson variety}.
	Let $\mathcal I = (I_1,\dots,I_r)$ be a sequence of subsets of $[n]$. When we take $w_k$ to be the longest element in $W_{I_k}$ for $1 \leq k \leq r$,  the generalized Bott--Samelson variety $\widehat{X}(w_1,\dots,w_r)$ is the flag Bott--Samelson variety $Z_{\mathcal I}$.
	Indeed, using the notations in~\cite[\S 5.2]{Perrin07small},  we obtain
	$P^{w_k} = P_{[n] \setminus I_k}, P_{w_k} = P_{I_k}, G_{w_k} = L_{I_k}$.	Therefore, $P^{w_k} \cap G_{w_k}$ is a Borel subgroup $B_{I_k}$ of $ L_{I_k}$ and $P_{w_k} \cap  L_{I_k} =  L_{I_k}$. Thus, we obtain
	\[
	\begin{split}
	\widehat{X}(w_1,\dots,w_r)  &= \overline{(P_{w_1} \cap G_{w_1}) w_1 (P^{w_1} \cap G_{w_1})} \times^{(P^{w_1} \cap G_{w_1})} \widehat{X}(w_2,\dots,w_{r}) \\
	&= \overline{ B_{I_1} w_1 L_{I_1}} \times^{B_{I_1}} \widehat{X}(w_2,\dots,w_{r}) \\
	&= L_{I_1} \times^{B_{I_1}} \widehat{X}(w_2,\dots,w_{r}) \\
	&\simeq P_{I_1} \times^B\widehat{X}(w_2,\dots,w_{r}).
	\end{split}
	\]
	Continuing this procedure, we get $\widehat{X}(w_1,\dots,w_r) \simeq Z_{\mathcal I}$. This shows that flag Bott--Samelson varieties are generalized Bott--Samelson varieties. Because we are considering sequences $\widehat{w} = (w_1,\dots,w_r)$ consisting of  longest elements, not all generalized Bott--Samelson varieties are flag Bott--Samelson varieties.
\end{remark}

Let $w_k \in W_{I_k}$ be the longest element in $W_{I_k}$
for $1 \leq k \leq r$.
Consider the following subset of $\mathbf{P}_{\mathcal I}$:
\[
\mathbf{P}_{\mathcal I}' \coloneqq 
B w_{1} B \times \cdots \times 
B w_{r} B \subset \mathbf{P}_{\mathcal I}. 
\]
One can check that $\mathbf{P}_{\mathcal I}'$ is closed under
the action $\Theta$ of $B^r$ in~\eqref{eq_Bm-action_defining_fBS}, 
so we consider the orbit space
\[
Z_{\mathcal I}' \coloneqq \mathbf{P}_{\mathcal I}'/\Theta.
\]
It is known that flag Bott--Samelson varieties $Z_{\mathcal I}$ have
following properties (see~\cite[II.13]{Jan07} for details).
\begin{proposition}\label{prop_properties_of_fBS}
	Let $\mathcal{I} = (I_1,\dots,I_r)$ be a sequence of 
	subsets of $[n]$. Then the flag Bott--Samelson variety $Z_{\mathcal I}$ has following properties:
	\begin{enumerate}
		\item $Z_{\mathcal I}$ is 
		a smooth projective variety.
		\item $Z_{\mathcal I}'$ is a dense open subset in
		$Z_{\mathcal I}$.
		\item $Z_{\mathcal I}' \simeq 
		\mathbb{C}^{\sum_{k=1}^r \ell(w_{k})}$,
		where $\ell(w_k)$ is the length of the element $w_k$.
	\end{enumerate}
\end{proposition}

Consider the multiplication map 
\begin{equation}\label{eq_fBS_to_flag}
\eta \colon Z_{\mathcal I} \to G/B,
\quad [p_1,\dots,p_r] \mapsto p_1 \cdots p_r
\end{equation}
which is a well-defined morphism. 
The following proposition says that
certain flag Bott--Samelson varieties 
are birationally equivalent to Schubert varieties 
via the map $\eta$.
\begin{proposition}[{\cite[II.13.5]{Jan07}}]\label{prop_fBS_and_flag}
	Let $\mathcal{I} = (I_1,\dots,I_r)$ be a sequence of 
	subsets of $[n]$, and
	let $w_{k} \in W_{I_k}$ be the longest element in $W_{I_k}
	= \langle s_i \mid i \in I_k \rangle$.
	Set $ w= w_{1} \cdots w_{r}$. If 
	$\ell(w) = \ell(w_1) + \cdots + \ell(w_r)$, then the morphism $\eta$
	induces an isomorphism between
	$Z_{\mathcal I}'$ and $B w B/B \subset G/B$.
	Indeed, the morphism $\eta$ maps $Z_{\mathcal I}$ birationally onto its image
	$X(w) \coloneqq \overline{B w B/B} \subset G/B$. 
\end{proposition}
\begin{example}\label{example_biratioanl_morphism_fBS_and_Sch_var}
	Let $G = \SL(4)$.
	\begin{enumerate}
		\item Suppose that $\mathcal I_1 = 	(\{1\}, \{2\}, \{1\}, \{3\})$.
		Then we have 
		$w_1 = s_1, w_2 = s_2, w_3 = s_1, w_4 = s_3$, 
		and $w = s_1s_2s_1s_3$, which 
		is a reduced decomposition. Hence the morphism
		$\eta$ gives a birational morphism between
		$Z_{\mathcal I_1}$ and $X(s_1s_2s_1s_3)$.	
		\item Let $\mathcal I_2 = 
		(\{ 1,2 \}, \{3\})$. 
		Then we have that $w_1 = s_1s_2s_1$, $w_2 = s_3$,
		and $w = w_1w_2 = s_1s_2s_1s_3$. Again, this is a reduced
		decomposition, so the morphism $\eta$ gives a birational
		morphism between $Z_{\mathcal I_2}$ and $X(s_1s_2s_1s_3)$. 
	\end{enumerate}
\end{example}
\begin{remark}
	Example~\ref{example_biratioanl_morphism_fBS_and_Sch_var} 
	gives two different choices of flag Bott--Samelson varieties
	each of which has a birational morphism onto the same Schubert variety
	$X(s_1s_2s_1s_3)$. For a given Schubert variety $X(w)$, 
	there are different choices of flag Bott--Samelson varieties 
	which define birational morphisms
	onto $X(w)$, and there are several studies about such different choices.
	See, for example,~\cite{El97, Es16, Ten06}.
\end{remark}

We now define a multiplication
map between two flag Bott--Samelson varieties. Let 
\begin{equation}\label{eq_set_J}
\mathcal J = (J_{1,1},
\dots,J_{1,N_1},
\dots,J_{r,1},\dots,J_{r,N_r})
\end{equation}
be a sequence of subsets of $[n]$ such that 
$J_{k,l} \subset I_k$ for $1 \leq l \leq N_k$ and $1 \leq k \leq r$. 
Since each $J_{k,l}$ is contained in $I_k$,
we have $P_{J_{k,l}} \subset P_{I_k}$
 by the definition of parabolic subgroups.
Hence we have a multiplication map 
\begin{equation}\label{eq_mul_two_fBS}
\eta_{\mathcal J,\mathcal I}
\colon Z_{\mathcal J} \to Z_{\mathcal I}
\end{equation}
defined as 
\[
[(p_{k,l})_{1 \leq k \leq r,1\leq l \leq N_k}]
\mapsto 
\left[ 
\prod_{l=1}^{N_1} p_{1,l},\dots,
\prod_{l = 1}^{N_r} p_{r,l}
\right].
\]
The following proposition describes 
a birational morphism between two flag Bott--Samelson
varieties.
\begin{proposition}\label{prop_fBS_and_BS}
Let $\mathcal{I} = (I_1,\dots,I_r)$ be a sequence of 
subsets of $[n]$, and let
	$\mathcal J = (J_{1,1},\dots,J_{1,N_1},
	\dots,J_{r,1},\dots,J_{r,N_r})$ 
	be a sequence of subsets of $[n]$ such that 
	$J_{k,1},\dots,J_{k,N_k} \subset I_k$ for $1 \leq k \leq r$.
	Let $w_{k,l}$, respectively $v_k$, be the longest element in $W_{J_{k,l}}$, respectively in $W_{I_k}$. 
	Suppose that $w_{{k,1}} \cdots w_{k,N_k} = v_k$ and 
	$\ell(w_{k,1}) + \cdots + \ell(w_{k,N_k}) = \ell(v_k)$
	for $1\leq k\leq r$. 
	Then the multiplication map $\eta_{\mathcal J,\mathcal I}
	\colon Z_{\mathcal J} \to Z_{\mathcal I}$ 
	in~\eqref{eq_mul_two_fBS} induces an isomorphism 
	between dense open subsets $Z_{\mathcal J}' 
	\stackrel{\sim}{\longrightarrow} Z_{\mathcal I}'$. 
\end{proposition}
There always exists a sequence $(i_{k,1},\dots,i_{k,N_k}) \in [n]^{N_k}$ which is a reduced word for the longest element in~$W_{I_k}$ for $1 \leq k \leq r$. Concatenating such sequences we get a sequence
$\mathbf i = (i_{k,l})_{1 \leq k \leq r, 1 \leq l \leq N_k} \in [n]^{N_1 + \cdots + N_r}$.
 Hence for a given flag Bott--Samelson variety $Z_{\mathcal I}$ one can always find a Bott--Samelson variety $Z_{\mathbf i}$ which is birationally isomorphic to~$Z_{\mathcal I}$.
\begin{proof}[Proof of Proposition~\ref{prop_fBS_and_BS}]
	First we recall from~\cite[VI. \S 1, Corollary~2 of Proposition~17]{Bourbaki02Lie} and~\cite[II.13.1]{Jan07}
	that for a reduced decomposition
	$w = s_{i_1} \cdots s_{i_N} \in W$, the subgroup $U(w) \subset G$ is defined to be
	\[
	U(w) \coloneqq U_{\alpha_{i_1}} \cdot U_{s_{i_1}(\alpha_{i_2})} \cdot 
	U_{s_{i_1} s_{i_2}(\alpha_{i_3})} \cdots
		U_{s_{i_1} \cdots s_{i_{N-1}}(\alpha_{i_N})}.
	\]
	Moreover, we have an isomorphism
	\begin{equation}\label{eq_U_and_BwB}
	\psi(w) \colon U_{\alpha_{i_1}} \times U_{\alpha_{i_2}} \times \cdots \times U_{\alpha_{i_N}} \stackrel{\sim}{\longrightarrow} U(w)
	\end{equation}
	which is defined to be $(u_1,\dots,u_N) \mapsto u_1 s_{i_1} u_2 s_{i_2} \cdots u_N s_{i_N} w^{-1}$.
	Also we have another isomorphism $\psi_{\mathcal I}$
	between varieties:
	\begin{equation}\label{eq_U_and_ZI}
	\psi_{\mathcal I} \colon U(v_1) \times \cdots \times U(v_r) 
	\stackrel{\sim}{\longrightarrow} Z_{\mathcal I}'
	\end{equation}
	which sends $(g_1,\dots,g_r)$ to $[g_1v_1,\dots,g_rv_r]$ (see~\cite[II.13.5]{Jan07}).
	
	Because of the assumption,
	the concatenation $w_{k,1} \cdots w_{k,N_k}$
	is a reduced decomposition of the element $v_k$. Hence
	we have an isomorphism induced by~\eqref{eq_U_and_BwB}:
	\[
	\psi_k \colon U(w_{k,1}) \times \cdots \times 
	U(w_{k,N_k}) \stackrel{\sim}{\longrightarrow} U(v_k)
	\]
	which maps $(u_1,\dots,u_{N_k})$ to $u_1w_{k,1}u_2w_{k,2}\cdots u_{N_k}w_{k,N_k} v_k^{-1}$
	for $1 \leq k\leq r$.
	Combining isomorphisms~$\psi_k$ and~\eqref{eq_U_and_ZI}
	we have the following commutative diagram:
	\[
	\begin{tikzcd}
	Z_{\mathcal J}' \arrow[d, "\eta_{\mathcal J,\mathcal I}"'] & 
	U(w_{1,1}) \times \cdots \times U(w_{1,N_1})
	\times \cdots \times U(w_{r,1}) \times \cdots \times U(w_{{r,N_r}})
	\arrow[l, "\psi_{\mathcal J}"', "\sim"]
	\arrow[d, "\psi_1 \times \cdots \times \psi_r"]
	\arrow{d}[anchor=center, rotate=90, yshift=1ex]{\sim}
	\\
	Z_{\mathcal I}' & 
	U(v_1) \times \cdots \times U(v_r)
	\arrow[l, "\psi_{\mathcal I}"', "\sim"]
	\end{tikzcd}
	\]
	Hence the result follows. 
\end{proof}
\begin{example}\label{ex_fBS_and_BS}
	Let $G = \SL(4)$, and let $\mathcal I = (\{1,2\},\{3\})$. Then 
	$w_1 =s_1s_2s_1$, respectively $w_2 = s_3$, is a reduced decomposition of 
	the longest element of $W_{\{1,2\}}$, respectively $W_{\{3\}}$. 
	Then we have
	the birational morphism $\eta_{(1,2,1,3), \mathcal I} \colon Z_{(1,2,1,3)} \to Z_{\mathcal I}$.
	Together with the birational morphism $\eta$ described in 
	Example~\ref{example_biratioanl_morphism_fBS_and_Sch_var}-(2), 
	we can see that 
	three varieties $Z_{(1,2,1,3)}$, $Z_{\mathcal I}$, and $X(s_1s_2s_1s_3)$
	are birationally equivalent:
	\[
	Z_{(1,2,1,3)} \stackrel{}{\longrightarrow} 
	Z_{\mathcal I} \stackrel{}{\longrightarrow}
	X(s_1s_2s_1s_3).
	\]
	On the other hand, we have another reduced decomposition 
	$w_1' = s_2s_1s_2$ of the longest element of $W_{\{1,2\}}$. This also
	gives the birational morphism $\eta_{(2,1,2,3), \mathcal I} \colon Z_{(2,1,2,3)} 
	\to Z_{\mathcal I}$. Hence we have the following diagram whose
	maps are all birational morphisms:
	\[
	\reqnomode
	\begin{tikzcd}[row sep = 0.5ex]
	Z_{(1,2,1,3)} \arrow[rd, ""]\\
	& Z_{\mathcal I} \arrow[r, ""]& X(s_1s_2s_1s_3)\\
	Z_{(2,1,2,3)} \arrow[ru, ""']
	\end{tikzcd} 
	\]
\end{example}

\subsection{Line bundles over flag Bott--Samelson varieties}
\label{sec_line_bundles_of_fBS}
Let $\mathcal I$ be a sequence of subsets of $[n]$. 
In this subsection we study line bundles over a flag Bott--Samelson
variety $Z_{\mathcal I}$ and their pullbacks in Proposition~\ref{prop_line_bdle_over_fBS_and_BS}. 
For an integral weight $\lambda \in \mathbb{Z}{\varpi_1}+\cdots+\mathbb{Z} \varpi_n$,
we have the homomorphism $e^{\lambda} \colon H \to \Cstar$.
We  extend it to the homomorphism $e^{\lambda} \colon B \to \Cstar$
by composing with the homomorphism 
\begin{equation}\label{eq_def_of_Upsilon}
\Upsilon \colon B \to H 
\end{equation}
induced by the canonical projection of Lie algebras $\mathfrak{b} \to \mathfrak{h}$ as in~\cite[II.1.8]{Jan07}.
Suppose that $\lambda_1,\dots,\lambda_r$ are integral weights.
Define a representation 
$\C_{\lambda_1,\dots,\lambda_r}$ of $B^r = 
\underbrace{B\times \cdots \times B}_r$ on $\C$ as
\[
(b_1,\dots,b_r) \cdot v = e^{\lambda_1}(b_1)\cdots e^{\lambda_r}(b_r) v.
\]
From this we can build a line bundle over $Z_{\mathcal I}$ by setting
\begin{equation}\label{eq_def_of_L_fBS}
\mathcal L_{\mathcal I, \lambda_1,\dots,\lambda_r} = %
\mathbf P_{\mathcal I} \times_{B^r} \C_{-\lambda_1,\dots,-\lambda_r},
\end{equation}
where an action of $B^r$ is defined as
\begin{equation}\label{eq_action_of_L_fBS}
\begin{split}
&(p_1,\dots,p_r,w)  \cdot (b_1,\dots,b_r) \\
&\qquad= (\Theta((p_1,\dots,p_r), (b_1,\dots,b_r)),
e^{\lambda_1}(b_1) \cdots e^{\lambda_r}(b_r) w).
\end{split}
\end{equation}
For simplicity, we use the following notation:
\begin{equation}
\mathcal L_{\mathcal I, \lambda} \coloneqq 
\mathcal{L}_{\mathcal I,0,\dots,0,\lambda}. 
\end{equation}

\begin{remark}\label{rmk_Pic_and_H2}
	Recall from~\cite[Example~19.1.11(d)]{Ful13} that for a flag bundle $X$ over $Y$, the cycle map $cl_X \colon A_{k} (X) \to H_{2k} (X) $ is an isomorphism if and only if $cl_Y$ is an isomorphism. Moreover, the cycle map is isomorphic for an arbitrary flag manifold.
	Since a flag Bott--Samelson variety is an iterated bundle of flags $P_{I_k}/B$ over a point, 
	the cycle map  $cl_{Z_{\mathcal I}} \colon A_k(Z_{\mathcal I}) \to H_{2k}(Z_{\mathcal I})$ is an isomorphism. On the other hand, since flag Bott--Samelson varieties are smooth  (see Proposition~\ref{prop_properties_of_fBS}(1)), we obtain the following isomorphisms
	\begin{equation}\label{eq_Pic_and_H2}
	\begin{tikzcd}
	\Pic(Z_{\mathcal I}) \arrow[r, "\cong"]
	& A_{(\dim_{\C}Z_{\mathcal I})-1}(Z_{\mathcal I})  \arrow[r, "cl_{Z_{\mathcal I}}"', "\cong"]
	& H_{2(\dim_{\C}Z_{\mathcal I})-2}(Z_{\mathcal I}) \arrow[r, "\cong"]
	& H^2(Z_{\mathcal I}).
	 \end{tikzcd}
	\end{equation}
	Here, the first isomorphism comes from~\cite[Example~2.1.1]{Ful13} and the last isomorphism is obtained by the Poicar\'e duality.
	Indeed, $c_1 \colon \Pic(Z_{\mathcal I}) \to H^2(Z_{\mathcal I})$ is the isomorphism~\eqref{eq_Pic_and_H2}. 
	When   the Levi subgroup $L_{I_k}$ of $P_{I_k}$ has Lie type $A$ for all $k$, 
	we present the association~\eqref{eq_Pic_and_H2} using a certain generator of $H^2(Z_{\mathcal I})$ in~\eqref{eq_c1_xi_and_a}, and  we present the first Chern class of the line bundle $\mathcal L_{\mathcal I,\lambda_1,\dots,\lambda_r}$ in Proposition~\ref{prop_deg_of_line_bdle}.
\end{remark}

Specifically when a flag Bott--Samelson variety is a usual 
Bott--Samelson variety, we will choose the weights 
to be of  special form. 
We recall a description of the Picard group of $Z_{\mathbf i}$ from~\cite{LaTh04}. 
Suppose given an integer vector
$\mathbf a = (a_1,\dots,a_r) \in \mathbb{Z}^r$,
we define a sequence of weights $\lambda_1,\dots,\lambda_r$
associated to the word $\mathbf i =(i_1,\dots,i_r)$ and the vector $\mathbf a $
by setting 
\[
\lambda_1 \coloneqq a_1 \varpi_{i_1}, \dots, \lambda_r \coloneqq a_r \varpi_{{i}_{r}}.
\]
For such $\lambda_j$ we use the notation
\begin{equation}\label{eq_def_of_L_ia}
\mathcal L_{\mathbf i, \mathbf a} \coloneqq 
\mathcal L_{\mathbf i, \lambda_1,\dots,\lambda_r}. 
\end{equation}
Since a Bott--Samelson variety is an iterated sequence of projective bundles,
the Picard group of Bott--Samelson variety $Z_{\mathbf i}$ is a free abelian group of rank~$r$ by~\cite[Exercise II.7.9]{Ha77}.
Indeed, the association between $\mathbf a \in \mathbb{Z}^r$ and~$\mathcal L_{\mathbf i, \mathbf a}$ gives an isomorphism between $\mathbb{Z}^r$ and $\text{Pic}(Z_{\mathbf i})$ (see~\cite[\S 3.1]{LaTh04}).

Let ${\bf i} = (i_{k, l})_{1 \le k \le r, 1 \le l \le N_k} \in [n]^{N_1 + \cdots + N_r}$ be a sequence such that
$(i_{k,1},\dots,i_{k,N_k})$ is a reduced word for the longest element in $W_{I_k}$ for $1 \leq k \leq r$. 
Recall from Proposition~\ref{prop_fBS_and_BS} that we have
the birational morphism 
$\eta_{{\bf i}, \mathcal{I}} \colon Z_{\mathbf i}
\to Z_{\mathcal I}$. 
The following proposition describes the pullback bundle 
$\eta_{{\bf i}, \mathcal{I}}^{\ast} \mathcal L_{\mathcal I, \lambda_1,
\dots,\lambda_r}$ under the morphism $\eta_{{\bf i}, \mathcal{I}}$
in terms of an integer vector.
\begin{proposition}\label{prop_line_bdle_over_fBS_and_BS}
Let $\mathcal I$, $\mathbf i$, and $\lambda_1,\dots,\lambda_r$ be as above.
The pullback bundle $\eta_{{\bf i}, \mathcal{I}}^{\ast} \mathcal L_{\mathcal I, \lambda_1,
	\dots,\lambda_r}$ over the Bott--Samelson variety $Z_{\mathbf i}$
is isomorphic to the line bundle $\mathcal L_{\mathbf i, \mathbf a}$ for
the integer vector
$\mathbf a  = (\mathbf a_1(1),\dots, \mathbf a_1(N_1),\dots, \mathbf a_r(1),\dots,\mathbf a_r(N_r))
\in \mathbb{Z}^{N_1}\oplus \cdots \oplus \mathbb{Z}^{N_r}$ 
given by
\begin{equation*}\label{eq_pull_back_bdle_over_BS}
\mathbf a_k(l)
= \begin{cases}
\displaystyle\langle \lambda_k, \alpha_{s}^{\vee} \rangle 
+ \sum_{\substack{
	k < j \leq r; \\
	s \notin \{ i_{t,u} \mid k < t \leq j, 1 \leq u \leq N_t\}
}}
\langle \lambda_j, \alpha_s^{\vee} \rangle
& \text{ if } l = \max \{q \mid i_{k,q} =s \},\\
0 & \text{ otherwise}.
\end{cases}
\end{equation*}
\end{proposition}
\begin{example}\label{example_of_prop2.8}
Let $G = \SL(4)$, $\mathcal I = (\{1,2\}, \{3\})$ and $\mathbf 
i = (1,2,1,3)$. Consider the line bundle 
$\mathcal L_{\mathcal I, \lambda_1, \lambda_2}$. Then
the pullback line bundle $\eta_{\mathbf i, \mathcal I}^{\ast} 
\mathcal L_{\mathcal I, \lambda_1, \lambda_2}$ corresponds to
the integer vector
\begin{align*}
\mathbf a &= (\mathbf a_1 (1), \mathbf a_1 (2), \mathbf a_1 (3), \mathbf a_2 (1)) \\
&= (0, \langle \lambda_1, \alpha_2^{\vee} \rangle + \langle \lambda_2, \alpha_2^{\vee} \rangle,
\langle \lambda_1, \alpha_1^{\vee}\rangle + \langle \lambda_2, \alpha_1^{\vee} \rangle,
\langle \lambda_2, \alpha_{3}^{\vee} \rangle).
\end{align*}
\end{example}
\begin{remark}\label{rmk_pullback_is_not_very_ample}
It is known from~\cite[Theorem 3.1, Corollary 3.3]{LaTh04} that the line bundle $\mathcal L_{\mathbf i, \mathbf a}$ is very ample, respectively generated by global sections, if and only if $\mathbf a \in \mathbb{Z}^{|\mathbf i|}_{>0}$, respectively $\mathbf a \in \mathbb{Z}^{|\mathbf i|}_{\geq 0}$. Suppose that $\mathbf i$ is a sequence satisfying the condition in Proposition~\ref{prop_line_bdle_over_fBS_and_BS}. 
As we saw in Example~\ref{example_of_prop2.8}, we cannot ensure that the pullback line bundle
$\eta_{{\bf i}, \mathcal{I}}^{\ast} \mathcal L_{\mathcal I, \lambda_1,
	\dots,\lambda_r}$ is very ample even if 
the weights $\lambda_1,\dots,\lambda_r$ are regular dominant weights.
\end{remark}
\begin{proof}[Proof of Proposition~\ref{prop_line_bdle_over_fBS_and_BS}]
By the definition of pullback line bundles, we have
\[
\eta_{\bf i, \mathcal{I}}^{\ast} \mathcal{L}_{\mathcal I, \lambda_1,\dots,
\lambda_r} =
\{(p,q) \in Z_{\bf i} \times \mathcal{L}_{\mathcal I, \lambda_1,\dots,\lambda_r}
\mid \eta_{{\bf i}, \mathcal I}(p) = \pi_{\mathcal I, \lambda_1,\dots,
	\lambda_r} (q)\},
\]
where $\pi_{\mathcal I, \lambda_1,\dots,\lambda_r}\colon 
\mathcal{L}_{\mathcal I, \lambda_1,\dots,\lambda_r} \to Z_{\mathcal I}$.
In other words,
\begin{equation}\label{eq_pull_back_eta_iI_LI}
\begin{split}
\eta_{{\bf i}, \mathcal I}^{\ast} \mathcal{L}_{\mathcal I,\lambda_1,\dots,\lambda_r}
&= \bigg\{(\left[(p_{k,l})_{1\leq k \leq r, 1 \leq l \leq N_k}], [p_1,\dots,p_r,w] \right) ~\bigg|~ \\
& \qquad \qquad \left[\prod_{l=1}^{N_1} p_{1,l}, \dots, \prod_{l=1}^{N_{r}} p_{r,l} \right]
= [p_1,\dots,p_r] \text{ in } Z_{\mathcal I}
\bigg\}.
\end{split}
\end{equation}

Define the line bundle $\mathcal{L}_{\mathbf{i},\lambda_1,\dots,\lambda_r}$ on $Z_{\bf i}$ by
\begin{align*}
\mathcal{L}_{\mathbf{i}, \lambda_1,\dots,\lambda_r} 
&\coloneqq \mathcal{L}_{\mathbf{i}, \scriptsize{\underbrace{0,\dots,0,\lambda_1}_{N_1},
	\underbrace{0,\dots,0,\lambda_2}_{N_2}},\dots,
	\scriptsize{\underbrace{0,\dots,0,\lambda_r}_{N_r}}} \\
&= ({\bf P}_{\bf i} \times \C_{0,\dots,0,\lambda_1,0,\dots,0,\lambda_2,
\dots,0,\dots,0,\lambda_r})/B^{N_1+\dots+N_r}.
\end{align*}
\smallskip 
\noindent\textbf{Claim 1.} $\eta_{\mathbf i, \mathcal I}^{\ast} \mathcal{L}_{\mathcal I, \lambda_1,\dots,\lambda_r} \cong \mathcal{L}_{\mathbf i, \lambda_1,\dots,\lambda_r}$. \\
\smallskip 
\noindent Consider a well-defined morphism $f \colon \eta_{\mathbf{i}, \mathcal{I}}^{\ast} \mathcal{L}_{\mathcal I, \lambda_1,\dots,\lambda_r} \to \mathcal{L}_{\mathbf i, \lambda_1,\dots,\lambda_r}$ given by 
\begin{equation}\label{eq_isomor_prop2.8_1}
f([(p_{k,l})_{k,l}], [p_1,\dots,p_r,w]) \coloneqq 
[(p_{k,l})_{k,l}, Cw].
\end{equation}
Here, the value $C$ is defined as follows. Because of the description in~\eqref{eq_pull_back_eta_iI_LI}, for each element $([(p_{k,l})_{k,l}], [p_1,\dots,p_r,w])$ in the pullback bundle $\eta_{\mathbf i, \mathcal I}^{\ast} \mathcal L_{\mathcal I, \lambda_1,\dots,\lambda_r}$, there exist $b_1,\dots,b_r \in B$ such that 
\begin{equation}\label{eq_bi_in_prop_2.8}
p_1b_1 = \prod_{l=1}^{N_1} p_{1,l}, \quad b_1^{-1} p_2 b_2 = \prod_{l=1}^{N_2} p_{2,l}, \quad \dots, \quad 
b_{r-1}^{-1}p_r b_r = \prod_{l=1}^{N_r} p_{r,l}.
\end{equation}
Using these elements $b_1,\dots,b_r$, the value $C$ is defined by
\begin{equation}\label{eq_def_of_C}
C \coloneqq e^{\lambda_1}(b_1)
e^{\lambda_2}(b_2) \cdots e^{\lambda_r}(b_r).
\end{equation}

On the other hand, we have a well-defined morphism $g \colon  \mathcal{L}_{\mathbf i, \lambda_1,\dots,\lambda_r} \to \eta_{\mathbf{i}, \mathcal{I}}^{\ast} \mathcal{L}_{\mathcal I, \lambda_1,\dots,\lambda_r}$ defined by
\begin{equation}\label{eq_isomor_prop2.8_11}
g([(p_{k,l})_{k,l},w]) \coloneqq 
\left(
[(p_{k,l})_{k,l}], \left[ \prod_{l=1}^{N_1} p_{1,l}, \prod_{l=1}^{N_2} p_{2,l},\dots,\prod_{l=1}^{N_r} p_{r,l}, w \right] 
\right).
\end{equation}
We claim that both compositions $f \circ g$ and $g \circ f$ are identities. First we consider the composition $f \circ g$:
\[
\begin{split}
f \circ g ([(p_{k,l})_{k,l},w]) 
&= f \left(
[(p_{k,l})_{k,l}], \left[ \prod_{l=1}^{N_1} p_{1,l}, \prod_{l=1}^{N_2} p_{2,l},\dots,\prod_{l=1}^{N_r} p_{r,l}, w \right] 
\right) \quad (\text{by}~\eqref{eq_isomor_prop2.8_11}) \\
&= ([(p_{k,l})_{k,l},w]).
\end{split}
\]
Here, the last equality holds because all the elements $b_1,\dots,b_r$ satisfying the equations~\eqref{eq_bi_in_prop_2.8} are  the identity element, and so  $C = 1$. 
For the composition $ g \circ f$, we obtain
\[
\begin{split}
& g \circ f ([(p_{k,l})_{k,l}], [p_1,\dots,p_r,w]) = g ([(p_{k,l})_{k,l}, Cw]) \\
&\qquad = \left(
[(p_{k,l})_{k,l}], \left[ \prod_{l=1}^{N_1} p_{1,l}, \prod_{l=1}^{N_2} p_{2,l},\dots,\prod_{l=1}^{N_r} p_{r,l}, Cw \right] 
\right) \\
&\qquad = \left(
[(p_{k,l})_{k,l}], [ p_1b_1, b_1^{-1} p_2b_2,\dots, b_{r-1}^{-1} p_r b_r, Cw] 
\right) \quad (\text{by}~\eqref{eq_bi_in_prop_2.8}) \\
&\qquad = \left(
[(p_{k,l})_{k,l}], [p_1,p_2,\dots,p_r, e^{\lambda_1}(b_1^{-1}) \cdots e^{\lambda_r}(b_r^{-1}) Cw]
\right) \\
& \qquad = ([(p_{k,l})_{k,l}], [p_1,\dots,p_r,w]) \quad (\text{by}~\eqref{eq_def_of_C}).
\end{split} 
\]
Accordingly, $f$ is an isomorphism, and Claim~1 holds.

\smallskip
\noindent\textbf{Claim 2.} $\mathcal L_{\mathbf i, \lambda_1,\dots,\lambda_r} \cong \mathcal{L}_{\mathbf i, \mathbf a}$, where $\mathbf a$ is the integer vector given in the statement of the proposition.\\
\smallskip
To present a concrete isomorphism, we set
\begin{equation}\label{eq_def_of_kjs}
\begin{split}
k(j,s) &\coloneqq \max \left( \{k  \mid 1 \leq k \leq j,\ i_{k,l} = s  \text{ for some } 1 \leq l \leq N_k
\} \cup \{0\}\right), \\
m(j,s) &\coloneqq \max \{ q \mid i_{k(j,s),q} = s\}
\end{split}
\end{equation}
for $1 \leq j \leq r$ and $s \in [n]$.
We define certain products $\zeta(j,s)$ of $p_{k,l}$ using $k(j,s)$ as follows:
\[
\zeta(j,s) \coloneqq \left(\prod_{m = m(j,s) +1}^{N_{k(j,s)}} p_{k(j,s),m}\right)
\left(\prod_{k = k(j,s) +1}^j \prod_{l = 1}^{N_k} p_{k,l}\right).
\]
Here, we set $p_{0,l}$ is the identity element. 

We denote the integral weight $\lambda_k$ by $d_{k,1} \varpi_1 + \cdots + d_{k,n} \varpi_n$ using integers $d_{k,j}$ for $1 \leq k \leq r$. We consider the following morphism:
\begin{equation}\label{eq_isomor_prop2.8_2}
\begin{split}
f_2 \colon \mathcal{L}_{\mathbf{i}, \lambda_1,\dots,\lambda_r}
&\stackrel{}{\longrightarrow} \mathcal{L}_{\mathbf{i}, \mathbf{a}} \\
[(p_{k,l})_{k,l},w]
& \mapsto [(p_{k,l})_{k,l}, C'w],
\end{split}
\end{equation}
where the value $C'$ is defined to be
\begin{equation}\label{eq_def_of_C'}
C' \coloneqq \prod_{s=1}^n \prod_{j=1}^r  e^{d_{j,s} \varpi_s}(\zeta(j,s))^{-1}.
\end{equation}
We note that if $I \subset [n]$ and $s \notin I$, then 
the map $e^{\varpi_s} \colon B \to \Cstar$ is naturally extended to 
$e^{\varpi_s} \colon P_I \to \Cstar$ by setting 
$e^{\varpi_s}(\exp(\mathfrak{g}_{\alpha})) = \{1\}$ for all $\alpha \in \Delta^- \cap \Delta_I$. Hence $e^{d_{j,s} \varpi_s}(\zeta(j,s))$ is defined. 

If the map $f_2$ is well-defined, then we obtain Claim~2 because the inverse of $f_2$ is attained by multiplying $(C')^{-1}$.  Therefore, to prove Claim~2, it is enough to show that $f_2$ is well-defined. 
Suppose that 
\[
[(b_{k,l-1}^{-1} p_{k,l} b_{k,l})_{k,l}, e^{\lambda_1}(b_{1,N_1}) \cdots e^{\lambda_r}(b_{r,N_r}) w]
\]
is another representative of the element $[(p_{k,l})_{k,l}, w]$ in $\mathcal L_{\mathbf i, \lambda_1,\dots,\lambda_r}$. Here, we use the convention that $b_{k,0} = b_{k-1,N_{k-1}}$ and $b_{0,l}$ is the identity element. 
To show the well-definedness of $f_2$, we have to prove that the following equality holds:
\begin{equation}\label{eq_prop_2.8_well_definedness}
\begin{split}
&\left( \prod_{k=1}^r \prod_{l=1}^{N_k} e^{\mathbf{a}_{k}(l) \varpi_{i_{k,l}}}(b_{k,l}) \right)
\left( \prod_{s=1}^n \prod_{j=1}^r  e^{d_{j,s} \varpi_s} ( \zeta(j,s))^{-1} \right) \\
& \qquad = \left(\prod_{j=1}^r \prod_{s=1}^n e^{d_{j,s} \varpi_s} (\zeta(j,s)')^{-1} \right)
\left(\prod_{k=1}^r e^{\lambda_k}(b_{k,N_k})\right).
\end{split}
\end{equation}
Here, $\zeta(j,s)'$ is defined by
\[
\begin{split}
\zeta(j,s)' &= 
\left(\prod_{m=m(j,s)+1}^{N_{k(j,s)}} b_{k(j,s),m-1}^{-1} p_{k(j,s),m} b_{k(j,s),m} \right)
\left(\prod_{k=k(j,s)+1}^{j} \prod_{l=1}^{N_k} b_{k,l-1}^{-1} p_{k,l} b_{k,l} \right)  \\
&= b_{k(j,s),m(j,s)}^{-1} \zeta(j,s) b_{j, N_j}.
\end{split}
\]
Furthermore, since the weight $\lambda_j$ is the sum of $d_{j,s}$, we have that
\[
\prod_{j=1}^r \prod_{s=1}^n e^{d_{j,s} \varpi_s} (b_{j,N_j}) 
= \prod_{j=1}^r e^{\lambda_j} (b_{j,N_j}).
\]
Therefore, the right hand side of the equation~\eqref{eq_prop_2.8_well_definedness} becomes
\[
\left(\prod_{j=1}^r \prod_{s=1}^n e^{d_{j,s} \varpi_s}(b_{k(j,s), m(j,s)}) \right)
\left( \prod_{s=1}^n  \prod_{j=1}^r e^{d_{j,s} \varpi_s} (\zeta(j,s))^{-1} \right).
\]
This implies that to show the equality~\eqref{eq_prop_2.8_well_definedness}, it is enough to show that
\begin{equation}\label{eq_prop_2.8_well_definedness2}
\prod_{k=1}^r \prod_{l=1}^{N_k} e^{\mathbf{a}_k(l) \varpi_{i_{k,l}}}(b_{k,l}) = 
\prod_{s=1}^n \prod_{j=1}^r e^{d_{j,s} \varpi_s}(b_{k(j,s), m(j,s)}).
\end{equation}
The left hand side of the equation~\eqref{eq_prop_2.8_well_definedness2} is written by
\[
\prod_{k=1}^r \prod_{l=1}^{N_k} e^{\mathbf{a}_k(l) \varpi_{i_{k,l}}}(b_{k,l})
= \prod_{s=1}^n ~~~ \prod_{\substack{ 1 \leq k \leq r, 1 \leq l \leq N_k, \\ i_{k,l} = s}} e^{\mathbf{a}_{k}(l) \varpi_s} (b_{k,l}).
\]
Using this observation, we verify the equality~\eqref{eq_prop_2.8_well_definedness2} by showing 
\begin{equation}\label{eq_prop_2.8_well_definedness3}
\prod_{\substack{ 1 \leq k \leq r, 1 \leq l \leq N_k, \\ i_{k,l} = s}} e^{\mathbf{a}_{k}(l) \varpi_s} (b_{k,l})
= \prod_{j=1}^r e^{d_{j,s} \varpi_s}(b_{k(j,s), m(j,s)})
\end{equation}
for all $s  \in [n]$.
Let $s$ be an arbitrary index in $[n]$. 
If $s$ does not appear in $(i_{k,l})_{k,l}$, then $k(j,s) = 0$ for all $j$, and so the equality~\eqref{eq_prop_2.8_well_definedness3} holds. Otherwise,
let $1 \leq j_1 < \cdots < j_x \leq r$ be the indices such that $s \in \{ i_{j,1},\dots,i_{j,N_{j}}\}$ if and only if $j \in \{j_1,\dots,j_x\}$. By the definition of the number $k(j,s)$, we have that
\[
\begin{split}
0 &= k(1,s) = \cdots = k(j_1-1,s),\\
j_u &= k(j_u,s) = \cdots = k(j_{u+1}-1,s) \quad \text{ for }1 \leq u \leq x.
\end{split}
\]
Here, we set $j_{x+1} =r+1$. 
Therefore, we have that
\[
\prod_{j=1}^r e^{d_{j,s} \varpi_s}(b_{k(j,s),m(j,s)})  
= \prod_{u=1}^x e^{(d_{j_u,s} + \dots + d_{j_{u+1}-1,s}) \varpi_s}(b_{j_u, m(j_u,s)}).
\]
On the other hand, by the definition of the integer vector $\mathbf a$, if ${i_{k,l}} = s$, then $k \in \{j_1,\dots,j_x\}$ and we have that
\[
\mathbf{a}_{k}(l) = 
\begin{cases}
d_{j_u,s} + d_{j_u+1,s} + \cdots + d_{j_{u+1}-1,s} & \text{ if $k = j_u$ and $l = m(j_u,s)$}, \\
0 & \text{ otherwise}. 
\end{cases}
\]
Accordingly, we obtain the equality~\eqref{eq_prop_2.8_well_definedness3} for all $s \in [n]$, and we get the equality~\eqref{eq_prop_2.8_well_definedness2}. Therefore, the morphism $f_2$ is a well-defined isomorphism. This proves Claim~2. 
By combining Claim~1 and Claim~2, the result follows.
\end{proof}
We put an example for explaining notations $C, k(j,s), C'$ in the proof of Proposition~\ref{prop_line_bdle_over_fBS_and_BS} for the reader's convenience.
\begin{example}
Let $G = \SL(4)$. 
Suppose that  $\mathcal I$ and $\bf{i}$ are given as in Example~\ref{example_of_prop2.8}.
Then for an element $([p_{1,1},p_{1,2},p_{1,3},p_{2,1}], [p_1,p_2,w])$ in 
$\eta_{{\bf i}, \mathcal I}^{\ast} \mathcal{L}_{\mathcal I,\lambda_1,\lambda_2}$ 
the value $C$ in~\eqref{eq_def_of_C} is given by
\[
C = e^{\lambda_1}\left(p_1^{-1} p_{1,1} p_{1,2} p_{1,3} \right)
e^{\lambda_2}\left(p_2^{-1} p_1^{-1} p_{1,1} p_{1,2} p_{1,3} p_{2,1} \right). 
\]
Moreover the indices $k(j,s)$ in~\eqref{eq_def_of_kjs} are computed by 
\[
k(1,1) = 1, k(1,2) = 1, k(1,3) = 0, k(2,1) = 1, k(2,2) = 1, k(2,3)=2.
\]
Hence the value $C'$ in~\eqref{eq_def_of_C'} is
\[
C' = e^{d_{1,2}\varpi_2}(p_{1,3})^{-1}
e^{d_{1,3} \varpi_3}(p_{1,1} p_{1,2} p_{1,3})^{-1}
e^{d_{2,1} \varpi_1}(p_{2,1})^{-1} 
e^{d_{2,2} \varpi_2}(p_{1,3} p_{2,1})^{-1},
\]
where $\lambda_k = d_{k,1} \varpi_1 + d_{k,2} \varpi_2 + d_{k,3} \varpi_3$ for
$k = 1,2$.
\end{example}

\subsection{Newton--Okounkov bodies of flag Bott--Samelson varieties}
\label{subsect_NOBY}
In this section we study the Newton--Okounkov bodies of flag
Bott--Samelson varieties in Theorem~\ref{thm_NOBY_of_fBS}.
First we recall the definition and background of Newton--Okounkov bodies. We refer the reader to~\cite{Fuj17, HaKa15, Kav15, KaKh12} for more details.
Let $R$ be a $\C$-algebra without nonzero zero-divisors, and fix a 
total order $<$ on $\mathbb{Z}^r$, $r \geq 1$, respecting
the addition. 
\begin{definition}
	A map $v \colon R \setminus \{0\} \to \mathbb{Z}^r$ is called 
	a \defi{valuation} on $R$ if the following conditions hold.
	For every $f, g \in R \setminus \{0\}$ and $c \in \C \setminus \{0\}$,
	\begin{enumerate}
		\item $v (f \cdot g) = v(f) + v(g)$,
		\item $v(c f) = v(f)$, and
		\item $v(f + g) \geq \min \{ v(f), v(g)\}$ unless
		$f + g = 0$. 
	\end{enumerate}
\end{definition}
Moreover we say the valuation $v$ has \defi{one-dimensional leaves}
if it satisfies that if $v(f) = v(g)$ then there exists 
a nonzero constant $\lambda \in \C$ such that 
$v(g - \lambda f) > v(g)$ or $g - \lambda f =0$. 

Let $X$ be a projective variety of dimension $r$ over $\C$
equipped with a line bundle $\mathcal L$ which is generated by global sections.
Fix a valuation $v$ 
which has one-dimensional leaves
on the function field $\C(X)$ of $X$. Using the valuation
$v$ one can associate a semigroup $S \subset \mathbb{N} 
\times \mathbb{Z}^r$ as follows. Fix a nonzero element $\tau\in 
H^0(X,{\mathcal L})$. We use $\tau$ to identify $H^0(X,\mathcal L)$
with a finite-dimensional subspace of $\C(X)$ by mapping
\[
H^0(X,\mathcal L) \to \C(X), \quad \sigma \mapsto \sigma/\tau.
\]
Similarly we have the map
\[
H^0(X, \mathcal L^{\otimes k}) \to \C(X),\quad
\sigma \mapsto \sigma/\tau^k. 
\]
Using these identifications we define the semigroup:
\[
S = S(v, \tau)
= \bigcup_{k > 0} \left\{(k, v(\sigma/ \tau^k))
\mid \sigma \in 
H^0(X, \mathcal L^{\otimes k}) \setminus \{0\} \right\} \subset \mathbb{N}
\times \mathbb{Z}^r,
\]
and denote by $C = C(v,\tau) \subset \mathbb{R}_{\geq 0} \times \mathbb{R}^r$ the smallest real closed cone containing $S(v,\tau)$.
Now we have the definition of Newton--Okounkov body:
\begin{definition}
	The \defi{Newton--Okounkov body  associated to 
		$(X,\mathcal L ,v,\tau)$}  is defined to be
	\[
\{ \mathbf x \in \mathbb{R}^r \mid (1,\mathbf x) \in C(v,\tau)\}.
	\]
	We denote the Newton--Okounkov body by 	$\Delta(X,\mathcal L ,v,\tau)$.
\end{definition}
If we take another section $\tau' \in H^0(X,\mathcal L) \setminus \{0\}$ then $\Delta(X, \mathcal L, v, \tau') = \Delta(X, \mathcal L, v, \tau) + v(\tau/\tau')$. Hence the Newton--Okounkov body $\Delta(X, \mathcal L, v, \tau)$ does not fundamentally depend on the choice of the nonzero section $\tau \in H^0(X,\mathcal L) \setminus \{0\}$.  
Accordingly, we sometimes denote the Newton--Okounkov body by $\Delta(X, \mathcal L, v)$. 

\begin{remark}\label{rmk_dimension_of_NOBY}
If we choose a very ample line bundle $\mathcal L$ in the construction, then it is known in~\cite[Theorem 3.9]{HaKa15} that the Newton--Okounkov body has maximal dimension, i.e. it has real dimension $r$. Since we do not necessarily assume that the line bundle $\mathcal L$ is very ample in this paper, the real dimension of the Newton--Okounkov body may be less than $r$.
\end{remark}

There are many possible valuations with one-dimensional leaves. We recall one of them introduced in~\cite{Kav15}. One can construct a valuation on the function field $\C(X)$ using a regular system of parameters $u_1,\dots,u_r$ in a neighborhood of a smooth point $p$ on $X$. Fix a total ordering on $\mathbb{Z}^r$ respecting the addition. Let $f$ be a polynomial in $u_1,\dots,u_r$. Suppose that $c_k u_1^{k_1} \cdots u_r^{k_r}$ is the term in $f$ with the largest exponent $k = (k_1,\dots,k_r)$. Then
\[
v(f) \coloneqq (-k_1,\dots,-k_r)
\]
defines a valuation on $\C(X)$, called the \defi{highest term valuation} with respect to the parameters $u_1,\dots,u_r$.
\begin{example}\label{example_valuation_vi}
Let $X = Z_{\mathbf i}$ be the Bott--Samelson variety determined by 
a word $\mathbf i = (i_1,\dots,i_r)$. 
Let $f_i$ be a nonzero element in $\mathfrak{g}_{-\alpha_i}$.
Then the following map $\Phi_{\mathbf i} \colon \mathbb{C}^r \to Z_{\mathbf i}$ defines a coordinate system as in~\cite[\S 2.3]{Fuj18} and~\cite[\S 2.2]{Kav15}:
\[
\Phi_{\mathbf i} \colon (t_1,\dots,t_r) \mapsto 
(\exp(t_1 f_{i_1}), \dots, \exp(t_r f_{i_r}))
\quad
\mod B^r
\]
We denote the highest term valuation with respect to the lexicographic order on $\mathbb{Z}^r$ by $v_{\mathbf i}^{\rm high}$.
\end{example}
There are some results on computing the Newton--Okounkov bodies using the valuation $v_{\mathbf i}^{\rm high}$. 
We recall a result of Kaveh~\cite{Kav15}:
\begin{example}
Let $X = G/B$ be the full flag variety, and
let $\mathcal L$ be the line bundle over $X$ given by a dominant weight $\lambda$. Suppose that 
$\mathbf i = (i_1,\dots,i_m)$ is a reduced word for the longest element in the Weyl group $W$ of $G$. 
Then the Bott--Samelson variety $Z_{\mathbf i}$ and the full flag variety $G/B$ 
are birational by~Proposition~\ref{prop_fBS_and_flag}. 
Hence their function fields are isomorphic, i.e., $\mathbb{C}(Z_{\mathbf i}) \cong \mathbb{C}(G/B)$. Using the valuation $v_{\mathbf i}^{\rm high}$  in~Example~\ref{example_valuation_vi}, 
Kaveh~\cite[Corollary 4.2]{Kav15} proves that
the Newton--Okounkov body $\Delta(G/B, \mathcal L,
v_{\mathbf i}^{\rm high})$  can be identified with the string polytope.
\end{example}
The following lemma directly comes from 
the definition of Newton--Okounkov bodies.
\begin{lemma}\label{lemma_NO_body_same}
	Let $f \colon X \to Y$ be a birational morphism between 
	varieties of dimension $r$, and let $\mathcal L$ be a 
	 line bundle on $Y$ generated by global sections.
	Suppose that the canonical morphism $H^0(Y, \mathcal{L}^{\otimes k}) \to H^0(X, f^{\ast} \mathcal{L}^{\otimes k})$ is an isomorphism for every $k > 0$. Then 
	their Newton--Okounkov bodies coincide, i.e.,
	\[
	\Delta(X, f^{\ast} \mathcal L, v, f^{\ast} \tau) 
	= \Delta(Y, \mathcal L, v, \tau)
	\]
	for any valuation $v \colon \C(X)\setminus \{0\} \to \Z^r$ and $\tau \in H^0 (Y, \mathcal{L}) \setminus \{0\}$. Here $v$ is regarded also as a valuation on $\C(Y)$ under the isomorphism $\C(Y) \cong \C(X)$.
\end{lemma}

Now we define left actions of $P_{I_1}$ on $Z_{\mathcal{I}}$ and  $\mathcal{L}_{\mathcal{I}, \lambda_1, \ldots, \lambda_r}$ by
\begin{equation}\label{eq_action_of_PI}
\begin{split}
&p \cdot [p_1, \ldots, p_r] \coloneqq [p p_1, p_2, \ldots, p_r],\\
&p \cdot [p_1, \ldots, p_r, v] \coloneqq [p p_1, p_2, \ldots, p_r, v]
\end{split}
\end{equation}
for $p, p_1 \in P_{I_1}$, $p_2 \in P_{I_2}, \ldots, p_r \in P_{I_r}$, and $v \in \mathbb{C}$. Since the projection $\mathcal{L}_{\mathcal{I}, \lambda_1, \ldots, \lambda_r} \twoheadrightarrow Z_{\mathcal{I}}$ is compatible with these actions, it follows that the space $H^0(Z_{\mathcal{I}}, \mathcal{L}_{\mathcal{I}, \lambda_1, \ldots, \lambda_r})$ of global sections has the natural $P_{I_1}$-module structure.

\begin{theorem}\label{thm_isom_between_holom_sections}
	Let $\mathcal{I} = (I_1,\dots,I_r)$ be a sequence of subsets of $[n]$, 
	and let ${\bf i} = (i_{k, l})_{1 \le k \le r, 1 \le l \le N_k} \in [n]^{N_1 + \cdots + N_r}$ be a sequence such that $(i_{k, 1}, \ldots, i_{k, N_k})$ is a reduced word for the longest element in $W_{I_k}$ for $1 \le k \le r$. 
	Let $\eta_{\mathbf i,\mathcal I} \colon
	Z_{\mathbf i} \to Z_{\mathcal I}$ be the birational morphism  
	in Proposition~\ref{prop_fBS_and_BS}.
	Then for integral weights $\lambda_k \coloneqq d_{k,1} \varpi_1 + \cdots
	+ d_{k,n} \varpi_n$ for $1\leq k \leq r$, and the corresponding integer vector $\mathbf a$ given in Proposition~\ref{prop_line_bdle_over_fBS_and_BS},
	\begin{enumerate}
		\item the canonical morphism $H^0(Z_{\mathcal I}, \mathcal{L}_{\mathcal I, \lambda_1,\dots,\lambda_r}) \to H^0(Z_{\bf i}, \mathcal{L}_{\mathbf i, \mathbf a})$ is an isomorphism.
%
		\item The isomorphism in (1) induces the $B$-module isomorphism
	\[
H^0(Z_{\mathcal I},\mathcal{L}_{\mathcal I,\lambda_1,\dots,\lambda_r})
\cong H^0(Z_{\mathbf i}, \mathcal{L}_{\bf i, \bf a}) \otimes \C_{-\mu},
\]
where 
$\mu$ is the weight defined by
\[
\mu = \sum_{j=1}^r
\sum_{s \in [n] \setminus \{i_{k,l} \mid 1 \leq k \leq j, 
	1 \le l \le N_k\}} d_{j,s} \varpi_s.
\]
	\end{enumerate}
\end{theorem}
To prove the theorem, we recall the following lemma.
\begin{lemma}[{\cite[II.14.5.(a)]{Jan07}}]
	\label{lemma_push_forward_str_sheaf}
	Let $\varphi \colon Y \to X$ be a dominant and projective
	morphism of noetherian and integral schemes such that $\varphi$
	induces an isomorphism $\C(X) \stackrel{\sim}{\longrightarrow}
	\C(Y)$ of function fields. If $X$ is normal, then
	$\varphi_{\ast} \mathcal{O}_Y = \mathcal{O}_X$.
\end{lemma}
\begin{proof}[Proof of Theorem~\ref{thm_isom_between_holom_sections}]
	\begin{enumerate}
	\item
	Because of Propositions~\ref{prop_properties_of_fBS} and~\ref{prop_fBS_and_BS}, the morphism 
	$\eta=\eta_{\mathbf i,\mathcal I} \colon Z_{\mathbf i}\to Z_{\mathcal I}$
	satisfies all the conditions in Lemma~\ref{lemma_push_forward_str_sheaf}. Hence we have that
	\begin{equation}\label{eq_structure_sheaves}
	\eta_{\ast} \mathcal{O}_{Z_{\mathbf i}}
	= \mathcal{O}_{ Z_{\mathcal I}}.
	\end{equation}
	Then we have the following:
	\begin{align*}
	\eta_{\ast}(\eta^{\ast} \mathcal{L}_{\mathcal I, \lambda_1,
		\dots,\lambda_r})
	&= 
	\eta_{\ast}(\mathcal{O}_{Z_{\mathbf i}}
	\otimes_{
		\mathcal{O}_{Z_{\mathbf i}}}
	\eta^{\ast} \mathcal{L}_{\mathcal I, \lambda_1,
		\dots,\lambda_r}) \\
	&\cong \eta_{\ast}\mathcal{O}_{Z_{\mathbf i}}
	\otimes_{\mathcal{O}_{ Z_{\mathcal I}}}
	\mathcal{L}_{\mathcal I, \lambda_1,	\dots,\lambda_r}
	\quad  
	\text{(by~\cite[Exercise II.5.1.(d)]{Ha77})} \\
	&= \mathcal{O}_{Z_{\mathcal I}}
	\otimes_{\mathcal{O}_{Z_{\mathcal I}}}
	\mathcal{L}_{\mathcal I, \lambda_1,	\dots,\lambda_r} 
	\quad  \text{(by~\eqref{eq_structure_sheaves})} \\
	&= \mathcal{L}_{\mathcal I, \lambda_1,\dots,\lambda_r}.
	\end{align*}
	Taking global sections we have an isomorphism between $H^0(Z_{\mathcal I}, \mathcal L_{\mathcal I, \lambda_1,\dots,\lambda_r}) $ and 
	$H^0(Z_{\mathbf i}, \eta_{{\bf i}, \mathcal I}^{\ast} \mathcal{L}_{\mathcal{I},\lambda_1,\dots,\lambda_r})$ as $\C$-vector spaces. And the later one is isomorphic to $H^0(Z_{\mathbf i}, \mathcal{L}_{\bf i, \bf a})$ as $\C$-vector spaces by Proposition~\ref{prop_line_bdle_over_fBS_and_BS}.
	\item
	Note that there is a bijective correspondence between the set $H^0(Z_{\mathcal I}, \mathcal{L}_{\mathcal I, \lambda_1,\dots,\lambda_r})$ of holomorphic sections and the set of morphisms $f \colon \mathbf{P}_{\mathcal I} \to \C$ satisfying 
	\begin{equation}\label{eq_section_f}
	f((p_1,\dots,p_r)\cdot (b_1,\dots,b_r)) = e^{\lambda_1}(b_1) \cdots e^{\lambda_r}(b_r) f(p_1,\dots,p_r)
	\end{equation} 
	for $(p_1,\dots,p_r) \in {\bf P}_{\mathcal I}$ and $(b_1,\dots,b_r) \in B^r$.
	Indeed, a morphism~$f$ defines a section $[p_1,\dots,p_r] \mapsto [p_1,\dots,p_r, f(p_1,\dots,p_r)]$. 
	Using $C$ and $C'$ defined in the proof of Proposition~\ref{prop_line_bdle_over_fBS_and_BS} (see~\eqref{eq_def_of_C} and~\eqref{eq_def_of_C'}), for a morphism $f$ satisfying~\eqref{eq_section_f}, we associate a morphism $\tilde{f} \colon \mathbf{P}_{\mathbf i} \to \C$ 
	\[
	\tilde{f}((p_{k,l})_{k,l}) = C' C f\left(\prod_{l=1}^{N_1} p_{1,l}, \dots, \prod_{l=1}^{N_r} p_{r,l} \right)
	\]
	which also gives a section in $H^0(Z_{\mathbf i}, \mathcal{L}_{\mathbf i, \mathbf a})$. Actually, this association is the isomorphism in (1).
	
	On the other hand, the left action of $P_{I_1}$ on $Z_{\mathcal I}$ and that of $P_{i_{1,1}}$ on $Z_{\mathbf i}$ given in~\eqref{eq_action_of_PI} define actions of $B$ on the sets of holomorphic sections. For $b \in B$, $f \colon P_{\mathcal I} \to \C$, and $\tilde{f} \colon P_{\mathbf i} \to \C$, we set
	\[
	\begin{split}
	(b \cdot f)(p_1,\dots,p_r) &\coloneqq f(b^{-1}p_1,p_2,\dots,p_r), \\
	(b \cdot \tilde{f})((p_{k,l})_{k,l})&\coloneqq \tilde{f}(b^{-1} p_{1,1}, p_{1,2},\dots,p_{1,N_1},\dots,p_{r,N_r}).
	\end{split}
	\]	
	Recall from~\eqref{eq_def_of_C'} that $C'$  is the product of $e^{d_{j,s} \varpi_s} (\zeta(j,s))^{-1}$. For each $s \in [n]$ and $j \in [r]$, by~\eqref{eq_def_of_kjs}, the following three conditions are equivalent:
	\begin{itemize}
		\item $p_{1,1}$ is involved in  $\zeta(j,s)$;
	 \item $k(j,s) = 0$;
	\item $s \in [n] \setminus \{ i_{1,1},\dots,i_{1,N_1},\dots,i_{j,N_j}\}$. 
	\end{itemize}
	Using this observation, we obtain that
	\[
	\begin{split}
	&(b \cdot \tilde{f})((p_{k,l})_{k,l})\\
	&\qquad=
	\tilde{f}(b^{-1}p_{1,1},p_{1,2},\dots,p_{1,N_1},\dots,p_{r,N_r}) \\
	&\qquad = \left(\prod_{j=1}^r \prod_{s \in [n] \setminus \{ i_{k,l} \mid 1 \leq k \leq j, 1 \leq l \leq N_k\}} e^{d_{j,s} \varpi_s} (b) \right) C' C f \left(b^{-1} \prod_{l=1}^{N_1} p_{1,l}, \dots, \prod_{l=1}^{N_r} p_{r,l} \right) \\
	&\qquad = e^{\mu}(b)  \left(  \widetilde{ b \cdot f}((p_{k,l})_{k,l}) \right)\!,
	\end{split}
	\]
	where $C$ and $C'$ are values determined by $(p_{k,l})_{k,l}$, and $\mu$ is the weight given in the statement.
	This proves the desired equality $\widetilde{b \cdot f} = e^{-\mu}(b) (b \cdot \tilde{f})$.
	\qedhere
	\end{enumerate}
\end{proof}

As a direct consequence of Theorem~\ref{thm_isom_between_holom_sections}(1)
and Lemma~\ref{lemma_NO_body_same} we have the following theorem.
\begin{theorem}\label{thm_NOBY_of_fBS}
	Let $\mathcal I = (I_1,\dots, I_r)$ be a sequence of subsets
	of $[n]$, and
	let ${\bf i} = (i_{k, l})_{1 \le k \le r, 1 \le l \le N_k} \in [n]^{N_1 + \cdots + N_r}$ be a sequence such that
	$(i_{k,1},\dots,i_{k,N_k})$ is a reduced word for the longest element in $W_{I_k}$ for $1\leq k \leq r$. 
	Let $\eta_{\mathbf i,\mathcal I} \colon
	Z_{\mathbf i} \to Z_{\mathcal I}$ be the birational morphism  
	defined in Proposition~\ref{prop_fBS_and_BS}.
	Then for integral dominant weights $\lambda_k$, $1 \leq k \leq r$, a valuation $v$ on $\C(Z_{\mathcal I})$, and a nonzero section $\tau \in H^0(Z_{\mathcal I}, \mathcal L_{\mathcal I,
		\lambda_1,\dots, \lambda_r})$, we have the equality
	\[
	\Delta(Z_{\mathcal I}, \mathcal L_{\mathcal I,
		\lambda_1,\dots, \lambda_r}, v, \tau)
	= \Delta(Z_{\mathbf i}, \eta_{\mathbf i,\mathcal I}^{\ast} \mathcal L_{\mathcal I,
		\lambda_1,\dots, \lambda_r}, v, \eta_{\mathbf i,\mathcal I}^{\ast} \tau).
	\]
\end{theorem}

\begin{remark}\label{remark_dimension_of_NOBY}
	Even if the line bundle $\mathcal L = \mathcal L_{\mathcal I, \lambda_1,\dots,\lambda_r}$ is very ample, the pullback bundle $\eta_{{\bf i}, \mathcal I}^{\ast} \mathcal L$ is not necessarily very ample when $Z_{\mathcal I}$ is not a Bott--Samelson variety (see Remark~\ref{rmk_pullback_is_not_very_ample}). Therefore the real dimension of Newton--Okounkov body $\Delta(Z_{\mathbf i}, \eta_{\mathbf i, \mathcal I}^{\ast} \mathcal L, v)$ can possibly be smaller than the complex dimension of $Z_{\mathbf i}$ as is mentioned in Remark~\ref{rmk_dimension_of_NOBY}. However, by Theorem~\ref{thm_NOBY_of_fBS}, when $\mathcal L$ is very ample, we can see that
	\[
	\dim_{\mathbb R} \Delta(Z_{\mathbf i}, \eta_{{\bf i}, \mathcal I}^{\ast} \mathcal L,v) 
	= \dim_{\mathbb R} \Delta(Z_{\mathcal I}, \mathcal L, v) = \dim_{\mathbb C} Z_{\mathcal I} = \dim_{\mathbb C} Z_{\mathbf i}
	\]
	for any valuation $v$ which has one-dimensional leaves. 
\end{remark}
By~Theorem~\ref{thm_NOBY_of_fBS} and~\cite[Corollary~5.4]{Fuj18},
we have the following corollary.
\begin{corollary}
Suppose that the line bundle $\mathcal L_{\mathcal I, \lambda_1,\dots, \lambda_r}$ constructed by weights $\lambda_1,\dots,\lambda_r$ is very ample. Then,
the Newton--Okounkov body $\Delta(Z_{\mathcal I}, 
\mathcal L_{\mathcal I, \lambda_1,\dots, \lambda_r}, v_{\mathbf i}^{\rm high})$
is a rational convex polytope of real dimension equal to the complex
dimension of $Z_{\mathcal I}$.
\end{corollary}

\section{Applications to representation theory}
\label{sec_appli_to_RT}
In this section, we give applications of Newton--Okounkov bodies of flag Bott--Samelson varieties to representation theory, using the theory of generalized string polytopes introduced in~\cite{Fuj18}. We restrict ourselves to a specific class of flag Bott--Samelson varieties $Z_{\mathcal{I}}$, that is, to the case of a sequence $\mathcal{I} = (I_1, \ldots, I_r)$ of subsets of $[n]$ such that $I_1 = [n]$. In this case, we have $P_{I_1} = P_{[n]} = G$. Hence the space $H^0(Z_{\mathcal{I}}, \mathcal{L}_{\mathcal{I}, \lambda_1, \ldots, \lambda_r})$ of global sections has a natural $G$-module structure. Let $\chi(H) \coloneqq \mathbb{Z} \varpi_1 + \cdots + \mathbb{Z} \varpi_n$ be the character lattice, and let $\chi_+ (H) \coloneqq \mathbb{Z}_{\ge 0} \varpi_1 + \cdots + \mathbb{Z}_{\ge 0} \varpi_n$ be the set of integral dominant weights. Fix nonzero elements $e_i \in \mathfrak{g}_{\alpha_i}$, $f_i \in \mathfrak{g}_{-\alpha_i}$ for $i \in [n]$. For $\lambda \in \chi_+ (H)$, let $V(\lambda)$ denote the irreducible highest weight $G$-module over $\mathbb{C}$ with the highest weight $\lambda$, and let $v_{\lambda} \in V(\lambda)$ be a highest weight vector. Recall that every finite-dimensional irreducible $G$-module is isomorphic to $V(\lambda)$ for some $\lambda \in \chi_+ (H)$, see~\cite[\S 31.3]{Hum75}, and that every finite-dimensional $G$-module is completely reducible, that is, isomorphic to a direct sum of irreducible $G$-modules (see~\cite[\S 14.3]{Hum75}). For $\lambda_1, \ldots, \lambda_r \in \chi_+ (H)$, we denote by $\tau_{\mathcal{I}, \lambda_1, \ldots, \lambda_r} \in H^0(Z_{\mathcal I}, \mathcal L_{\mathcal I, \lambda_1,\dots,\lambda_r})$ the section corresponding to $\tau_{\bf i, \bf a} \in H^0(Z_{\bf i}, \mathcal L_{{\bf i}, {\bf a}})$ under the isomorphism in Theorem \ref{thm_isom_between_holom_sections} (1), where $\tau_{\bf i, \bf a}$ is the section defined in~\cite[\S 2.3]{Fuj18}. Let $\pi_{\ge 2} \colon \mathbb{R}^{N_1 + \cdots + N_r} \twoheadrightarrow \mathbb{R}^{N_2 + \cdots + N_r}$ be the canonical projection given by $\pi_{\ge 2}((x_{k, l})_{1 \le k \le r, 1 \le l \le N_k}) \coloneqq (x_{k, l})_{2 \le k \le r, 1 \le l \le N_k}$, and set \[\widehat{\Delta}_{{\bf i}, \lambda_1, \ldots, \lambda_r} \coloneqq \pi_{\ge 2}(-\Delta(Z_{\mathcal{I}}, \mathcal{L}_{\mathcal{I}, \lambda_1, \ldots, \lambda_r}, v_{\bf i} ^{\rm high}, \tau_{\mathcal{I}, \lambda_1, \ldots, \lambda_r})).\] 
Since $\Delta(Z_{\mathcal{I}}, \mathcal{L}_{\mathcal{I}, \lambda_1, \ldots, \lambda_r}, v_{\bf i} ^{\rm high}, \tau_{\mathcal{I}, \lambda_1, \ldots, \lambda_r})$ is a rational convex polytope, the image $\widehat{\Delta}_{{\bf i}, \lambda_1, \ldots, \lambda_r}$ is also a rational convex polytope. The following is the main result in this section. 

\begin{theorem}\label{t:representation}
	Let $\mathcal{I} = (I_1, \ldots, I_r)$ be a sequence of subsets of $[n]$ such that $I_1 = [n]$, and fix ${\bf i} = (i_{1, 1}, \ldots, i_{1, N_1}, \ldots, i_{r, 1}, \ldots, i_{r, N_r}) \in [n]^{N_1 + \cdots + N_r}$ such that $(i_{k, 1}, \ldots, i_{k, N_k})$ is a reduced word for the longest element in $W_{I_k}$ for $1 \le k \le r$. For $\lambda_1, \ldots, \lambda_r \in \chi_+ (H)$, write \[H^0(Z_{\mathcal{I}}, \mathcal{L}_{\mathcal{I}, \lambda_1, \ldots, \lambda_r})^\ast \simeq \bigoplus_{\nu \in \chi_+ (H)} V(\nu)^{\oplus c_{\mathcal{I}, \lambda_1, \ldots, \lambda_r} ^\nu}\] as a $G$-module. Then, the multiplicity $c_{\mathcal{I}, \lambda_1, \ldots, \lambda_r} ^\nu$ equals the cardinality of 
	\[
	\begin{split}
	\bigg\{{\bf x} = (x_{k, l})_{2 \le k \le r, 1 \le l \le N_k} &\in \widehat{\Delta}_{{\bf i}, \lambda_1, \ldots, \lambda_r} \cap \mathbb{Z}^{N_2 + \cdots + N_r} ~\bigg| \\
	&\lambda_1 + \cdots + \lambda_r - \sum_{2 \le k \le r, 1 \le l \le N_k} x_{k, l} \alpha_{i_{k, l}} = \nu\bigg\}.
	\end{split}
	\]
\end{theorem}

\begin{remark}\normalfont
	Since $\Delta(Z_{\mathcal{I}}, \mathcal{L}_{\mathcal{I}, \lambda_1, \ldots, \lambda_r}, v_{\bf i} ^{\rm high}, \tau_{\mathcal{I}, \lambda_1, \ldots, \lambda_r}) = \Delta(Z_{\bf i}, \mathcal{L}_{{\bf i}, {\bf a}}, v_{\bf i} ^{\rm high}, \tau_{{\bf i}, {\bf a}})$ by Theorem \ref{thm_NOBY_of_fBS}, it is natural to ask why we consider not only $Z_{\bf i}$ but also $Z_{\mathcal{I}}$. The reason is that the space $H^0 (Z_{\bf i}, \mathcal{L}_{{\bf i}, {\bf a}})$ of global sections does not have a natural $G$-module structure because $Z_{\bf i}$ is not a $G$-variety. The theory of flag Bott-Samelson varieties gives a natural framework to relate the usual Bott-Samelson variety $Z_{\bf i}$ with $G$-modules.
\end{remark}

In order to prove Theorem~\ref{t:representation}, we use the theory of crystal bases, see~\cite{Kas95} for a survey on this topic. Lusztig~\cite{Lus90, Lus91, Lus10} and Kashiwara~\cite{Kas91} constructed a specific $\mathbb{C}$-basis of $V(\lambda)$ via the quantized enveloping algebra associated with $\mathfrak{g}$. This is called (the specialization at $q = 1$ of) the \emph{lower global basis} ($=$ the \emph{canonical basis}), and denoted by $\{G_{\lambda}^{\rm low}(b) \mid b \in \mathcal{B}(\lambda)\} \subset V(\lambda)$. See, for example,~\cite[\S 12]{Kas95} for the definition of $G_{\lambda}^{\rm low}(b)$. In this manuscript, we put `$\rm{low}$' to emphasize that we are considering the \textit{lower} global basis while Kashiwara~\cite{Kas95} denoted it by $G_{\lambda}(b)$.
The index set $\mathcal{B}(\lambda)$ is endowed with specific maps 
\begin{gather*}
{\rm wt} \colon \mathcal{B}(\lambda) \rightarrow \chi(H),\ \varepsilon_i, \varphi_i \colon \mathcal{B}(\lambda) \rightarrow \mathbb{Z}_{\ge 0},\ {\rm and}\ \\
\tilde{e}_i, \tilde{f}_i \colon \mathcal{B}(\lambda) \rightarrow \mathcal{B}(\lambda) \cup \{0\}\ {\rm for}\ i \in [n],
\end{gather*}
which have the following properties:
\begin{align*}
&{\rm wt}(b_\lambda) = \lambda,\\ 
&{\rm wt}(\tilde{e}_i b) = {\rm wt}(b) +\alpha_i \quad{\rm if}\ \tilde{e}_i b \neq 0,\\ 
&{\rm wt}(\tilde{f}_i b) = {\rm wt}(b) -\alpha_i \quad{\rm if}\ \tilde{f}_i b \neq 0,\\ 
&\varepsilon_i(b) = \max\{k \in \mathbb{Z}_{\ge 0} \mid \tilde{e}_i ^k b \neq 0\},\\ 
&\varphi_i (b) = \max\{k \in \mathbb{Z}_{\ge 0} \mid \tilde{f}_i ^k b \neq 0\},\\
&e_i \cdot G^{\rm low} _{\lambda} (b) \in \mathbb{C}^\ast G^{\rm low} _{\lambda} (\tilde{e}_i b) + \sum_{\substack{b^\prime \in \mathcal{B}(\lambda);\ {\rm wt}(b^\prime) = {\rm wt}(b) + \alpha_i,\\ \varphi_i (b^\prime) > \varphi_i (b) +1}} \mathbb{C} G^{\rm low} _{\lambda} (b^\prime),\\
&f_i \cdot G^{\rm low} _{\lambda} (b) \in \mathbb{C}^\ast G^{\rm low} _{\lambda} (\tilde{f}_i b) + \sum_{\substack{b^\prime \in \mathcal{B}(\lambda);\ {\rm wt}(b^\prime) = {\rm wt}(b) -\alpha_i,\\ \varepsilon_i (b^\prime) > \varepsilon_i (b) +1}} \mathbb{C} G^{\rm low} _{\lambda} (b^\prime)
\end{align*}
for $i \in [n]$ and $b \in \mathcal{B}(\lambda)$, where $\Cstar = \mathbb{C} \setminus \{0\}$, and $b_\lambda \in \mathcal{B}(\lambda)$ is defined as $G^{\rm low} _{\lambda} (b_\lambda) \in \mathbb{C}^\ast v_\lambda$, called the \emph{highest element}. We call $\mathcal{B}(\lambda)$ the \emph{crystal basis} for $V(\lambda)$, which satisfies the axiom of \emph{crystals}, see~\cite[Definition 1.2.1]{Kas93} for the definition of crystals. The operations $\tilde{e}_i$ and $\tilde{f}_i$ are called the \defi{Kashiwara operators}.

\begin{definition}[{see~\cite[\S 4.2]{Kas95}}]\normalfont
	The {\it crystal graph} of a crystal $\mathcal{B}$ is the $[n]$-colored, directed graph with vertex set $\mathcal{B}$ whose directed edges are given by: $b \xrightarrow{i} b^\prime$ if and only if $b^\prime = \tilde{f}_i b$.
\end{definition}

In this paper, we identify a crystal $\mathcal{B}$ with its crystal graph. By \cite[Theorem 3]{Kas91}, for a $G$-module $V = V(\nu_1) \oplus \cdots \oplus V(\nu_M)$, the crystal graph of the corresponding crystal basis $\mathcal{B}(V)$ is the disjoint union of the crystal graphs $\mathcal{B}(\nu_1), \ldots, \mathcal{B}(\nu_M)$. 

\begin{proposition}[{see~\cite[Proposition 3.2.3]{Kas93}}]\label{p:Demazure crystal}
	Let ${\bf i} = (i_1, \ldots, i_r) \in [n]^r$ be a reduced word for $w \in W$, and $\lambda \in \chi_+ (H)$. Then, the subset \[\mathcal{B}_w(\lambda) \coloneqq \left\{\tilde{f}_{i_1} ^{x_1} \cdots \tilde{f}_{i_r} ^{x_r} b_\lambda ~\Big|~ x_1, \ldots, x_r \in \mathbb{Z}_{\ge 0} \right\} \setminus \{0\} \subset \mathcal{B}(\lambda)\] is independent of the choice of a reduced word ${\bf i}$.
\end{proposition}

The subset $\mathcal{B}_w (\lambda)$ is called a {\it Demazure crystal}. 

\begin{example}\normalfont
	Let $G = \SL(3)$, and $\lambda = \alpha_1 + \alpha_2 = \varpi_1 + \varpi_2$. Then, the crystal graph of $\mathcal{B}(\lambda)$ is given as follows:
	
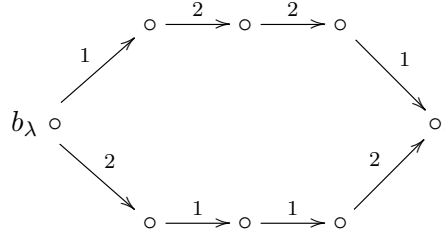
\[\xymatrix{
		& {\circ} \ar[r]^-{2} & {\circ} \ar[r]^-{2} & {\circ} \ar[dr]^-{1}\\
		{b_\lambda\ \circ} \ar[ur]^-{1} \ar[dr]^-{2} & &  & & {\circ}.\\
		& {\circ} \ar[r]^-{1} & {\circ} \ar[r]^-{1} & {\circ} \ar[ur]^-{2} & 
	}\]
	In addition, for $w = s_2 s_1 \in W$, the following directed graph gives the Demazure crystal $\mathcal{B}_w(\lambda)$:
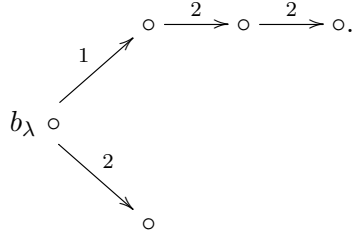
	\[\xymatrix{
		& {\circ} \ar[r]^-{2} & {\circ} \ar[r]^-{2} & {\circ}.\\
		{b_\lambda\ \circ} \ar[ur]^-{1} \ar[dr]^-{2} & &  & \\
		& {\circ} & &
	}\]
\end{example}

The following is an immediate consequence of~\cite[Proposition 3.2.3]{Kas93}.

\begin{lemma}\label{l:crystal for longest}
	Let ${\bf i} = (i_1, \ldots, i_N) \in [n]^N$ be a reduced word for the longest element $w_0 \in W$. Then, the following equalities hold for all $\lambda \in \chi_+ (H)$$:$ \[\mathcal{B}(\lambda) = \mathcal{B}_{w_0}(\lambda) = \{\tilde{f}_{i_1} ^{x_1} \cdots \tilde{f}_{i_N} ^{x_N} b_\lambda \mid x_1, \ldots, x_N \in \mathbb{Z}_{\ge 0}\} \setminus \{0\}.\] In particular, the following equality holds for all $w \in W$$:$
	\[\left\{\tilde{f}_{i_1} ^{x_1} \cdots \tilde{f}_{i_N} ^{x_N} b ~\Big|~ x_1, \ldots, x_N \in \mathbb{Z}_{\ge 0},\ b \in \mathcal{B}_w (\lambda)\right\} \setminus \{0\} = \mathcal{B}(\lambda).\]
\end{lemma}

For two crystals $\mathcal{B}_1, \mathcal{B}_2$, we can define another crystal $\mathcal{B}_1 \otimes \mathcal{B}_2$, called the \emph{tensor product} of $\mathcal{B}_1$ and $\mathcal{B}_2$, see~\cite[\S 1.3]{Kas93} for the definition. For $\lambda_1, \ldots, \lambda_r \in \chi_+ (H)$, the tensor product $\mathcal{B}(\lambda_1) \otimes \cdots \otimes \mathcal{B}(\lambda_r)$ is identical to the crystal basis for the tensor product module $V(\lambda_1) \otimes \cdots \otimes V(\lambda_r)$ by~\cite[Theorem 1]{Kas91}. Let us recall the definitions of generalized Demazure crystals and generalized string polytopes.

\begin{definition}[{see~\cite[\S 1.2]{LLM02}}]\normalfont
	Let ${\bf i} = (i_1, \ldots, i_r) \in [n]^r$ be an arbitrary word, and ${\bf a} = (a_1, \ldots, a_r) \in \mathbb{Z}_{\ge 0} ^r$. We define $\mathcal{B}_{{\bf i}, {\bf a}} \subset \mathcal{B}(a_1 \varpi_{i_1}) \otimes \cdots \otimes \mathcal{B}(a_r \varpi_{i_r})$ to be the subset
	\[
	\begin{split}
	\Big\{\tilde{f}_{i_1} ^{x_1} (b_{a_1 \varpi_{i_1}} \otimes \tilde{f}_{i_2} ^{x_2} (b_{a_2 \varpi_{i_2}} \otimes \cdots \otimes \tilde{f}_{i_{r-1}} ^{x_{r-1}} (b_{a_{r-1} \varpi_{i_{r-1}}} \otimes \tilde{f}_{i_r} ^{x_r} (b_{a_r \varpi_{i_r}}))\cdots))~\Big|~ \\
	{x_1,\dots,x_r \in \mathbb{Z}_{\ge 0}} \Big\} \setminus \{0\};
	\end{split}
	\] this is called a {\it generalized Demazure crystal}.
\end{definition}

\begin{definition}[{\cite[Definition 4.4]{Fuj18}}]\normalfont\label{generalized string parametrization}
	Let ${\bf i} = (i_1, \ldots, i_r) \in [n]^r$ be an arbitrary word, and ${\bf a} = (a_1, \ldots, a_r) \in \mathbb{Z}_{\ge 0} ^r$. For $b \in \mathcal{B}_{{\bf i}, {\bf a}}$, we set $b(1) \coloneqq b$,
	\begin{align*}
	&x_1 \coloneqq \max\{x \in \mathbb{Z}_{\ge 0} \mid \tilde{e}_{i_1} ^x b(1) \neq 0\},\quad \tilde{e}_{i_1} ^{x_1} b(1) = b_{a_1 \varpi_{i_1}} \otimes b(2),\\
	&x_2 \coloneqq \max\{x \in \mathbb{Z}_{\ge 0} \mid \tilde{e}_{i_2} ^x b(2) \neq 0\},\quad \tilde{e}_{i_2} ^{x_2} b(2) = b_{a_2 \varpi_{i_2}} \otimes b(3),\\
	&\ \vdots \\
	&x_r \coloneqq \max\{x \in \mathbb{Z}_{\ge 0} \mid \tilde{e}_{i_r} ^x b(r) \neq 0\}, 
	\end{align*}
	and define the {\it generalized string parametrization} $\Omega_{\bf i} (b)$ of $b$ with respect to ${\bf i}$ by $\Omega_{\bf i} (b) \coloneqq (x_1, \ldots, x_r)$.
\end{definition}

\begin{definition}[{\cite[Definition 4.7]{Fuj18}}]\normalfont\label{generalized string polytope}
	For an arbitrary word ${\bf i} \in [n]^r$ and ${\bf a} \in \mathbb{Z}_{\ge 0} ^r$, define a subset $\mathcal{S}_{{\bf i}, {\bf a}} \subset \mathbb{Z}_{>0} \times \mathbb{Z}^r$ by \[\mathcal{S}_{{\bf i}, {\bf a}} \coloneqq \bigcup_{k>0} \{(k, \Omega_{\bf i}(b)) \mid b \in \mathcal{B}_{{\bf i}, k{\bf a}}\},\] and denote by $\mathcal{C}_{{\bf i}, {\bf a}} \subset \mathbb{R}_{\ge 0} \times \mathbb{R}^r$ the smallest real closed cone containing $\mathcal{S}_{{\bf i}, {\bf a}}$. Let us define a subset $\Delta_{{\bf i}, {\bf a}} \subset \mathbb{R}^r$ by \[\Delta_{{\bf i}, {\bf a}} \coloneqq \{{\bf x} \in \mathbb{R}^r \mid (1, {\bf x}) \in \mathcal{C}_{{\bf i}, {\bf a}}\};\] this is called the {\it generalized string polytope} associated to ${\bf i}$ and ${\bf a}$.
\end{definition}

The following is a fundamental property of generalized string polytopes.

\begin{proposition}[{see~\cite[Corollaries 4.16, 5.4 (3)]{Fuj18}}]\label{finite union of polytopes}
	The generalized string polytope $\Delta_{{\bf i}, {\bf a}}$ is a rational convex polytope, and the equality $\Omega_{\bf i}(\mathcal{B}_{{\bf i}, {\bf a}}) = \Delta_{{\bf i}, {\bf a}} \cap \mathbb{Z}^r$ holds.
\end{proposition}

Fujita proved the following relation between the generalized string polytope and a Newton--Okounkov body of the Bott--Samelson variety $Z_{\bf i}$. 
\begin{theorem}[{see~\cite[Corollary 5.3]{Fuj18}}]\label{thm_generalized_string_polytope}
	Let $Z_{\mathbf i}$ be the Bott--Samelson variety
	determined by a word $\mathbf i \in [n]^r$, and 
	let $\mathcal L_{\mathbf i, \mathbf a}$ be the line
	bundle on $Z_{\mathbf i}$ determined by an integer vector $\mathbf a  \in \Z^r_{\geq 0}$ as in~\eqref{eq_def_of_L_ia}. Then we have that
	\[
	\Delta(Z_{\mathbf i}, \mathcal L_{\bf i, \bf a}, v_{\bf i}^{\rm high}, \tau_{\bf i, \bf a}) = - \Delta_{{\bf i}, {\bf a}}.
	\]
\end{theorem}

\begin{remark}\normalfont
	The combinatorial structure of generalized string polytopes is quite complicated that even their real dimensions are not easy to be determined.
	By~Remark~\ref{remark_dimension_of_NOBY}, Theorem~\ref{thm_generalized_string_polytope} determines the dimensions of generalized string polytopes of the type $\Delta(Z_{\mathbf i}, \eta_{{\bf i}, \mathcal I}^{\ast} \mathcal L, v_{\mathbf i}^{\rm high}, \tau_{\bf i, \bf a})$, where 
	$\mathcal I$ is a sequence of subsets of $[n]$ and $\mathcal L$ is a very ample line bundle over $Z_{\mathcal I}$.
\end{remark}

Let $\mathcal{I} = (I_1, \ldots, I_r)$ be a sequence of subsets of $[n]$, and fix a sequence ${\bf i} = (i_{k, l})_{1 \le k \le r, 1 \le l \le N_k} \in [n]^{N_1 + \cdots + N_r}$ such that $(i_{k, 1}, \ldots, i_{k, N_k})$ is a reduced word for the longest element in $W_{I_k}$ for $1 \le k \le r$. Given $\lambda_1, \ldots, \lambda_r \in \chi_+ (H)$, we denote the dual $P_{I_1}$-module $H^0(Z_{\mathcal{I}}, \mathcal{L}_{\mathcal{I}, \lambda_1, \ldots, \lambda_r})^\ast$ by $V_{\mathcal{I}, \lambda_1, \ldots, \lambda_r}$, and define $\mathcal{B}_{{\bf i}, \lambda_1, \ldots, \lambda_r} \subset \mathcal{B}(\lambda_1) \otimes \cdots \otimes \mathcal{B}(\lambda_r)$ to be the set of elements of the form 
\begin{equation}\label{equation_element_in_generalized_Demazure}
\begin{aligned}
\tilde{f}_{i_{1, 1}} ^{x_{1, 1}} \cdots \tilde{f}_{i_{1, N_1}} ^{x_{1, N_1}}\left(b_{\lambda_1} \otimes \cdots \otimes \tilde{f}_{i_{r-1, 1}} ^{x_{r-1, 1}} \cdots \tilde{f}_{i_{r-1, N_{r-1}}} ^{x_{r-1, N_{r-1}}} \left(b_{\lambda_{r-1}} \otimes \tilde{f}_{i_{r, 1}} ^{x_{r, 1}} \cdots \tilde{f}_{i_{r, N_r}} ^{x_{r, N_r}} (b_{\lambda_r})\right)\cdots\right)
\end{aligned}
\end{equation}
for some $x_{1, 1}, \ldots, x_{1, N_1}, \ldots, x_{r, 1}, \ldots, x_{r, N_r} \in \mathbb{Z}_{\ge 0}$. 

\begin{proposition}\label{proposition_relation_with_generalized_Demazure}
	For $\lambda_1,\dots,\lambda_r \in \chi_+ (H)$, let ${\bf a} \in \mathbb{Z}^{N_1 + \cdots + N_r}$ be the integer vector such that
	$\mathcal{L}_{{\bf i}, {\bf a}} \simeq \eta_{{\bf i}, \mathcal{I}} ^\ast \mathcal{L}_{\mathcal{I}, \lambda_1, \ldots, \lambda_r}$ as given in Proposition~\ref{prop_line_bdle_over_fBS_and_BS}, and let $\mu \in \chi_+ (H)$ be the weight defined in Theorem \ref{thm_isom_between_holom_sections} {\rm (2)}.
	\begin{enumerate}
		\item The $B$-module $V_{\mathcal{I}, \lambda_1, \ldots, \lambda_r}$ is naturally isomorphic to $\mathbb{C}_\mu \otimes V_{{\bf i}, {\bf a}}$, where $V_{{\bf i}, {\bf a}}$ is the generalized Demazure module defined in~\cite[\S 1.1]{LLM02}.
		\item There is a natural bijective map \[\mathcal{B}_{{\bf i}, \lambda_1, \ldots, \lambda_r} \xrightarrow{\sim} b_\mu \otimes \mathcal{B}_{{\bf i}, {\bf a}}\] compatible with the crystal structures.
		\item The crystal graph of $\mathcal{B}_{{\bf i}, \lambda_1, \ldots, \lambda_r}$ is identical to that of $\mathcal{B}_{{\bf i}, {\bf a}}$.
	\end{enumerate}
\end{proposition}

\begin{proof}
	\begin{enumerate}
		\item The assertion is an immediate consequence of Theorem~\ref{thm_isom_between_holom_sections} and~\cite[Theorem 6]{LLM02}.
		\item For $\lambda, \mu \in \chi_+ (H)$, the crystal basis $\mathcal{B}(\lambda +\mu)$ can be regarded as a connected component of $\mathcal{B}(\lambda) \otimes \mathcal{B}(\mu)$ by identifying $b_{\lambda +\mu}$ with $b_\lambda \otimes b_\mu$ (see~\cite[\S 4.5]{Kas95}). If we identify $b_\lambda$ with $b_{\lambda - \langle \lambda, \alpha_i ^\vee \rangle \varpi_i} \otimes b_{\langle \lambda, \alpha_i ^\vee \rangle \varpi_i}$ for $i \in [n]$ and $\lambda \in \chi_+ (H)$, then the definition of tensor product crystals implies that $\tilde{f}_i ^a b_\lambda = b_{\lambda - \langle \lambda, \alpha_i ^\vee \rangle \varpi_i} \otimes \tilde{f}_i ^a b_{\langle \lambda, \alpha_i ^\vee \rangle \varpi_i}$ for all $a \in \mathbb{Z}_{\ge 0}$ (see~\cite[Appendix A]{Fuj18}). Hence it follows that
		\begin{align*}
		&\tilde{f}_{i_{1, 1}} ^{x_{1, 1}} \cdots \tilde{f}_{i_{1, N_1}} ^{x_{1, N_1}}(b_{\lambda_1} \otimes b)\\ 
		=\ &b_{\lambda_1 - \sum_{1 \le l \le N_1} \mu_l} \otimes \tilde{f}_{i_{1, 1}} ^{x_{1, 1}} \left(b_{\mu_1} \otimes \tilde{f}_{i_{1, 2}} ^{x_{1, 2}} \left(b_{\mu_2} \otimes \cdots \otimes \tilde{f}_{i_{1, N_1}} ^{x_{1, N_1}}(b_{\mu_{N_1}} \otimes b) \cdots\right)\right)
		\end{align*}
		for $b \in \mathcal{B}_{{\bf i}_{\ge 2}, \lambda_2, \ldots, \lambda_r}$ and $x_{1, 1}, \ldots, x_{1, N_1} \in \mathbb{Z}_{\ge 0}$, where 
		\[\mu_l \coloneqq  
		\begin{cases}
		\langle \lambda_1, \alpha_{i_{1, l}} ^\vee \rangle \varpi_{i_{1, l}}\quad&{\rm if}\ l = \max\{1 \le q \le N_1 \mid i_{1, q} = i_{1, l}\},\\
		0\quad&{\rm otherwise}
		\end{cases}\]
		for $1 \le l \le N_1$, and ${\bf i}_{\ge 2} \coloneqq (i_{k, l})_{2 \le k \le r, 1 \le l \le N_k}$. By repeating this deformation, all the elements of the form \eqref{equation_element_in_generalized_Demazure} can be naturally written as elements in $b_\mu \otimes \mathcal{B}_{{\bf i}, {\bf a}}$. This proves part (2).
		\item Let us prove that $\tilde{e}_i (b_\mu \otimes b) = b_\mu \otimes \tilde{e}_i b$ for all $i \in [n]$ and $b \in \mathcal{B}_{{\bf i}, {\bf a}}$. By the definition of $\mathcal{B}_{{\bf i}, {\bf a}}$, we have \[{\rm wt}(b) -{\rm wt}(b') \in \sum_{j \in \{i_{k, l} \mid 1 \le k \le r,\ 1 \le l \le N_k\}} \mathbb{Z} \alpha_j\] for all $b, b' \in \mathcal{B}_{{\bf i}, {\bf a}}$. Hence $\mathcal{B}_{{\bf i}, {\bf a}}$ does not have edges labeled by $j \notin \{i_{k, l} \mid 1 \le k \le r,\ 1 \le l \le N_k\}$. From this, we may assume that $i \in \{i_{k, l} \mid 1 \le k \le r,\ 1 \le l \le N_k\}$. Then, we have $\langle \mu, \alpha_{i} ^\vee \rangle = 0$ by the definition of $\mu$, which implies by the definition of tensor product crystals that $\tilde{e}_i (b_\mu \otimes b) = b_\mu \otimes \tilde{e}_i b$. Thus, we have proved that the crystal graph of $b_\mu \otimes \mathcal{B}_{{\bf i}, {\bf a}}$ is identical to that of $\mathcal{B}_{{\bf i}, {\bf a}}$. Then, part (3) follows immediately from part (2). \qedhere
	\end{enumerate}
\end{proof}

Proposition \ref{proposition_relation_with_generalized_Demazure} implies that all the results in~\cite{LLM02} for $V_{{\bf i}, {\bf a}}$ and $\mathcal{B}_{{\bf i}, {\bf a}}$ are applicable also for $V_{\mathcal{I}, \lambda_1, \ldots, \lambda_r}$ and $\mathcal{B}_{{\bf i}, \lambda_1, \ldots, \lambda_r}$. 

\begin{proposition}
	The set $\mathcal{B}_{{\bf i}, \lambda_1, \ldots, \lambda_r}$ depends only on $\mathcal{I}, \lambda_1, \ldots, \lambda_r$, that is, does not depend on the choice of ${\bf i}$.
\end{proposition}

\begin{proof}
	We proceed by induction on $r$. If $r = 1$, then the assertion is an immediate consequence of Proposition \ref{p:Demazure crystal}. Assume that $r \ge 2$, and that $\mathcal{B}_{{\bf i}_{\ge 2}, \lambda_2, \ldots, \lambda_r}$ is independent of the choice of ${\bf i}_{\ge 2}$. By~\cite[Theorem 2]{LLM02} and Proposition \ref{proposition_relation_with_generalized_Demazure}, it follows that $b_{\lambda_1} \otimes \mathcal{B}_{{\bf i}_{\ge 2}, \lambda_2, \ldots, \lambda_r}$ is a disjoint union of Demazure crystals. Hence it suffices to prove that
	for each connected component $\mathcal{B}_v(\lambda)$ of $b_{\lambda_1} \otimes \mathcal{B}_{{\bf i}_{\ge 2}, \lambda_2, \ldots, \lambda_r}$
	 the set \[\mathcal{B}_{v, i_{1, 1}, \ldots, i_{1, N_1}}(\lambda) \coloneqq \left\{\tilde{f}_{i_{1, 1}} ^{x_1} \cdots \tilde{f}_{i_{1, N_1}} ^{x_{N_1}} b ~\Big|~ x_1, \ldots, x_{N_1} \in \mathbb{Z}_{\ge 0},\ b \in \mathcal{B}_v(\lambda)\right\} \setminus \{0\}\] does not depend on the choice of $(i_{1, 1}, \ldots, i_{1, N_1})$. We define $v_1, \ldots, v_{N_1} \in W$ inductively by 
	\begin{align*}
	&v_1 \coloneqq 
	\begin{cases}
	s_{i_{1, N_1}} v\quad &{\rm if}\ \ell(s_{i_{1, N_1}} v) > \ell (v),\\
	v\quad &{\rm if}\ \ell(s_{i_{1, N_1}} v) < \ell (v),
	\end{cases}\\
	&v_l \coloneqq 
	\begin{cases}
	s_{i_{1, N_1 -l +1}} v_{l -1}\quad &{\rm if}\ \ell(s_{i_{1, N_1 -l +1}} v_{l -1}) > \ell (v_{l -1}),\\
	v_{l -1}\quad &{\rm if}\ \ell(s_{i_{1, N_1 -l +1}} v_{l -1}) < \ell (v_{l -1}).
	\end{cases} 
	\end{align*}
	Then, we deduce by~\cite[Proposition~3.2.3~(iii)]{Kas93} that $\mathcal{B}_{v, i_{1, 1}, \ldots, i_{1, N_1}}(\lambda) = \mathcal{B}_{v_{N_1}} (\lambda)$. In addition, it follows by~\cite[Lemma~3.2.1 and Proposition~3.2.3~(i)]{Kas93} that 
	\[\sum_{x_1, \ldots, x_{N_1} \in \mathbb{Z}_{\ge 0}} f_{i_{1, 1}} ^{x_1} \cdots f_{i_{1, N_1}} ^{x_{N_1}} \left(\sum_{b \in \mathcal{B}_v(\lambda)} \mathbb{C} G^{\rm low} _\lambda(b)\right) = \sum_{b \in \mathcal{B}_{v_{N_1}} (\lambda)} \mathbb{C} G^{\rm low} _\lambda(b).\]
	From these, we have
	\[\sum_{b \in \mathcal{B}_{v, i_{1, 1}, \ldots, i_{1, N_1}}(\lambda)} \mathbb{C} G^{\rm low} _\lambda(b) = \sum_{x_1, \ldots, x_{N_1} \in \mathbb{Z}_{\ge 0}} f_{i_{1, 1}} ^{x_1} \cdots f_{i_{1, N_1}} ^{x_{N_1}} \left(\sum_{b \in \mathcal{B}_v(\lambda)} \mathbb{C} G^{\rm low} _\lambda(b)\right);\] the right hand side does not depend on the choice of $(i_{1, 1}, \ldots, i_{1, N_1})$ by~\cite[Proposition~3.2.5~(v)]{Kas93}, which implies that the set $\mathcal{B}_{v, i_{1, 1}, \ldots, i_{1, N_1}}(\lambda)$ is also independent. This proves the proposition.
\end{proof}

We denote $\mathcal{B}_{{\bf i}, \lambda_1, \ldots, \lambda_r}$ by $\mathcal{B}_{\mathcal I, \lambda_1,\dots,\lambda_r}$, which is also called a \emph{generalized Demazure crystal}. By definition, we have 
\begin{align*}
&\mathcal{B}_{\mathcal{I}, \lambda_1, \ldots, \lambda_r} \\
&\quad = \left\{\tilde{f}_{i_{1, 1}} ^{x_1} \cdots \tilde{f}_{i_{1, N_1}} ^{x_{N_1}} (b_{\lambda_1} \otimes b)~\Big|~ x_1, \ldots, x_{N_1} \in \mathbb{Z}_{\ge 0},\ b \in \mathcal{B}_{(I_2, \ldots, I_r), \lambda_2, \ldots, \lambda_r} \right\} \setminus \{0\}.
\end{align*}
Assume that $I_1 = [n]$, and hence that $(i_{1, 1}, \ldots, i_{1, N_1})$ is a reduced word for $w_0 \in W$. By~\cite[Theorem~2]{LLM02} and Proposition \ref{proposition_relation_with_generalized_Demazure}, the set $b_{\lambda_1} \otimes \mathcal{B}_{(I_2, \ldots, I_r), \lambda_2, \ldots, \lambda_r}$ is a disjoint union of Demazure crystals. Hence the second assertion of Lemma~\ref{l:crystal for longest} implies that each connected component of $\mathcal{B}_{\mathcal{I}, \lambda_1, \ldots, \lambda_r}$ is of the form $\mathcal{B}(\nu)$ for some $\nu \in \chi_+ (H)$. Note that the character of $V_{\mathcal{I}, \lambda_1, \ldots, \lambda_r}$ equals the formal character of $\mathcal{B}_{\mathcal{I}, \lambda_1, \ldots, \lambda_r}$ by~\cite[Theorem~5 and Corollary~10]{LLM02} and Proposition \ref{proposition_relation_with_generalized_Demazure}. Since finite-dimensional $G$-modules are characterized by their characters, we obtain the following.

\begin{proposition}\label{p:crystal basis for flag BS}
	Let $\mathcal{I} = (I_1, \ldots, I_r)$ be a sequence of subsets of $[n]$ such that $I_1 = [n]$. Then, the generalized Demazure crystal $\mathcal{B}_{\mathcal{I}, \lambda_1, \ldots, \lambda_r}$ is isomorphic to the crystal basis for the $G$-module $V_{\mathcal{I}, \lambda_1, \ldots, \lambda_r}$. In particular, if $\mathcal{B}_{\mathcal{I}, \lambda_1, \ldots, \lambda_r}$ is the disjoint union of $\mathcal{B}(\nu_1), \ldots, \mathcal{B}(\nu_M)$, then $V_{\mathcal{I}, \lambda_1, \ldots, \lambda_r}$ is isomorphic to $V(\nu_1) \oplus \cdots \oplus V(\nu_M)$.
\end{proposition}

Since the crystal graph of $\mathcal{B}_{{\bf i}, {\bf a}}$ is identical to that of $\mathcal{B}_{\mathcal{I}, \lambda_1, \ldots, \lambda_r}$ by Proposition \ref{proposition_relation_with_generalized_Demazure}~(3), the generalized string parametrization $\Omega_{\bf i}$ of $\mathcal{B}_{{\bf i}, {\bf a}}$ can be regarded as a parametrization of $\mathcal{B}_{\mathcal{I}, \lambda_1, \ldots, \lambda_r}$. We denote $\Delta_{{\bf i}, {\bf a}}$ by $\Delta_{{\bf i}, \lambda_1, \ldots, \lambda_r}$. Then, we have $\widehat{\Delta}_{{\bf i}, \lambda_1, \ldots, \lambda_r} = \pi_{\ge 2} (\Delta_{{\bf i}, \lambda_1, \ldots, \lambda_r})$ by Theorems~\ref{thm_NOBY_of_fBS}, \ref{thm_generalized_string_polytope}. 

\begin{proof}[Proof of Theorem~\ref{t:representation}]
	By Proposition~\ref{p:crystal basis for flag BS}, the multiplicity $c_{\mathcal{I}, \lambda_1, \ldots, \lambda_r} ^\nu$ equals the number of connected components of $\mathcal{B}_{\mathcal{I}, \lambda_1, \ldots, \lambda_r}$ isomorphic to $\mathcal{B}(\nu)$. For $b \in \mathcal{B}_{\mathcal{I}, \lambda_1, \ldots, \lambda_r}$, we write $\Omega_{\bf i}(b) = (x_{1, 1}, \ldots, x_{1, N_1}, \ldots, x_{r, 1}, \ldots, x_{r, N_r})$. By the definition of $\Omega_{\bf i}$, we have 
	\begin{equation}\label{equation_usual_string_parametrization}
	\begin{aligned}
	x_{1, l} = \max\{x \in \mathbb{Z}_{\ge 0} \mid \tilde{e}_{i_{1, l}} ^x \tilde{e}_{i_{1, l-1}} ^{x_{1, l-1}} \cdots \tilde{e}_{i_{1, 1}} ^{x_{1, 1}} b \neq 0\}
	\end{aligned}
	\end{equation}
	for $1 \le l \le N_1$. Let $\mathcal{C}_b$ denote the connected component of $\mathcal{B}_{\mathcal{I}, \lambda_1, \ldots, \lambda_r}$ containing $b$. Since $I_1 = [n]$, it follows that $(i_{1, 1}, \ldots, i_{1, N_1})$ is a reduced word for $w_0 \in W$. So we deduce by~\cite[Proposition~3.2.3]{Kas93} that $\tilde{e}_{i_{1, N_1}} ^{x_{1, N_1}} \cdots \tilde{e}_{i_{1, 1}} ^{x_{1, 1}} b$ is the highest element in $\mathcal{C}_b$. Hence \begin{equation}\label{equation_par_of_height_wts_in_Cb}
	(0, \ldots, 0, x_{2, 1}, \ldots, x_{2, N_2}, \ldots, x_{r, 1}, \ldots, x_{r, N_r})
	\end{equation} 
	is the generalized string parametrization of the highest element. In particular, the surjective map $\mathcal{B}_{\mathcal{I}, \lambda_1, \ldots, \lambda_r} \twoheadrightarrow \widehat{\Delta}_{{\bf i}, \lambda_1, \ldots, \lambda_r} \cap \mathbb{Z}^{N_2 + \cdots + N_r}$ given by $b \mapsto \pi_{\ge 2}(\Omega_{\bf i}(b))$ induces a bijective map 
	\[\Psi \colon \{{\rm connected\ components\ of}\ \mathcal{B}_{\mathcal{I}, \lambda_1, \ldots, \lambda_r}\} \xrightarrow{\sim} \widehat{\Delta}_{{\bf i}, \lambda_1, \ldots, \lambda_r} \cap \mathbb{Z}^{N_2 + \cdots + N_r}.\] In addition, for a connected component $\mathcal{C}$ of $\mathcal{B}_{\mathcal I, \lambda_1,\dots,\lambda_r}$, the weight of the highest element in $\mathcal{C}$ is determined by $\Psi(\mathcal{C})$ due to the definition of generalized string parametrizations (see Definition~\ref{generalized string parametrization}).  Indeed, if $ \Psi (\mathcal{C}) = (x_{2, 1}, \ldots, x_{2, N_2}, \allowbreak \ldots, x_{r, 1}, \ldots, x_{r, N_r})$, then the weight of the highest element in $\mathcal{C}$ is given by 
	\[
	\lambda_1 + \cdots + \lambda_r - \sum_{2 \le k \le r, 1 \le l \le N_k} x_{k, l} \alpha_{i_{k, l}}
	\]
	since the generalized string parametrization of this element is given by~\eqref{equation_par_of_height_wts_in_Cb}.
	By these reasons, we deduce the assertion of the theorem.
\end{proof}

The following is an immediate consequence of the proof of Theorem~\ref{t:representation}.

\begin{corollary}\label{c:connected component}
	Let $\mathcal{I} = (I_1, \ldots, I_r)$ be a sequence of subsets of $[n]$ such that $I_1 = [n]$. Then, the number of connected components of $\mathcal{B}_{\mathcal{I}, \lambda_1, \ldots, \lambda_r}$ equals the cardinality of \[\widehat{\Delta}_{{\bf i}, \lambda_1, \ldots, \lambda_r} \cap \mathbb{Z}^{N_2 + \cdots + N_r}.\]
\end{corollary}

Let $\pi_1 \colon \mathbb{R}^{N_1 + \cdots + N_r} \twoheadrightarrow \mathbb{R}^{N_1}$ denote the canonical projection given by \[(x_{1, 1}, \ldots, x_{1, N_1}, \ldots, x_{r, 1}, \ldots, x_{r, N_r}) \mapsto (x_{1, 1}, \ldots, x_{1, N_1}).\] 

\begin{proposition}\label{proposition_fiber_equal_string}
	For ${\bf x} \in \widehat{\Delta}_{{\bf i}, \lambda_1, \ldots, \lambda_r} \cap \mathbb{Z}^{N_2 + \cdots + N_r}$, the set $\pi_1 (\pi_{\ge 2} ^{-1}({\bf x}) \cap \Delta_{{\bf i}, \lambda_1, \ldots, \lambda_r})$ is identical to the string polytope for the connected component $\Psi ^{-1}({\bf x})$ of $\mathcal{B}_{\mathcal{I}, \lambda_1, \ldots, \lambda_r}$ with respect to the reduced word $(i_{1, 1}, \ldots, i_{1, N_1})$ for $w_0 \in W;$ see~\cite[Definition~3.5]{Kav15} and \cite[\S 1]{Lit98}  for the definition of string polytopes.
\end{proposition}

\begin{proof}
	Recall that $\Omega_{\bf i} \colon \mathcal{B}_{\mathcal{I}, \lambda_1, \ldots, \lambda_r} \rightarrow \Delta_{{\bf i}, \lambda_1, \ldots, \lambda_r} \cap \mathbb{Z}^{N_1 + \cdots + N_r}$ is bijective by Proposition \ref{finite union of polytopes}. Hence by the definition of $\Psi$, we obtain the following bijective map:
	\begin{align*}
	\Psi ^{-1}({\bf x}) &\rightarrow \pi_{\ge 2} ^{-1}({\bf x}) \cap \Delta_{{\bf i}, \lambda_1, \ldots, \lambda_r} \cap \mathbb{Z}^{N_1 + \cdots + N_r},\\ 
	b &\mapsto \Omega_{\bf i} (b).
	\end{align*}
	In addition, we see by \eqref{equation_usual_string_parametrization} that $\pi_1 (\Omega_{\bf i} (b))$ is the string parametrization of $b \in \Psi ^{-1}({\bf x})$ with respect to the reduced word $(i_{1, 1}, \ldots, i_{1, N_1})$; see~\cite[\S 1]{Lit98} and \cite[Definition~3.2]{Kav15} for the definition of string parametrizations. From these, we obtain the assertion of the proposition.
\end{proof}

\begin{remark}\normalfont
	Kaveh and Khovanskii~\cite{KaKh12a} gave a general framework to describe multiplicities of irreducible representations by using the Newton--Okounkov bodies. Our results give concrete constructions of convex bodies appearing in~\cite{KaKh12a}. Indeed, by the proof of Theorem~\ref{t:representation} and \cite[Theorem 5.2]{Fuj18}, it is not hard to prove that the rational convex polytope $\widehat{\Delta}_{{\bf i}, \lambda_1, \ldots, \lambda_r}$ is identical to the multiplicity convex body $\widehat{\Delta}_G (A)$ in~\cite[\S 4.1]{KaKh12a} for the valuation $v_{\bf i} ^{\rm high}$, where \[A \coloneqq \bigoplus_{k \ge 0} H^0(Z_{\mathcal{I}}, \mathcal{L}_{\mathcal{I}, \lambda_1, \ldots, \lambda_r} ^{\otimes k}).\] From this and Proposition \ref{proposition_fiber_equal_string}, we deduce that the generalized string polytope $\Delta_{{\bf i}, \lambda_1, \ldots, \lambda_r}$ equals the string convex body $\widetilde{\Delta} (A)$ in~\cite[\S 5.2]{KaKh12a}.
\end{remark}

In representation theory, it is a fundamental problem to determine the $G$-module structure of the tensor product module $V(\lambda) \otimes V(\mu)$, which is equivalent to determining the multiplicity $c_{\lambda, \mu} ^\nu$ of $V(\nu)$ in $V(\lambda) \otimes V(\mu)$. Berenstein and Zelevinsky~\cite[Theorems 2.3, 2.4]{BeZe01} describes the multiplicity $c_{\lambda, \mu} ^\nu$ as the number of lattice points in some explicit rational convex polytope. In the following, we see that Theorem~\ref{t:representation} gives a different approach to such polyhedral expressions for $c_{\lambda, \mu} ^\nu$. Let us consider the case $\mathcal{I} = ([n], [n])$. In this case, the flag Bott--Samelson variety $Z_{\mathcal{I}}$ is identical to $G \times_B G/B$, and the following map is an isomorphism of varieties: 
\[Z_{\mathcal{I}} \xrightarrow{\sim} G/B \times G/B,\ [g_1, g_2] \mapsto (g_1 B/B, g_1 g_2 B/B);\] 
the inverse map is given by $(g_1 B/B, g_2 B/B) \mapsto [g_1, g_1 ^{-1} g_2]$. It is easily seen that under the isomorphism $Z_{\mathcal{I}} \simeq G/B \times G/B$, the $G$-action on $Z_{\mathcal{I}}$ coincides with the diagonal action on $G/B \times G/B$, and the line bundle $\mathcal{L}_{\mathcal{I}, \lambda, \mu}$ corresponds to the direct product of $\mathcal{L}_{\lambda}$ and $\mathcal{L}_{\mu}$, where $\mathcal{L}_\nu$ denotes the line bundle $\mathcal{L}_{([n]), \nu}$ over $G/B$ for $\nu \in \chi_+ (H)$. Hence we obtain the following isomorphisms of $G$-modules:
\begin{align*}
H^0(Z_{\mathcal{I}}, \mathcal{L}_{\mathcal{I}, \lambda, \mu})^\ast &\simeq H^0(G/B \times G/B, \mathcal{L}_{\lambda} \times \mathcal{L}_{\mu})^\ast\\ 
&\simeq H^0(G/B, \mathcal{L}_{\lambda})^\ast \otimes H^0(G/B, \mathcal{L}_{\mu})^\ast\\ 
&\simeq V(\lambda) \otimes V(\mu)
\end{align*} 
by the Borel--Weil theorem (see~\cite[Corollary~II.5.6]{Jan07}). If we write \[V(\lambda) \otimes V(\mu) \simeq \bigoplus_{\nu \in \chi_+ (H)} V(\nu) ^{\oplus c_{\lambda, \mu} ^{\nu}}\] as a $G$-module, then we obtain the following by Theorem~\ref{t:representation}:

\begin{theorem}\label{c:tensor product multiplicity}
	Let $\mathcal{I} = ([n], [n])$, and $(i_1, \ldots, i_N), (j_1, \ldots, j_N) \in [n]^N$ reduced words for $w_0 \in W$. Then, the tensor product multiplicity $c_{\lambda, \mu} ^{\nu}$ equals the cardinality of \[\left\{(y_1, \ldots, y_N) \in \widehat{\Delta}_{{\bf i}, \lambda, \mu} \cap \mathbb{Z}^N ~\bigg|~ \lambda + \mu - \sum_{1 \le l \le N} y_l \alpha_{j_l} = \nu\right\}\!,\] where ${\bf i} \coloneqq (i_1, \ldots, i_N, j_1, \ldots, j_N)$.
\end{theorem}

\begin{example}\label{ex_projection_of_gstring_polytope}\normalfont
	Let $G = \SL(3)$, $\mathcal{I} = ([2], [2])$, and ${\bf i} = (1, 2, 1, 1, 2, 1)$. By~\cite[Corollary~4.15]{Fuj18}, the generalized string polytope $\Delta_{{\bf i}, \lambda, \mu}$ is identical to the set of $(x_1, x_2, x_3, y_1, y_2, y_3) \in \mathbb{R}_{\ge 0} ^6$ satisfying the following inequalities:
	\begin{align*}
	&0 \le y_3 \le \min\{\lambda_2, \mu_1\},\\
	&y_3 \le y_2 \le y_3 + \mu_2,\\
	&y_2 - \lambda_2 \le y_1 \le \min\{\lambda_1, y_2 -2y_3 +\mu_1\},\\
	&\max\{y_3 -\lambda_2, -y_1 +y_2 -\lambda_2\} \le x_3 \le -2y_1 +y_2 -2y_3 +\lambda_1 +\mu_1,\\
	&x_3 \le x_2 \le x_3 +y_1 -2y_2 +y_3 +\lambda_2 +\mu_2,\\
	&0 \le x_1 \le x_2 -2x_3 -2y_1 +y_2 -2y_3 +\lambda_1 +\mu_1, 
	\end{align*}
	where $\lambda_i \coloneqq \langle \lambda, \alpha_i ^\vee\rangle$ and $\mu_i \coloneqq \langle \mu, \alpha_i ^\vee\rangle$ for $i = 1, 2$. Hence the polytope $\widehat{\Delta}_{{\bf i}, \lambda, \mu}$ is identical to the set of $(y_1, y_2, y_3) \in \mathbb{R}_{\ge 0} ^3$ satisfying the following inequalities:
	\begin{align*}
	&0 \le y_3 \le \min\{\lambda_2, \mu_1\},\\
	&y_3 \le y_2 \le y_3 + \mu_2,\\
	&y_2 - \lambda_2 \le y_1 \le \min\{\lambda_1, y_2 -2y_3 +\mu_1\}.
	\end{align*}
	We deduce by Theorem~\ref{c:tensor product multiplicity} that the tensor product multiplicity $c_{\lambda, \mu} ^{\nu}$ equals the cardinality of $(y_1, y_2, y_3) \in \widehat{\Delta}_{{\bf i}, \lambda, \mu} \cap \mathbb{Z}^3$ such that $\lambda +\mu -(y_1 +y_3) \alpha_1 -y_2 \alpha_2 = \nu$. 
	
	If $\lambda = \mu = \varpi_1 +\varpi_2$, then the polytope $\widehat{\Delta}_{{\bf i}, \lambda, \mu}$ is identical to the set of $(y_1, y_2, y_3) \in \mathbb{R}_{\ge 0} ^3$ satisfying the following inequalities:
	\begin{align*}
	0 \le y_3 \le 1,\quad y_3 \le y_2 \le y_3 + 1,\quad y_2 -1 \le y_1 \le \min\{1, y_2 -2y_3 +1\};
	\end{align*}
	see Figure~\ref{figure_ex_projection_of_gstring_polytope}. 
	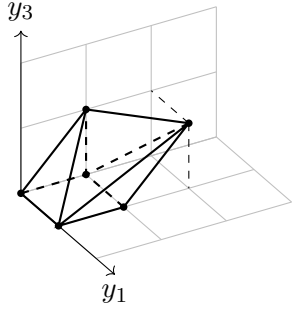
\begin{figure}
		\tdplotsetmaincoords{60}{60}
		\begin{tikzpicture}[tdplot_main_coords]
		
		\coordinate (1) at (0,0,0);
		\coordinate (2) at (1,0,0); 
		\coordinate (3) at (0,1,0);
		\coordinate (4) at (1,1,0);
		\coordinate (5) at (0,1,1);
		\coordinate (6) at (1,2,1);
		
		\begin{scope}[color=gray!50, thin]
		\foreach \xi in {0,1,2} { \draw (\xi, 0,0) -- (\xi, 3, 0); }%
		\foreach \yi in {0,1,2,3} {\draw (0,\yi,2) -- (0,\yi,0) -- (2,\yi,0);}%
		\foreach \zi in {0,1,2} {\draw (0,3,\zi) -- (0,0,\zi);}%
		\end{scope}
		
		\draw[->] (0,0,0) -- (2.5,0,0) node[anchor = north] {$y_1$};
		\draw[->] (0,0,0) -- (0,0,2.5) node[anchor = south] {$y_3$};
		
		\draw[thick] (1)-- (5) -- (6)--(4)--(2)--(6);
		\draw[thick] (5) -- (2) -- (1) ;
		
		\draw[thick, dashed] (6)--(3)--(1);
		\draw[thick, dashed] (4)--(3)--(5);
		
		\draw[dashed] (6) -- (1,2,0);
		\draw[dashed] (6) -- (0,2,1);

		\foreach \x in {1,...,6}{
			\node[circle,fill=black,inner sep=1pt] at (\x) {};
		}
		\end{tikzpicture}
		\caption{The polytope $\widehat{\Delta}_{\mathbf{i}, \lambda, \mu}$ in Example~\ref{ex_projection_of_gstring_polytope}.}
		\label{figure_ex_projection_of_gstring_polytope}
	\end{figure}
	Hence we deduce that \[V(\varpi_1 +\varpi_2)^{\otimes 2} \simeq V(2\varpi_1 +2\varpi_2) \oplus V(3\varpi_1) \oplus V(3\varpi_2) \oplus V(\varpi_1 +\varpi_2)^{\oplus 2} \oplus V(0).\]
\end{example}

Theorem~\ref{t:representation} can be applied to a more general class of representations than~\cite{BeZe01}. We next consider the case $\mathcal{I} = (\underbrace{[n], [n], \ldots, [n]}_r)$. In this case, we have \[Z_{\mathcal{I}} = \underbrace{G \times_B G \times_B \cdots \times_B G}_r /B,\] and this is isomorphic to $(G/B)^r \coloneqq \underbrace{G/B \times G/B \times \cdots \times G/B}_r$ as follows:
\[Z_{\mathcal{I}} \xrightarrow{\sim} (G/B)^r,\ [g_1, g_2, \ldots, g_r] \mapsto (g_1 B/B, g_1 g_2 B/B, \ldots, g_1 g_2 \cdots g_r B/B);\] 
the inverse map is given by $(g_1 B/B, g_2 B/B, \ldots, g_r B/B) \mapsto [g_1, g_1 ^{-1} g_2, g_2 ^{-1} g_3, \ldots, g_{r-1} ^{-1} g_r]$. As in the case $r = 2$, under the isomorphism $Z_{\mathcal{I}} \simeq (G/B)^r$, the $G$-action on $Z_{\mathcal{I}}$ coincides with the diagonal action on $(G/B)^r$, and the line bundle $\mathcal{L}_{\mathcal{I}, \lambda_1, \ldots, \lambda_r}$ corresponds to the direct product of $\mathcal{L}_{\lambda_1}, \ldots, \mathcal{L}_{\lambda_r}$. From this, we have the following isomorphisms of $G$-modules:
\begin{align*}
H^0(Z_{\mathcal{I}}, \mathcal{L}_{\mathcal{I}, \lambda_1, \ldots, \lambda_r})^\ast &\simeq H^0((G/B)^r, \mathcal{L}_{\lambda_1} \times \cdots \times \mathcal{L}_{\lambda_r})^\ast\\ 
&\simeq H^0(G/B, \mathcal{L}_{\lambda_1})^\ast \otimes \cdots \otimes H^0 (G/B, \mathcal{L}_{\lambda_r})^\ast\\ 
&\simeq V(\lambda_1) \otimes \cdots \otimes V(\lambda_r).
\end{align*}
If we write \[V(\lambda_1) \otimes \cdots \otimes V(\lambda_r) \simeq \bigoplus_{\nu \in \chi_+ (H)} V(\nu) ^{\oplus c_{\lambda_1, \ldots, \lambda_r} ^{\nu}}\] as a $G$-module, then Theorem~\ref{t:representation} implies the following.

\begin{corollary}
	Let $\mathcal{I} = (\underbrace{[n], [n], \ldots, [n]}_r)$, and take reduced words $(i_{k, 1}, \ldots, i_{k, N}) \in [n]^N$, $1 \le k \le r$, for $w_0 \in W$. Then, the multiplicity $c_{\lambda_1, \ldots, \lambda_r} ^{\nu}$ equals the cardinality of 
	\begin{align*}
	\bigg\{{\bf x} = (x_{k, l})_{2 \le k \le r, 1 \le l \le N} &\in \widehat{\Delta}_{{\bf i}, \lambda_1, \ldots, \lambda_r} \cap \mathbb{Z}^{(r-1) N} ~\bigg| \\
	&\lambda_1 + \cdots + \lambda_r - \sum_{2 \le k \le r, 1 \le l \le N} x_{k, l}\alpha_{i_{k, l}} = \nu \bigg\},
	\end{align*} 
	where ${\bf i} \coloneqq (i_{1, 1}, \ldots, i_{1, N}, \ldots, i_{r, 1}, \ldots, i_{r, N})$.
\end{corollary}

The following gives an application to $Z_{\mathcal{I}}$ for general $\mathcal{I}$ which does not necessarily start with $[n]$:

\begin{corollary}
	Let $\mathcal{I} = (I_1, \ldots, I_r)$ be a sequence of subsets of $[n]$, and set $I_0 \coloneqq [n]$. Fix ${\bf i}_0 = (i_{k, l})_{0 \le k \le r, 1 \le l \le N_k} \in [n]^{N_0 + \cdots + N_r}$ such that $(i_{k, 1}, \ldots, i_{k, N_k})$ is a reduced word for the longest element in $W_{I_k}$ for $0 \le k \le r$. Then, the number of connected components of $\mathcal{B}_{\mathcal{I}, \lambda_1, \ldots, \lambda_r}$ equals the cardinality of \[\widehat{\Delta}_{{\bf i}_0, 0, \lambda_1, \ldots, \lambda_r} \cap \mathbb{Z}^{N_1 + \cdots + N_r}.\]
\end{corollary}

\begin{proof}
	We set $\mathcal{I}_0 \coloneqq (I_0, I_1, \ldots, I_r)$. By the definition of tensor product crystals, the bijective map $\mathcal{B}_{\mathcal{I}, \lambda_1, \ldots, \lambda_r} \xrightarrow{\sim} b_0 \otimes \mathcal{B}_{\mathcal{I}, \lambda_1, \ldots, \lambda_r}$, $b \mapsto b_0 \otimes b$, is compatible with their crystal structures,  where we mean by $b_0 \in \mathcal{B}(0)$ the element $b_{\lambda}$ for $\lambda = 0$. Hence we may identify $b_0 \otimes \mathcal{B}_{\mathcal{I}, \lambda_1, \ldots, \lambda_r}$ with $\mathcal{B}_{\mathcal{I}, \lambda_1, \ldots, \lambda_r}$. This implies by the definition that the crystal basis $\mathcal{B}_{\mathcal{I}_0, 0, \lambda_1, \ldots, \lambda_r}$ is obtained from $\mathcal{B}_{\mathcal{I}, \lambda_1, \ldots, \lambda_r}$ by actions of $\tilde{f}_i$, $i \in [n]$. By~\cite[Proof of Theorem 2]{LLM02} and Proposition \ref{proposition_relation_with_generalized_Demazure}, all connected components of $\mathcal{B}_{\mathcal{I}, \lambda_1, \ldots, \lambda_r}$ are Demazure crystals in connected components of $\mathcal{B}(\lambda_1) \otimes \cdots \otimes \mathcal{B}(\lambda_r)$. Hence they are not joined by $\tilde{f}_i$, $i \in [n]$, since they have different highest elements. From these, the crystal basis $\mathcal{B}_{\mathcal{I}_0, 0, \lambda_1, \ldots, \lambda_r}$ has the same number of connected components as $\mathcal{B}_{\mathcal{I}, \lambda_1, \ldots, \lambda_r}$, which implies the assertion of the corollary by Corollary~\ref{c:connected component}. 
\end{proof}

\section{Flag Bott--Samelson varieties and flag Bott towers}
\label{sec_fBS_and_fBT}
In this section, we study complex structures on the flag Bott--Samelson variety $Z_{\mathcal I}$, and its relation with a flag Bott tower in Theorem~\ref{thm_fBS_is_fBT}. 
We first recall flag Bott manifolds introduced in~\cite{KLSS}.
Let $M$ be a complex manifold and $E$ a holomorphic
vector bundle over $M$. The \defi{associated flag bundle} $\flag(E) \to M$
is a fiber bundle obtained from $E$ by replacing each fiber $E_p$  over a point $p \in M$
by the full flag manifold $\flag(E_p)$. 
\begin{definition}[{\cite[Definition 2.1]{KLSS}}]\label{def_fBT}
	A \defi{flag Bott tower} $\{F_k\}_{0 \leq k \leq r}$ of height~$r$
	(or an \defi{$r$-stage flag Bott tower}) is a sequence,
	\[
	\begin{tikzcd}
	F_r \arrow[r, "p_r"]& F_{r-1} \arrow[r, "p_{r-1}"] & \cdots \arrow[r, "p_2"]& F_1 \arrow[r, "p_1"] & F_0 = \{\text{a point}\}
	\end{tikzcd}
	\]
	of manifolds $F_k = \flag\left( \bigoplus_{l=1}^{m_k+1} \xii{l}{k} \right)$,
	where $\xii{l}{k}$ is a holomorphic line bundle over $F_{k-1}$ for each
	$1\leq l \leq m_k+1$ and $1\leq k \leq r$. We call $F_k$ 
	the \defi{$k$-stage flag Bott manifold} of the flag Bott tower. 
\end{definition}
For example, the flag manifold $\flag(\C^{m+1}) = \flag(m+1)$ is a 1-stage
flag Bott manifold, and the product of flag manifolds
$\flag(m_1+1) \times \cdots \times \flag(m_r+1)$ is an $r$-stage flag Bott manifold. Also an $r$-stage Bott manifold is an $r$-stage
flag Bott manifold (see~\cite{GrKa94} for the definition of Bott manifolds). 
We call two flag Bott towers $\{F_k\}_{0 \leq k \leq r}$ and
$\{F_k'\}_{0 \leq k \leq r}$ \defi{isomorphic} if there
is a collection of diffeomorphisms 
$\varphi \colon F_k \to F_k'$ which commutes with the maps 
$p_k \colon F_k \to F_{k-1}$ and $p_k' \colon F_k' \to F_{k-1}'$. 

\begin{remark}\label{rmk_fBT_general_type}
In~\cite{KKLS}, an iterated flag bundle whose fibers are not only full flag manifolds of type $A$ but also other flag manifolds of general Lie type is considered. We recall their construction briefly.	
	For $1 \leq k \leq r$, let $K_k$ be a compact connected Lie group, $T_k \subset K_k$  a maximal torus, and $Z_k \subset K_k$  the centralizer of a circle subgroup of $T_k$. 
Recall from~\cite[Definition~3.1]{KKLS} that an \defi{$r$-stage flag Bott tower $\{F_k\}_{0 \leq k \leq r}$ of general Lie type associated to $\{(K_{k}, Z_{k})\}_{0 \leq k \leq r}$} is defined recursively:
\begin{enumerate}
	\item $F_0$ is a point.
	\item $F_k$ is the flag bundle over $F_{k-1}$ with fiber $K_k/Z_k$ associated to a map
	\[
	f_k \colon F_{k-1} \to BK_k,
	\]
	where $f_k$ factors through $BT_k$.
\end{enumerate}
Here, the map $f_k$ induces the flag bundle $F_k \to F_{k-1}$ from the universal flag bundle $K_k/Z_k \hookrightarrow BZ_k \to BK_k$. 
\begin{equation}
\begin{tikzcd}[column sep = 0cm]
K_{k}/Z_k \arrow[rr, equal] \arrow[d, hook] && K_{k}/Z_k  \arrow[d, hook]\\
F_k \arrow[rr] \arrow[d]  && BZ_k \arrow[d] \\
F_{k-1} \arrow[rr, "f_{k}"] \arrow[dr] && BK_{k} \\
& BT_k \arrow[ur]
\end{tikzcd}
\end{equation}
Because  the map $f_k$ factors through $BT_k$, the bundle $F_k$ is the associated $K_k/Z_k$-flag bundle of the sum of complex line bundles over $F_{k-1}$. A flag Bott tower defined in Definition~\ref{def_fBT} is a flag Bott tower of general Lie type associated to $\{(U(m_k+1),T^{m_k+1})\}_{0\leq k \leq r}$.
\end{remark}

\begin{lemma}\label{lemma_tensor_preserves_fBT_str}
	Let $M$ be a complex manifold and $E$ a
	holomorphic vector bundle over $M$. 
	Let $\mathcal L$ be a holomorphic line bundle over $M$. 
	Then we have that $\flag(E) \cong \flag(E \otimes \mathcal L)$ 
	as differentiable manifolds.
\end{lemma}
\begin{proof}
	It is well-known that for a holomorphic vector bundle $E \to M$ over a smooth manifold $M$ and a holomorphic line bundle $\mathcal L \to M$,
	there is a diffeomorphism $\mathbb{P}(E \otimes \mathcal L) \cong \mathbb{P}(E)$ (see, for example,~\cite[Lemma~2.1]{CMS10-trans}).
	Since the induced flag bundle is a sequence of projective bundles as shown in~\cite[Proposition 21.15]{BoTu82}, we have a diffeomorphism $\flag(E) \cong \flag(E \otimes \mathcal L)$.
\end{proof}

The flag manifold $\flag(m+1)$ can be identified with 
an orbit space $\GL(m+1)/B_{\GL(m+1)}$. Similarly, an $r$-stage 
flag Bott manifold $F_r$ can also be considered as an orbit space. 
We briefly review the orbit space construction of~\cite[\S 2.2]{KLSS}. 
Recall from~\cite[Lemma~2.12]{KLSS} that for a given Bott tower $\{F_k\}_{0 \leq k \leq r}$ such that $\flag(m_k+1) \hookrightarrow F_k \to F_{k-1}$, there is a surjective group homomorphism: 
\begin{equation}\label{eq_Pic_gp_of_Fj}
\psi \colon \Z^{m_1+1} \times \cdots 
\times \Z^{m_k+1} \twoheadrightarrow \text{Pic}(F_k)
\quad \text{ for } 1\leq k\leq r.
\end{equation}
We briefly explain the geometric meaning of the homomorphism~\eqref{eq_Pic_gp_of_Fj}. 
For the flag bundle $\flag(E) \stackrel{p}{\to} M$ obtained by a vector bundle $E$ of rank $n$ over a complex manifold $M$, consider the universal flag of bundles $0 \subset E_1 \subset E_2 \subset \cdots \subset E_n = p^{\ast}E$ on $\flag(E)$. Then every element of $\Pic(\flag(E))$ can be written as a polynomial in $x_i = c_1(E_i/E_{i-1})$ for $1 \leq i \leq n$ with coefficients in $\Pic(M)$ (see, for example,~\cite[Example~3.3.5]{Ful13}). 
Because $F_k$ is an iterated flag bundle, applying this procedure  recurrently, we obtain the homomorphism~\eqref{eq_Pic_gp_of_Fj}. Moreover, for $\xi \in \Pic(F_k)$, if we have 
$\xi = \psi(\mathbf a_1,\dots, \mathbf a_k)$, where $\mathbf{a}_j$ is an integer vector $(\mathbf{a}_{j}(1),\dots,\mathbf{a}_j(m_j+1)) \in \Z^{m_j+1}$ for $1 \leq j \leq k$, then 
\begin{equation}\label{eq_c1_xi_and_a}
c_1(\xi) = \sum_{j=1}^k \sum_{l=1}^{m_j+1} {\mathbf{a}_{j}(l)} x_{j,l}.
\end{equation}
Here, $x_{j,l}$ is the first Chern class of the quotient bundle $E_{j,l}/E_{j,l-1}$ obtained by the universal flag of bundles $0 \subset E_{j,1} \subset E_{j,2} \subset \cdots \subset E_{j,m_j+1}$ on $F_{j}$.\footnote{The classes $x_{j,l}$ generate the cohomology $H^2(F_j;\Z)$ with the relations $x_{j,1} + \dots + x_{j,m_j+1} = c_1(\xii{j}{1}) + \cdots + c_1(\xii{j}{m_j+1})$ for $1 \leq j \leq k$ (cf.~\cite[Example~3.3.5]{Ful13} or \cite[Corollary~2.4]{KKLS}).} 

Suppose that $c_1(\xii{k}{l})$ is determined by a set of integer vectors 
$\{\bolda{l}{k}{j} \in \Z^{m_{j}+1}\}_{1 \leq l \leq m_k+1, 1 \leq j < k \leq r}$:
\[
\psi(\bolda{l}{k}{1},\bolda{l}{k}{2},\dots,\bolda{l}{k}{k-1})
= \xii{l}{k} \to F_{k-1}
\]
for each $1\leq l \leq m_k+1$ and $2 \leq k \leq r$. 
Using this set of integer vectors, we define a right action 
$\Phi_k$ of $B_{\GL(m_1+1)} \times \cdots \times B_{\GL(m_k+1)}$ 
on $\GL(m_1+1) \times \cdots \times \GL(m_k +1)$ as
\begin{align*}
	&\Phi_k((g_1,\dots,g_k),(b_1,\dots,b_k)) \\
	&\qquad \coloneqq
	(g_1b_1, \Lambda_{2,1}(b_1)^{-1} g_2b_2,
	\Lambda_{3,1}(b_1)^{-1} \Lambda_{3,2}(b_2)^{-1} g_3b_3, \dots, \\
	&\qquad \qquad
	\Lambda_{k,1}(b_1)^{-1} \Lambda_{k,2}(b_2)^{-1} \cdots \Lambda_{k,k-1}(b_{k-1})^{-1} g_kb_k)
\end{align*}
for $1 \leq k \leq r$. Here $\Lambda_{k,j}$ is a homomorphism
$B_{\GL(m_{j}+1)} \to H_{\GL(m_k+1)}$ which sends $b\in B_{\GL(m_{j}+1)}$ to 
\[
\text{diag}(\Upsilon(b)^{\bolda{1}{k}{j}},
\Upsilon(b)^{\bolda{2}{k}{j}},\dots,
\Upsilon(b)^{\bolda{m_k+1}{k}{j}}) \in H_{\GL(m_k+1)}, 
\]
where $\Upsilon \colon B_{\GL(m_{j}+1)} \to H_{\GL(m_{j}+1)}$ is
the canonical projection in~\eqref{eq_def_of_Upsilon}, and 
\[
h^{\mathbf {a}} \coloneqq h_1^{\mathbf{a}(1)} h_2^{\mathbf{a}(2)}
\cdots h_{m+1}^{\mathbf{a}(m+1)}
\]
for $h = \text{diag}(h_1,\dots,h_{m+1}) \in H_{\GL(m+1)}$ and 
$\mathbf{a} = (\mathbf{a}(1),\dots,\mathbf{a}(m+1)) \in \Z^{m+1}$.
Now we can describe the flag Bott manifold
$F_r$ as an orbit space as follows:
\begin{proposition}[{\cite[Propositions~2.8 and~2.11]{KLSS}}]	\label{prop_fBT_quotient_description}
	Let $\{F_k\}_{0 \leq k \leq r}$ be a flag Bott tower. 
Suppose that $c_1(\xii{k}{l})$ is determined by a set of integer vectors 
$\{\bolda{l}{k}{j} \in \Z^{m_{j}+1}\}_{1 \leq l \leq m_k+1, 1 \leq j < k \leq r}$
and let $\Phi_k$ be the action determined by these integer vectors.
	Then  
	the flag Bott tower $\{F_k\}_{0 \leq k \leq r}$ is isomorphic to 
	\[
	\{ (\GL(m_1+1) \times \cdots \times \GL(m_k+1))/\Phi_k\}_{0 \leq k \leq r}
	\]
	as flag Bott towers. 
\end{proposition}

A Bott--Samelson variety has a family of complex structures which gives a toric degeneration (see~\cite[\S 3.4]{GrKa94} and~\cite{Pas10}).
Now we study a family of complex
structures on a given flag Bott--Samelson variety. 
Since the simple roots are linearly independent elements in $\mathfrak{h}^{\ast}$,
there exist $q \in \mathbb{Z}_{>0}$ and an injective homomorphism $\lambda
\colon \mathbb{C}^{\ast} \to H$ such that 
\begin{equation}\label{eq_condition_on_Upsilon}
e^\alpha(\lambda(t)) = t^q 
\end{equation}
for all simple roots $\alpha$ and $t \in \Cstar$. 
Here $e^{\alpha} \colon H \to \Cstar$ is a character induced from $\alpha \colon \mathfrak{h} \to \C$. 
For example, when $G = \SL(2k+1)$ and $q =1$, consider the homomorphism $\lambda \colon \Cstar \to H$ defined by
\begin{equation}\label{eq_example_of_lambda}
\lambda \colon t \mapsto \text{diag}(t^{k}, t^{k-1},\dots,t,1,t^{-1},\dots,t^{-k+1},t^{-k}).
\end{equation}
Then this homomorphism satisfies the condition on~\eqref{eq_condition_on_Upsilon}. 
We define $\Upsilon_t \colon B \to B$ by 
\[
\Upsilon_t \colon b \mapsto \lambda(t) b (\lambda(t))^{-1}
\]
for $t \in \Cstar$. It is proved in~\cite[Proposition~3.5]{GrKa94} that $\Upsilon = \lim_{t \to 0} \Upsilon_t$, where $\Upsilon \colon B \to H$ is the homomorphism in~\eqref{eq_def_of_Upsilon}. We put $\Upsilon_0 \coloneqq \Upsilon$. 
\begin{example}
	Suppose that $G = \SL(3)$ and $q = 1$. Considering the homomorphism $\lambda \colon \Cstar \to H$ defined in~\eqref{eq_example_of_lambda}, the homomorphism $\Upsilon_t \colon B \to B$ is given by 
	\[
	\begin{bmatrix}
	b_{11} & b_{12} & b_{13} \\ 0 & b_{22} & b_{23} \\ 0 & 0 & b_{33}
	\end{bmatrix}
	\mapsto 
	\begin{bmatrix}
	b_{11} & tb_{12} & t^2 b_{13} \\ 0 & b_{22} & t b_{23} \\
	0 & 0 & b_{33}
	\end{bmatrix}.
	\]
	Hence we have that $\lim_{t \to 0} \Upsilon_t = \Upsilon$. 
\end{example}

We use the homomorphism $\Upsilon_t \colon B \to B$ to construct a family
of complex structures on the flag Bott--Samelson manifold
$Z_{\mathcal I} = \mathbf P_{\mathcal I}/ B^r$. For $t \in \C$,
we define a right action $\Theta_t$ of $B^r$ on 
$\mathbf{P}_{\mathcal I}$ as
\begin{equation}\label{eq_degeneration_of_cplx_str}
\Theta_t((p_1,\dots,p_r), (b_1,\dots,b_r)) = (p_1b_1, %
\Upsilon_t(b_1)^{-1}p_2b_2, \dots,
\Upsilon_t(b_{r-1})^{-1}p_rb_r)
\end{equation}
for $(p_1,\dots,p_r) \in \mathbf P_{\mathcal I}$ and
$(b_1,\dots,b_r) \in B^r$. Then $\Theta_1$ coincides with the right 
action in~\eqref{eq_Bm-action_defining_fBS}
because $\lambda(1) = e \in H$ and hence $\Upsilon_{1} = \text{Id}_B$. 
Again we consider
the family of orbit spaces 
\[
Z_{\mathcal I}^t \coloneqq \mathbf P_{\mathcal I}/ \Theta_t
\]
for $t \in \C$.
The holomorphic line bundle $\mathcal L_{\mathcal I,\lambda_1,\dots,\lambda_r}^t$
over $Z_{\mathcal I}^t$ can be defined in a way similar to 
$\mathcal L_{\mathcal I, \lambda_1,\dots,\lambda_r}$  
in~\eqref{eq_def_of_L_fBS} for integral weights $\lambda_1,\dots,\lambda_r$.
Set $\mathcal L_{\mathcal I, \lambda}^t\coloneqq 
\mathcal L_{\mathcal I,0,\dots,0,\lambda}^t$ for simplicity. 
\begin{proposition}\label{prop_degeneration_diffeomorphic}
	For a given sequence $\mathcal I = (I_1,\dots,I_r)$, the manifolds
	$Z_{\mathcal I}^t$ are all diffeomorphic for $ t \in \C$.
\end{proposition}
\begin{proof}
	We use the similar argument to the proof of Proposition~3.7 in~\cite{GrKa94}.
	Let $K_{I_j}$ be the maximal compact subgroup of $P_{I_j}$. 
	Let $T$ be the 
	maximal compact torus in $G$, i.e., $T = (S^1)^n$.
	Recall that
	$K_{I_j} \cap B = T$.
	Define a right action of $T^{(r)}\coloneqq 
	\underbrace{T \times T \times
		\cdots \times T}_r$ on $K_{\mathcal I}\coloneqq K_{I_1} \times \cdots \times
	K_{I_r}$ as
	\begin{equation}\label{eq_Km-action_defining_fBS}
	(g_1,\dots,g_r) \cdot (a_1,\dots,a_r) = (g_1a_1, a_1^{-1}g_2a_2,
	\dots, a_{r-1}^{-1}g_r a_r).
	\end{equation}
	Let $X_{\mathcal I}$ be the orbit space 
	\begin{equation}\label{eq_def_of_XI}
	X_{\mathcal I} \coloneqq (K_{I_1} \times
	\cdots \times K_{I_r})/(T \times \cdots \times T).
	\end{equation}
	The inclusion map 
	\[
	K_{\mathcal I} = K_{I_1} \times \cdots \times K_{I_r} 
	\hookrightarrow \mathbf P_{\mathcal I} = P_{I_1} \times 
	\cdots \times P_{I_r}
	\]
	is $T^{(r)}$-equivariant with respect to 
	the $T^{(r)}$-action of~\eqref{eq_Km-action_defining_fBS} on $K_{\mathcal I}$
	and the restricted $T^{(r)}$-action of~\eqref{eq_degeneration_of_cplx_str} on $\mathbf{P}_{\mathcal I}$
	via the inclusion $T^{(r)} \hookrightarrow
	B^r$ because $\Upsilon_t(a) = a$ for all $a \in T$. 
	Therefore we get a map 
	\begin{equation}\label{eq_def_of_fI_diffeomorphism}
	f_{\mathcal I}^t  \colon X_{\mathcal I} \to Z_{\mathcal I}^t. 
	\end{equation}
	Since, for all $k$, the
	inclusion $K_{I_k} \hookrightarrow P_{I_k}$ induces a diffeomorphism
	$K_{I_k}/T \cong P_{I_k}/B$, the map  $f_{\mathcal I}^t$  is  a diffeomorphism. 
\end{proof}

The manifold $Z_{\mathcal I}^t$ 
has a fibration structure, similar to a flag Bott--Samelson manifold
in~\eqref{eq_bdle_str_on_fBS}:
\begin{equation}\label{eq_fibra_of_fBS_t}
P_{I_r}/B \hookrightarrow Z_{\mathcal I}^t
\stackrel{\pi}{\longrightarrow} Z_{\mathcal I'}^t,
\end{equation}
where $\mathcal I' = (I_1,\dots,I_{r-1})$ is the subsequence of
$\mathcal I$ and $\pi$ is the first $r-1$ coordinates projection
for all $t \in \mathbb{C}$.

	Let $\mathcal I = (I_1,\dots,I_{r-1},I_r)$ and  $\mathcal I' = (I_1,\dots,I_{r-1})$. 
	We note that the orbit space $X_{\mathcal I}$ has a bundle structure.
	\[
	\begin{tikzcd}
	X_{\mathcal I} = \mathbf{P}_{\mathcal I'} \times_{T} (K_{I_r}/T) \dar & K_{I_r}/T \arrow[l, hook'] \\
	X_{\mathcal I'}
	\end{tikzcd}
	\]
	Because the structure group $T$ of this bundle is an abelian group, the map $f_k$ inducing the flag bundle $X_{\mathcal I} \to X_{\mathcal I'}$ from the universal flag bundle factors through $BT$.
	\[
	\begin{tikzcd}[column sep = 0cm]
	K_{I_r}/T \arrow[rr, equal] \arrow[d, hook] && K_{I_r}/T  \arrow[d, hook]\\
	X_{\mathcal I} \arrow[rr] \arrow[d]  && BT \arrow[d] \\
	X_{\mathcal I'} \arrow[rr, "f_{k}"] \arrow[dr] && BK_{I_r} \\
	& BT \arrow[ur]
	\end{tikzcd}
	\]
	Continuing this procedure, we obtain the following corollary.
	\begin{corollary}\label{cor_fBS_are_fBT_general}
	The manifold $X_{\mathcal I}$ is an $r$-stage flag Bott tower of general Lie type associated to $\{(K_{I_j}, T)\}_{0 \leq j \leq r}$, and so are $Z_{\mathcal I}^t$ for all $t \in \C$ \textup{(}see Remark~\ref{rmk_fBT_general_type} for the definition of flag Bott towers of general Lie type\textup{)}. 
	\end{corollary}

For the remaining part of this section, we consider the case when the Levi subgroup $L_{I_k}$ of the parabolic subgroup $P_{I_k}$ has Lie type $A$, that is, the flag Bott tower $X_{\mathcal I}$ is a flag Bott manifold whose fibers are all full flag manifolds of Lie type $A$. Moreover, we describe the line bundles appearing in the construction explicitly (see Theorem~\ref{thm_fBS_is_fBT}).
We can always take an enumeration
$I_k = \{u_{k,1},\dots,u_{k,m_k}\}$ so that
\begin{equation}\label{eq_Cartan_integer_A}
\langle \alpha_{u_{k,s}}, \alpha_{u_{k,t}}^{\vee} \rangle
= \begin{cases}
2 & \text{ if } s = t, \\
-1 & \text{ if } s - t = \pm 1,\\
0 & \text{ otherwise}.
\end{cases}
\end{equation}
\begin{proposition}\label{prop_flag_bundle_str}
	Let $Z_{\mathcal I}$ be a flag Bott--Samelson manifold.
	Let $\mathcal I' = (I_1,\dots,I_{r-1})$
	be the subsequence of $\mathcal I$. 
	Assume that 
	the Levi subgroup $L_{I_k}$ of the parabolic subgroup $P_{I_k}$ has Lie type $A_{m_k}$ for all 
	$1 \leq k \leq r$. 
	Then the manifold $Z_{\mathcal I}^0$ is diffeomorphic to the induced
	flag bundle over $Z_{\mathcal I'}^0$:
	\[
	Z_{\mathcal I}^0 \cong \flag(
	\mathcal{L}_{\mathcal I',\chi_1}^0
	\oplus \cdots \oplus \mathcal{L}_{\mathcal I', \chi_{m_r}}^0
	\oplus \underline{\C}),
	\]
	where $\chi_{j} = \alpha_{u_{r, j}} + \cdots + \alpha_{u_{r, m_r}} \in
	\mathfrak{h}^{\ast}$ for $1\leq j \leq m_r$, 
	$\mathcal L^0_{\mathcal I', \chi} 
	= \mathcal L^0_{\mathcal I', 0,\dots,0,\chi}$,
	and $\underline{\C}$ is the trivial line bundle. 
\end{proposition}
Before proving the proposition, we observe the following.
Suppose that the Levi subgroup $L_I$ of the parabolic subgroup $P_I$ for a subset $I\subset[n]$
has Lie type $A_m$. Then we can label the elements of $I$ as
$u_1,\dots,u_m$ which satisfy the relation~\eqref{eq_Cartan_integer_A}.
Also we have 
the group homomorphism 
$F \colon \SL(m+1) \to L_I \hookrightarrow P_{I}$.
Then the map $F$ induces the homomorphism 
$F_{\ast} \colon \mathfrak{h}_{\SL(m+1)} \to \mathfrak{h}$.
We label the coroots of $\SL(m+1)$ as $\beta_1^{\vee},\beta_2^{\vee},\dots,\beta_m^{\vee}$ so that $F_{\ast}$ sends $\beta_l^{\vee}$ to $\alpha_{u_l}^{\vee}$ for $1 \leq l \leq m$.
Then we have that 
\[
\langle F^{\ast}\lambda, \beta_{l}^{\vee} \rangle
= \langle \lambda, F_{\ast} \beta_{l}^{\vee} \rangle
= \langle \lambda, \alpha_{u_l}^{\vee} \rangle
\]
for a weight $\lambda \in \mathfrak{h}^{\ast}$ and $1\leq l \leq m$.
Here, we note that $F^{\ast} \lambda = \lambda \circ F $ for $\lambda \in \mathfrak{h}^{\ast}$.
Let $\varpi_1, \varpi_2,
\dots,\varpi_m \in \mathfrak{h}_{\SL(m+1)}^{\ast}$ be 
the fundamental weights. Then
the pullback $F^{\ast} \lambda$ is given by
\begin{equation}\label{eq_pull_back_lam}
F^{\ast} \lambda = 
\sum_{l=1}^{m} \langle \lambda, 
\alpha_{u_l}^{\vee} \rangle \varpi_l \in \mathfrak{h}_{\SL(m+1)}^{\ast}.
\end{equation}
\begin{proof}[Proof of Proposition~\ref{prop_flag_bundle_str}]
We write $I = I_r$, $m = m_r$, and $u_j = u_{r, j}$ for $1 \leq j \leq m$.
	Note that we have $P_I = L_I U_I$ (see \S\ref{subsec_def_of_fBS}). Since we have an isomorphism of varieties
	\[
	P_I/B = (L_I U_I)/B = L_I/(B \cap L_I) = L_I/B_I,
	\]
	we get a diffeomorphism
	\[
	F_1 \colon \SL(m+1)/B_{\SL(m+1)} \to P_I/B.
	\]
	Moreover, the map which sends an element $g$ in $\SL(m+1)$
	to a full flag $(V_1 \subsetneq V_2 \subsetneq \cdots \subsetneq V_m)$,
	where $V_l = \langle c_1,\dots,c_l \rangle$ and $c_l$ is the $l$-th column vector of $g$, descends to a diffeomorphism 
\[F_2 \colon \SL(m+1)/B_{\SL(m+1)} \to \flag(m+1).
\]

	The map $F_2$ is equivariant with respect to the following actions 
	of the torus $H_{\SL(m+1)}$:
	each element $h = \text{diag}(h_1, h_2,\dots,h_{m+1}) 
	\in H_{\SL(m+1)}$ acts on 
	$\SL(m+1)/B_{\SL(m+1)}$ by the left
	multiplication, and on $\flag(m+1)$ 
	as the induced action from the representation space 
	$\C^{m+1}$ with weights 
	\begin{equation}\label{eq_weights_of_H_on_C_m+1}
	(\varpi_1, -\varpi_1+\varpi_2,
	\dots, -\varpi_{m-1}+\varpi_m, -\varpi_m),
	\end{equation}
	namely $h \cdot v = (h_1v_1,h_2v_2,\dots,h_{m+1}v_{m+1})$ 
	for $v = (v_1,\dots,v_{m+1}) \in \C^{m+1}$.
	On the other hand, the map $F_1$ is equivariant with respect to the left multiplication actions of
	$H_{\SL(m+1)}$ and of $H$
	via the homomorphism $H_{\SL(m+1)} \to H$ given by the map $F$.
	
	By the relation~\eqref{eq_pull_back_lam}
	between weights in $\mathfrak{h}^{\ast}$ and $\mathfrak{h}_{\SL(m+1)}^{\ast}$, we have the following:
	\begin{align*}
		F^{\ast}(\chi_{j}) 
		&= F^{\ast}(\alpha_{u_j} + \cdots + \alpha_{u_m}) \\
		&= \sum_{l=1}^m \langle
		\alpha_{u_j} + \cdots + \alpha_{u_m}, \alpha_{u_l}^{\vee}
		\rangle \varpi_l \\
		&= -\varpi_{j-1} + \varpi_{j} + \varpi_m,
	\end{align*}
	where $\varpi_0 = 0$ for $1\leq j\leq m$. 
	Here the third equality follows by considering the Cartan matrix of $\SL(m+1)$.
	The $H_{\SL(m+1)}$-representation on $\C^{m+1}$ with weights~\eqref{eq_weights_of_H_on_C_m+1} becomes an $H$-representation on $\mathbb{C}^{m+1}$ with weights
	\[
	(\chi_1 - \chi', \chi_2 - \chi',\dots,
	\chi_m - \chi', -\chi'),
	\]
	where $\chi'$ is a weight which maps to $\varpi_m$ under the map $F^{\ast}$
	such that $F_2 \circ F_1 ^{-1}$ is equivariant with respect to the actions of elements in $H \setminus F(H_{\SL(m+1)})$.
	This proves that $F_2 \circ F_1^{-1}$ is 
	a left $H$-equivariant diffeomorphism
	\[
	F_2 \circ F_1^{-1} \colon P_I/B \to \flag(\C_{\chi_1 - \chi'} \oplus \cdots 
	\oplus \C_{\chi_{m} - \chi'} \oplus \C_{-\chi'}).
	\]
	
	We notice that the construction of twisted product is functorial, i.e., for a topological group $G$ and a right $G$-space $X$, if $f \colon Y \to Y'$ is an equivariant map of left $G$-spaces then we have the induced map $X \times_G Y \to X \times_G Y'$, see, for example,~\cite[\S II.2]{Bredon72Introduction}. 
	Since the unipotent part of $B$ acts trivially on $P_I/B$ and $\flag(\C_{\chi_1 - \chi'} \oplus \cdots 
	\oplus \C_{\chi_{m} - \chi'} \oplus \C_{-\chi'})$, the left $H$-equivariant diffeomorphism
	$F_2 \circ F_1^{-1}$ 
	induces a diffeomorphism 
	\[
	\mathbf{P}_{\mathcal I}/\Theta_0
	\cong\flag(\mathcal L^0_{\mathcal I', \chi_1 - \chi'} 
	\oplus \cdots \oplus \mathcal L^0_{\mathcal I', \chi_m - \chi'}
	\oplus \mathcal L_{\mathcal I', -\chi'}^0).
	\]
	Moreover we have that 
	\begin{align*}
	&\flag(\mathcal L^0_{\mathcal I', \chi_1 - \chi'} 
		\oplus \cdots \oplus \mathcal L^0_{\mathcal I', \chi_m - \chi'}
		\oplus \mathcal L_{\mathcal I', -\chi'}^0)  \\
	&\qquad =
	\flag((\mathcal L^0_{\mathcal I', \chi_1} \oplus
		\cdots \oplus \mathcal L^0_{\mathcal I', \chi_m} \oplus \underline{\C})
		\otimes \mathcal L^0_{\mathcal I', -\chi'}).
	\end{align*}
	Then by Lemma~\ref{lemma_tensor_preserves_fBT_str}, we 
	are done. 
\end{proof}

By Proposition~\ref{prop_flag_bundle_str},
we can conclude that $Z_{\mathcal I}^0$ is an $r$-stage flag
Bott manifold. 
For given integral weights $\lambda_1,\dots,\lambda_r$,
consider the line bundle 
$\mathcal{L}_{\mathcal I, \lambda_1,\dots,\lambda_r}^0$
over a flag Bott manifold $Z_{\mathcal I}^0$. 
By~\eqref{eq_Pic_gp_of_Fj} there is a set of integer vectors
$\{\mathbf{a}_k \in \Z^{m_k+1}\}_{1 \leq k \leq r}$ determined by $c_1(\mathcal{L}_{\mathcal I, \lambda_1,\dots,\lambda_r}^0)$. Indeed, we have
\[
\psi(\mathbf{a}_1,\dots,\mathbf{a}_r) \cong \mathcal{L}_{\mathcal I, \lambda_1,\dots,\lambda_r}^0.
\]
The following proposition computes these integer vectors in terms of 
integral weights $\lambda_1,\dots,\lambda_r$ and 
a sequence $\mathcal I$ of subsets of $[n]$.
\begin{proposition}\label{prop_deg_of_line_bdle}
	Let $\mathcal I =(I_1,\dots,I_r)$ be a sequence of subsets of $[n]$.
		Assume that 
	the Levi subgroup $L_{I_k}$ of the parabolic subgroup $P_{I_k}$ has Lie type $A_{m_k}$ for all 
	$1 \leq k \leq r$. 
	For given integral weights $\lambda_1,\dots,\lambda_r \in \Z \varpi_1 + \cdots + \Z \varpi_n$, the first Chern class of the line bundle 
	$\mathcal{L} = \mathcal{L}_{\mathcal I, \lambda_1,\dots,\lambda_r}^0$ is given by integer vectors
		$\mathbf{a}_{k} 
	= (\mathbf{a}_{k}(1),\dots,\mathbf{a}_{k}(m_{k}+1)) \in \Z^{m_{k}+1} $ for ${1\leq k \leq r}$, where 
		\begin{align*}
	&\mathbf{a}_{k}(l) = 
	\langle \lambda_k+\cdots+\lambda_r,
	\alpha_{u_{k,l}}^{\vee} + \cdots+\alpha_{u_{k,m_k}}^{\vee}
	\rangle \quad
	\text{ for } 1 \leq l \leq m_k, \\
	& \mathbf a_{k}(m_k+1) = 0.
	\end{align*}	
	Here, we take an enumeration $I_k = \{u_{k,1},\dots,u_{k,m_k}\}$ which satisfies~\eqref{eq_Cartan_integer_A}.
	Indeed, $\mathcal L$ is	isomorphic to 	the line bundle $\psi(\mathbf{a}_1,\dots,\mathbf{a}_r)$.
\end{proposition}
\begin{proof}
	Since the Levi subgroup $L_{I_k}$ of $P_{I_k}$ is Lie type $A_{m_k}$,
	we have a Lie group homomorphism 
	$F_k \colon \SL(m_k+1) \to P_{I_k}$. 
	For each $1 \leq k \leq r$,
	consider the homomorphism
	$	\psi_k \colon \SL(m_k+1) \to P_{I_1} \times \cdots \times P_{I_r}$
	defined as 
	\[
	p \mapsto (e,\dots,e,\underbrace{F_k(p)}_{k\text{th}},e,\dots,e)
	\]
	and consider
	\begin{equation}\label{eq_map_B}
	\varphi_k \colon B_{\SL(m_k+1)}\to \underbrace{B \times \cdots \times B}_r = B^r
	\end{equation}
	which sends $b$ to 
	\[
	(e,\dots,e,\underbrace{F_k(b)}_{k\text{th}},
	\underbrace{F_k(h)}_{(k+1)\text{th}},\dots,
	\underbrace{F_k(h)}_{r\text{th}}),
	\]
	where $h = \Upsilon(b)$.
	Then the map $\psi_k$ is $\varphi_k$-equivariant, namely,
	for $b \in B_{\SL(m_k+1)}$ and $g \in \SL(m_k+1)$ we have that
	\begin{align*}
		\Theta_0(&\psi_k(g), \varphi_k(b)) \\
		&= \Theta_0((e,\dots,e,F_k(g),e,\dots,e), 
		(e,\dots,e,F_k(b),F_k(h),\dots,F_k(h))) \\
		&= (e,\dots,e, F_k(g)F_k(b), 
		\Upsilon(F_k(b))^{-1}F_k(h),e,\dots,e) \\
		&= (e,\dots,e, F_k(gb),e,e,\dots,e)\\
		&= \psi_k(g b).
	\end{align*}
	Here the third equality comes from the fact that $F_k$ is a homomorphism 
	and $\Upsilon(F_k(b)) = F_k(\Upsilon(b))$.

	Under the map~\eqref{eq_map_B} the weight $(\lambda_1,\dots,\lambda_r)$ of $H^r$ pulls back to the weight 
	\begin{equation}\label{eq_weight_pull_back}
	\sum_{l=1}^{m_k} \langle \lambda_k+\cdots+\lambda_r,
	\alpha_{u_{k,l}}^{\vee} \rangle
	\varpi_l \in \mathfrak{h}_{\SL(m_k+1)}^{\ast}
	\end{equation}
	by~\eqref{eq_pull_back_lam}.
	The integer vector $\mathbf a_k \in \Z^{m_k+1}$ is completely determined by the weight in~\eqref{eq_weight_pull_back} because of the construction of a flag Bott manifold (see~\cite[\S 2.2]{KLSS}).	Indeed, the integer vector $\mathbf{a}_k \in \Z^{m_k+1}$ should satisfy the equality:
	\begin{equation}\label{eq_prop_4.7_1}
	\sum_{l=1}^{m_k} \langle \lambda_k+\cdots+\lambda_r,
	\alpha_{u_{k,l}}^{\vee} \rangle
	\varpi_l 
	= \sum_{l=1}^{m_k +1}\mathbf{a}_k(l) \varepsilon_l,
\end{equation}
where $\varepsilon_i \in \mathfrak{h}_{\SL(m_k+1)}^{\ast}$ sends $\text{diag}(h_1,\dots,h_{m_k+1})$ in $\mathfrak{h}_{\SL(m_k+1)}$ to $h_i$.
	Using the identification $\varpi_l = \varepsilon_1 + \cdots + \varepsilon_l$, 
	we have that
	\begin{equation}\label{eq_prop_4.7_2}
	\begin{split}
		\sum_{l=1}^{m_k}& \langle \lambda_k+\cdots+\lambda_r,
		\alpha_{u_{k,l}}^{\vee} \rangle
		\varpi_l \\
		&= 	\langle \lambda_k + \cdots + \lambda_r, \alpha_{u_{k,1}}^{\vee} 
		\rangle \varepsilon_1 
		+ \langle \lambda_k + \cdots + \lambda_r, \alpha_{u_{k,2}}^{\vee} 
		\rangle (\varepsilon_1 + \varepsilon_2) \\
		& \qquad + \cdots +
		\langle \lambda_k + \cdots + \lambda_r, \alpha_{u_{k,m_k}}^{\vee} 
		\rangle (\varepsilon_1 + \cdots + \varepsilon_{m_k}) \\
		&= \langle \lambda_k + \cdots + \lambda_r, 
		\alpha_{u_{k,1}}^{\vee} + \alpha_{u_{k,2}}^{\vee} 
		+ \cdots + \alpha_{u_{k,m_k}}^{\vee} \rangle \varepsilon_1 \\
		& \qquad + 
		\langle \lambda_k + \cdots + \lambda_r, 
		\alpha_{u_{k,2}}^{\vee} 
		+ \cdots + \alpha_{u_{k,m_k}}^{\vee} \rangle \varepsilon_2  \\
		& \qquad \qquad + \cdots + 
		\langle \lambda_k+\cdots +\lambda_r, \alpha_{u_{k,m_k}}^{\vee} \rangle 
		\varepsilon_{m_k}.
	\end{split}
	\end{equation}
	Comparing equations~\eqref{eq_prop_4.7_1} and~\eqref{eq_prop_4.7_2}, we obtain the assertion of the proposition.
\end{proof}

By combining Propositions~\ref{prop_flag_bundle_str} 
and~\ref{prop_deg_of_line_bdle}, we can prove the following
theorem:
\begin{theorem}\label{thm_fBS_is_fBT}
	Suppose that $\mathcal I = (I_1,\dots,I_r)$ is a sequence of subsets of $[n]$ such that the Levi subgroup $L_{I_k}$ of the parabolic subgroup $P_{I_k}$ has Lie type $A_{m_k}$ for all $1 \leq k \leq r$. Take an enumeration $I_k = \{u_{k,1},\dots,u_{k,m_k}\}$ which satisfies~\eqref{eq_Cartan_integer_A}.
	Then the manifold $Z_{\mathcal I}^0$ is an $r$-stage flag Bott manifold
	which is determined by
$
	\{\bolda{l}{k}{j}
	=  (\bolda{l}{k}{j}(1), \bolda{l}{k}{j}(2),
	\dots, \bolda{l}{k}{j}(m_{j}+1))\}_{1 \leq l \leq m_k+1, 1 \leq j < k \leq r}$ in the sense of Proposition~\ref{prop_fBT_quotient_description},
	where $\bolda{l}{k}{j}(p)$ is
\[	\langle
	\alpha_{u_{k,l}} + \cdots + \alpha_{u_{k,m_k}},
	\alpha_{u_{j,p}}^{\vee} + \cdots + 
	\alpha_{u_{j,m_j}}^{\vee}\rangle
\]	 if $1\leq l\leq m_k$ and $1\leq p \leq m_j$, and $0$ otherwise.

\end{theorem}
\begin{proof}
	Consider the subsequence $\mathcal I_k \coloneqq (I_1,\dots,I_k)$ of 
	the sequence $\mathcal I$ for all $1\leq k \leq r$. 
	Recall from Proposition~\ref{prop_flag_bundle_str} that 
	the flag Bott manifold $Z_{\mathcal I_k}^0$ is the induced
	flag bundle over~$Z_{\mathcal I_{k-1}}^0$:
	\[
	Z_{\mathcal I_k}^0 = \flag(\mathcal{L}_{\mathcal I_{k-1}, \chi_{1}}^0 \oplus
	\cdots \oplus \mathcal{L}_{\mathcal I_{k-1},\chi_{m_k}}^0 \oplus
	\underline{\C}),
	\]
	where
	$\chi_l = \alpha_{u_{k,l}} + \cdots + 
	\alpha_{u_{k,m_k}}$ for $1 \leq l \leq m_k$.
	By Proposition~\ref{prop_deg_of_line_bdle}, 
	the integer vectors $\{\bolda{l}{k}{j} \in \Z^{m_{j}+1}\}_{1 \leq j \leq k-1}$
	which define the line bundle $\mathcal{L}_{\mathcal I_{k-1}, \chi_l}^0$ are
	given by
	\begin{align*}
		&\bolda{l}{k}{j}(p) \\
		&\quad= \langle \chi_l, \alpha_{u_{j,p}}^{\vee} 
		+ \cdots + \alpha_{u_{j,m_j}}^{\vee} \rangle 
		\quad (\text{by Proposition~\ref{prop_deg_of_line_bdle}})
		\\
		&\quad= \langle \alpha_{u_{k,l}} + \cdots + \alpha_{u_{k,m_k}}, 
		\alpha_{u_{j,p}}^{\vee} 	+ \cdots + \alpha_{u_{j,m_j}}^{\vee} \rangle
		\quad (\text{by the definition of }\chi_l)
	\end{align*}
	for $1\leq l \leq m_k$ and 	$1 \leq p \leq m_{j}$. 
	Moreover we have $\bolda{l}{k}{j}(p)=0$ 
	if $l = m_k+1$ or $p = m_{j}+1$ 	by  Proposition~\ref{prop_deg_of_line_bdle}.
	Hence the result follows. 
\end{proof}
\begin{example}
	Let $G = \SL(4)$. Consider the sequence 
	$\mathcal I = (\{1,2\}, \{ 1,2 \} )$.
	Hence $u_{1,1} = 1, u_{1,2} = 2, u_{2,1} = 1, u_{2,2} =2$. 
	The manifold $Z_{\mathcal I}^0$ is a $2$-stage flag 
	Bott manifold with $F_2 = \flag(\xii{1}{2} \oplus \xii{2}{2} \oplus
	\underline{\C})$,
	where line bundles $\xii{1}{2}$ and $\xii{2}{2}$ are determined by
	the following integer vectors:
	{{\reqnomode
			\begin{align*}
				\bolda{1}{2}{1} 
				&= (\langle 
				\alpha_1 + \alpha_2, \alpha_1^{\vee} + \alpha_2^{\vee}
				\rangle, 
				\langle \alpha_1 + \alpha_2, \alpha_2^{\vee} \rangle, 0)
				= (2,1,0), \\
				\bolda{2}{2}{1} 
				&= (\langle 
				\alpha_2, \alpha_1^{\vee} + \alpha_2 ^{\vee} \rangle,
				\langle \alpha_2, \alpha_2^{\vee} \rangle, 0)
				= (1,2,0).
	\end{align*}}}
\end{example}
\begin{remark}
	Suppose that the flag Bott--Samelson variety 
	$Z_{\mathcal I}$ is a Bott--Samelson variety, i.e., 
	$m_1=\cdots=m_r=1$.
	Then integer vectors $\{\bolda{l}{k}{j} \in \Z^2\}_{l \in [2], 1 \leq j < k \leq r}$ 
	determining the flag Bott tower $Z_{\mathcal I}^0$ is 
	\[
	\bolda{l}{k}{j} 
	= \begin{cases}
	(\langle \al_{u_{k,1}}, \al_{u_{j,1}}^{\vee} \rangle, 0)
	& \text{ if } l = 1, \\
	(0,0) & \text{ if } l =2
	\end{cases}
	\]
	by Theorem~\ref{thm_fBS_is_fBT}. 
	This computation of $\bolda{1}{k}{j}(1)$ 
	for $1 \leq j < k \leq r$
	coincides with the known result in~\cite[\S 3.7]{GrKa94}.
\end{remark}

\section{Torus actions and Duistermaat--Heckman measure}
\label{sec_torus_actions_on_fBS}
Let $\mathcal I= (I_1,\dots,I_r)$ be a sequence of subsets of $[n]$ such that $|I_k| = m_k$. 
In this section we study torus actions on the manifold
$Z_{\mathcal I}^0$.
We define a torus invariant closed $2$-form induced from a given complex
line bundle, and we consider the Duistermaat--Heckman measure of the flag Bott--Samelson manifold using a Bott--Samelson variety $Z_{\mathbf i}$ admitting 
the birational morphism $\eta_{{\bf i}, \mathcal I} \colon Z_{\mathbf i} \to Z_{\mathcal I}$ (see Theorem~\ref{thm_DH_measure_of_fBS}). 

We first study torus actions on $Z_{\mathcal I}^0$.
Let $T$ be the maximal compact torus of $G$ contained in $H$.
Define an action of $T^{(r)}$ on $Z_{\mathcal I}^0$ as
\begin{equation}\label{eq_def_of_T_on_fBS}
\begin{split}
(s_1,\dots,s_r) \cdot [p_1,\dots,p_r] 
&= [s_1p_1, s_1^{-1}s_2p_2,\dots,
s_{r-1}^{-1}s_rp_r]  \\
&= [s_1p_1s_1^{-1},\dots,s_rp_rs_r^{-1}].
\end{split}
\end{equation}
This action is smooth but not effective.
We now find the subtorus which acts trivially on $Z_{\mathcal I}^0$. 
Define a subtorus $T_{I} \subset T$ for a subset $I \subset [n]$ as
\[
T_I \coloneqq \{s \in T \mid 
\alpha_i(s) = 1 \text{ for all } i \in I \}^0
\]
which is similar to~\eqref{eq_def_of_HI}. 
Here, we consider a simple root $\alpha \in \chi(H)$ as a homomorphism $T \to S^1$. For a given sequence $\mathcal I = (I_1,\dots,I_r)$ of subsets of~$[n]$, we define the subtorus $T_{\mathcal I}$ of $T^{(r)}$ as
\[
T_{\mathcal I}\coloneqq T_{I_1} \times \cdots \times T_{I_r}.
\]
Similarly, we set $T_{\mathbf i} \coloneqq T_{\{i_1\}} \times \cdots \times T_{\{i_r\}}$ for a sequence $\mathbf i = (i_1,\dots,i_r) \in [n]^r$. Then the following proposition comes from~\eqref{eq_def_of_T_on_fBS}.
\begin{proposition}\label{prop_torus_on_ZI0}
	The torus $T_{\mathcal I}$ acts trivially on $Z_{\mathcal I}^0$.
\end{proposition}
By Proposition~\ref{prop_torus_on_ZI0}, we have the torus action on $Z_{\mathcal I}^0$: 
\begin{equation}\label{eq_torus_action_on_ZI0}
T^{(r)}/T_{\mathcal I} \curvearrowright Z_{\mathcal I}^0.
\end{equation}
Note that $T^{(r)}/T_{\mathcal I} \cong (S^1)^{m_1+\cdots +m_r}$. 

Suppose that ${\bf i} = (i_{k, l})_{1 \le k \le r, 1 \le l \le N_k} \in [n]^{N_1 + \cdots + N_r}$ is a sequence such that
$(i_{k,1},\dots,i_{k,N_k})$ is a reduced word for the longest element in $W_{I_k}$ for $1\leq k \leq r$. From now on, we ignore the complex structure on the flag Bott--Samelson manifold $Z_{\mathcal I}$ and regard it as a smooth manifold. Therefore we can identify $Z_{\mathcal I}$ with $Z_{\mathcal I}^0$ and $Z_{\mathbf i}$ with $Z_{\mathbf i}^0$ by Proposition~\ref{prop_degeneration_diffeomorphic}.
Using the observation~\eqref{eq_torus_action_on_ZI0}, we have the torus action on the Bott--Samelson manifold~$Z_{\mathbf i}$:
\[
(S^1)^{N} \cong T^{(N)}/T_{\mathbf i} \curvearrowright Z_{\mathbf i},
\]
where $N\coloneqq N_1 + N_2 + \cdots + N_r$. 
We denote $\widetilde{\mathbf T} \coloneqq (T^{(N)})/T_{\mathbf i}$ and
$\mathbf T \coloneqq (T^{(r)})/T_{\mathcal I}$ for simplicity. 
\begin{lemma}\label{lemma_homo_A_between_tori}
	There is a homomorphism 
	$A \colon \mathbf T \to \widetilde{\mathbf T}$ 
	such that the map $\eta_{{\bf i}, \mathcal I} \colon Z_{\mathbf i}\to Z_{\mathcal I}$
	is equivariant with respect to the action of 
	$\mathbf T$, i.e.
	\[
	\eta_{{\bf i}, \mathcal I}(A(t) \cdot x) = t \cdot \eta_{{\bf i}, \mathcal I}(x)
	\] 
	for any $t \in \mathbf T$ and 
	$x \in Z_{\mathbf i}$. 
\end{lemma}
\begin{proof}
	Define an inclusion map $\iota \colon T^{(r)} \hookrightarrow
	T^{(N)}$ as
	\[
	(a_1,\dots,a_r) \stackrel{\iota}{\mapsto}
	(\underbrace{a_1,\dots,a_1}_{N_1}, \dots,
	\underbrace{a_k,\dots,a_k}_{N_k},\dots,
	\underbrace{a_r,\dots,a_r}_{N_r}).
	\]
	Then we have the action $T^{(r)} \curvearrowright Z_{\mathbf i}$ via the inclusion $\iota$ and 
	the map $\eta_{{\bf i}, \mathcal I} \colon Z_{\mathbf i}\to Z_{\mathcal I}$ is 
	equivariant with respect to the action of $T^{(r)}$
	by the definition of torus action in~\eqref{eq_def_of_T_on_fBS}.
	
	We claim that $\iota(T_{\mathcal I}) \subset T_{\mathbf i}$.
	For an element $(a_1,\dots,a_r) \in T^{(r)}$, we have that
	\begin{align*}
	(a_1,\dots,a_r) \in T_{\mathcal I} & 
	\Longleftrightarrow a_k
	\in T_{I_k} \quad \text{ for all } 1\leq k \leq r.
	\end{align*}
	Hence we have $a_k \in T_{i_{k,1}},\dots, a_k \in T_{i_{k,N_k}}$
	since $\{i_{k,1},\dots,i_{k,N_k}\}= I_k$
	for all $1 \leq k \leq r$. 
	This gives that $\iota(T_{\mathcal I}) \subset T_{\mathbf i}$ as claimed. We thus have the homomorphism
	\begin{equation}\label{eq_def_of_A}
	A \colon T^{(r)}/T_{\mathcal I} \to T^{(N)}/T_{\mathbf i}
	\end{equation}
	induced from the inclusion $\iota$. 
	Moreover the projection map $Z_{\mathbf i} \to Z_{\mathcal I}$ is equivariant
	with respect to the action of $\mathbf T$ because of the $T^{(r)}$-equivariance of the projection. 
\end{proof}
We set $A_k \colon T/T_{I_k} \to T^{(N_k)}/T_{(i_{k,1},\dots,i_{k,N_k})}$ for $1 \leq k \leq r$. By the definition of $T_{(i_{k,1},\dots,i_{k,N_k})}$, the torus $T^{(N_k)}/T_{(i_{k,1},\dots,i_{k,N_k})}$ has dimension $N_k$. Suppose that $\{f_{k,1},\dots,f_{k,N_k}\}$ is the standard basis of $\Lie((S^1)^{N_k})^{\ast} \cong \mathbb{R}^{N_k}$. Then it is known from~\cite[\S 3.7]{GrKa94} that the pullback of $f_{k,l}$ is $\alpha_{i_{k,l}}$ for $1 \leq l \leq N_k$. 
Since the homomorphism $A$ can be identified with $A_1 \times \cdots \times A_r$, the Lie algebra homomorphism $(dA)^{\ast} \colon \mathbb{R}^N \to \mathbb{R}^{m_1+\cdots+m_r}$ maps $f_{k,l}$ to $\alpha_{i_{k,l}}$ for $1 \leq k \leq r$ and $1 \leq l \leq N_k$. 

\begin{example}\label{example_projection_A}
	Recall from Example~\ref{ex_fBS_and_BS} that we have a morphism
	$\eta_{(1,2,1,3), \mathcal I}$ from $Z_{(1,2,1,3)}$ to $Z_{\mathcal I}$, where
	$\mathcal I = (\{1,2\}, \{3\})$. 
	Suppose that $A \colon T^{(2)}/T_{\mathcal I} \to T^{(4)}/T_{(1,2,1,3)}$ is the homomorphism in Lemma~\ref{lemma_homo_A_between_tori}. Then the Lie algebra homomorphism $(dA)^{\ast} \colon \mathbb{R}^4 \to \mathbb{R}^3$ is defined using
	the integer matrix:
	\[
	\begin{bmatrix}
	1 & 0 & 1 & 0 \\
	0 & 1 & 0 & 0 \\
	0 & 0 & 0 & 1
	\end{bmatrix}.
	\]
\end{example}

We now consider Duistermaat--Heckman measures corresponding
to flag Bott--Samelson manifolds. 
 We recall definitions from~\cite{Audin04torus}.
Suppose that $M$ is an oriented, compact manifold of real dimension $2d$
with an action of a compact torus $T$. 
Let $\omega$ be a presymplectic form, i.e. a $T$-invariant closed not necessarily non-degenerate $2$-form. 
Then we call the manifold $(M, \omega, T)$ \defi{presymplectic $T$-manifold}.
A \defi{moment map} on $(M, \omega, T)$ is defined to be a map 
$\Phi \colon M \to \Lie(T)^{\ast}$ such that 
\[
\langle d \Phi, \xi \rangle 
= - \iota(\xi_M) \omega \quad \text{ for all } \xi \in \Lie(T),
\]
where $\xi_M$ is the vector field on $M$ which generates the action
of the one-parameter subgroup $\{ \exp(t \xi) \mid t \in \mathbb{R}\}$ of $T$. 
Note that the Liouville measure on $M$ is defined to be 
$\int_A \omega^d/d!$ for an open subset $A \subset M$, and its push-forward
$\Phi_{\ast} \omega^d/d!$ is called the \defi{Duistermaat--Heckman measure} in $\Lie(T)^{\ast}$.

Consider the line bundle $\mathcal{L}_{\mathcal I, \lambda_1,\dots,\lambda_r}$
over $Z_{\mathcal I}$ determined by integral weights $\lambda_1,\dots,\lambda_r$.
Then we have an integer vector
$\mathbf a = (\mathbf a^{(1)},\dots,\mathbf a^{(r)}) 
\in \Z^{N_1} \oplus \cdots \oplus \Z^{N_r}$
such that
$\eta^{\ast} \mathcal L_{\mathcal I, \lambda_1,\dots,\lambda_r}
= \mathcal{L}_{\mathbf i, {\mathbf a}}$ by 
Proposition~\ref{prop_line_bdle_over_fBS_and_BS}.
Let $\omega_{\mathbf i}'$, respectively $\omega_{\mathcal I}'$,
be a closed 2-form corresponding to the first
Chern class of the line bundle $\mathcal{L}_{\mathbf i, {\mathbf a}} \to 
Z_{\mathbf i}$, respectively $\mathcal{L}_{\mathcal I, \lambda_1,\dots,\lambda_r} \to Z_{\mathcal I}$.
By taking averages of $\omega_{\mathbf i}'$ and $\omega_{\mathcal I}'$ by corresponding torus actions
we have the following two 2-forms:
\begin{equation}\label{eq_def_of_omega}
{\omega}_{\mathbf i} \coloneqq \int_{a \in \widetilde{\mathbf{T}}}
(a^{\ast} \omega_{\mathbf i}') da
\quad \text{ and } \quad
{\omega}_{\mathcal I} \coloneqq \int_{t \in {\mathbf{T}}}
(t^{\ast} \omega_{\mathcal I}') dt.
\end{equation}
Then the form ${\omega}_{\mathbf i}$, respectively 
${\omega}_{\mathcal I}$, is a 
$\widetilde{\mathbf T}$-invariant, respectively 
$\mathbf T$-invariant, 
closed 2-form on 
$(Z_{\mathbf i}, \widetilde{\mathbf T})$, respectively
$(Z_{\mathcal I}, {\mathbf T})$.
Since compact tori $\widetilde{\mathbf T}$ and $\mathbf T$ are connected,
we have that 
\begin{equation}\label{eq_omega_and_omega_prime}
[\omega_{\mathbf i}] = [\omega_{\mathbf i}'] \text{ in } H^2(Z_{\bf i}; \mathbb{R}), \quad
[\omega_{\mathcal I}] = [\omega_{\mathcal I}'] \text{ in }
H^2(Z_{\mathcal I}; \mathbb{R})
\end{equation}
(see~\cite[Corollary~B.13]{GGK02}).

Grossberg and Karshon~\cite{GrKa94}  proved that
the Duistermaat--Heckman measure of the presymplectic manifold
$(Z_{\mathbf i}, {\omega}_{\mathbf i}, \widetilde{\mathbf T})$
can be computed by considering a combinatorial object, called a \defi{Grossberg--Karshon twisted cube}.
We use it to compute the Duistermaat--Heckman measure
of the presymplectic manifold $(Z_{\mathcal I}, \omega_{\mathcal I},
\bf T)$.

We recall from~\cite[\S 2.5]{GrKa94} the definition of Grossberg--Karshon twisted cubes. 
Let $\mathbf i =(i_1,\dots,i_N)$ be a sequence of elements
in $[n]$ and $\mathbf a = (a_1,\dots,a_N) \in \Z^N$.
A Grossberg--Karshon twisted cube is a pair $(C(\mathbf i, \mathbf a), \rho)$, where 
$C(\mathbf i, \mathbf a)$
is a subset of $\R^N$ and $\rho \colon \R^N \to \R$ is a density function
with support equal to $C(\mathbf i, \mathbf a)$. 
We define the following functions on $\R^N$:
\begin{align*}
	A_N(x) &= A_N(x_1,\dots,x_N) = -\langle 
	a_N \varpi_{{i}_N},  \alpha_{{i}_N}^{\vee}
	\rangle, \\
	A_{\ell}(x) &= A_{\ell}(x_1,\dots,x_N) \\
	&= -\langle 
	a_{\ell} \varpi_{{i}_{\ell}} + \cdots + a_N \varpi_{{i}_N}, \alpha_{{i}_{\ell}}^{\vee}
	\rangle
	- \sum_{j > \ell} 
	\langle \alpha_{{i}_j}, \alpha_{{i}_{\ell}}^{\vee} \rangle x_j  \quad \text{ for } 1 \leq \ell \leq N-1. 
\end{align*}
We also define a function $\text{sign} \colon \mathbb{R} 
\to \{\pm 1\}$ as $\text{sign}(x) = -1$ for $x \leq 0$ and 
$\text{sign}(x) = 1$ for $x >0$. 
\begin{definition}
	Let $C(\mathbf i, \mathbf a)$ be the following 
	subset of $\mathbb{R}^N$:
	\begin{align*}
		C(\mathbf i, \mathbf a)\coloneqq 
		\{x = (x_1,\dots,x_N) \in \R^N &\mid 
		A_j(x) \leq x_j \leq 0 \text{ or }
		0 < x_j < A_j(x) \\
		& \qquad \quad \text{ for } 1 \leq j \leq N
		\}.
	\end{align*}
	We define a density function $\rho \colon \R^N \to \R$ 
	whose support is $C(\mathbf i, \mathbf a)$ and 
	$\rho(x) = (-1)^N \text{sign}(x_1) \cdots \text{sign}(x_N)$ on 
	the set $C(\mathbf i, \mathbf a)$. 
	We call the pair $(C(\mathbf i, \mathbf a), \rho)$ the 
	\defi{Grossberg--Karshon twisted cube associated to $\mathbf i$ and $\mathbf a$}. 
	Also we define a measure 
	\[
	m_{C(\mathbf i, \mathbf a)} = \rho(\alpha)|d \alpha|,
	\]
	where 
	$|d \alpha|$ is Lebesgue measure in $\R^N$. 
\end{definition}
Now we have the following theorem.
\begin{theorem}\label{thm_DH_measure_of_fBS}
	Let $(Z_{\mathcal I}, \omega_{\mathcal I}, \mathbf T)$ be as above, and
	let $\Phi \colon Z_{\mathcal I} \to \R^{m_1+ \cdots +m_r}$
	be a moment map of $(Z_{\mathcal I}, \omega_{\mathcal I}, \mathbf T)$.
	Then there is a Grossberg--Karshon twisted cube $(C(\mathbf i, \mathbf a), \rho)$
	and an affine projection $L \colon \R^N \to \R^{m_1+ \cdots +m_r}$
	such that the Duistermaat--Heckman measure in ${\rm Lie}(\mathbf{T})^{\ast} \cong \R^{m_1+\cdots+m_r}$
	is $L_{\ast} m_{C(\mathbf i, \mathbf a)}$. 
\end{theorem}
To give a proof, we need the following theorem of 
Grossberg and Karshon.
\begin{theorem}[{\cite[Theorem 2]{GrKa94}}]
	\label{thm_GK_DH_measure}
	Let $\widetilde{\Phi} \colon Z_{\mathbf i} \to \R^N$ 
	be a moment map of~$(Z_{\mathbf i}, {\omega}_{\mathbf i}, \widetilde{\mathbf T})$. 
	Then the Duistermaat--Heckman measure in ${\rm Lie}(\widetilde{\mathbf T})^{\ast} \cong 
	\R^N$ coincides with the measure $m_{C(\mathbf i, \mathbf a)}$ for the Grossberg--Karshon twisted cube ~$C(\mathbf i, \mathbf a)$. 
\end{theorem}
\begin{proof}[Proof of Theorem~\ref{thm_DH_measure_of_fBS}]
	Suppose that $\mathbf i \in [n]^N$ defines a Bott--Samelson manifold 
	$Z_{\mathbf i}$ which has a birational morphism 
	$\eta \colon Z_{\mathbf i} \to Z_{\mathcal I}$.
	For given weights $\lambda_1,\dots,\lambda_r$, 
	let $\mathbf a \in \Z^N$ be an integer vector such that 
	$\eta^{\ast} \mathcal L_{\mathcal I, \lambda_1,\dots,\lambda_r} = \mathcal{L}_{\mathbf i,
		\mathbf a}$. 
	Consider the pullback of $\omega_{\mathcal I}$ under the map
	$\eta$. Then we have 
	$[\omega_{\mathbf i}] = [\eta^{\ast}(\omega_{\mathcal I})]$ in $H^2(Z_{\mathbf i}; \R)$ by~\eqref{eq_omega_and_omega_prime}.
	
	Now we have the following diagram which does not necessarily commute
	because two forms $\eta^{\ast} \omega_{\mathcal I}$ and  $\omega_{\mathbf i}$
	do not necessarily coincide because of taking averages:
	\[
	\begin{tikzcd}
	Z_{\mathbf i} \arrow[r, "\widetilde{\Phi}"]
	\arrow[d, "\eta"]
	& \R^N  \cong \Lie (\widetilde{\mathbf T})^{\ast}
	\arrow[d, "L"]\\
	Z_{\mathcal I} \arrow[r, "\Phi"]
	& \R^{m_1+ \cdots + m_r}  \cong \Lie(\mathbf T)^{\ast}
	\end{tikzcd}
	\]
	Here, the map $L \colon \R^N \to \R^{m_1+ \cdots +m_r}$ 
	is defined as $dA^{\ast}$, where
	$A \colon \mathbf{T}\rightarrow \widetilde{\mathbf T} $ 
	in~\eqref{eq_def_of_A}. 
	
	But one can see that $L \circ \widetilde{\Phi}$,
	respectively $\Phi \circ \eta$, is a moment map for 
	$(Z_{\mathbf i},  \omega_{\mathbf i}, \mathbf T)$, respectively
	$(Z_{\mathbf i},  \eta^{\ast} \omega_{\mathcal I}, \mathbf T)$.
	Recall from~\cite[Theorem 1]{GrKa94} that
	the push-forward of Liouville measure only depends on the cohomology
	class, so we have that
	\[
	(L \circ \widetilde{\Phi})_{\ast} \omega_{\mathbf i}^N
	= (\Phi \circ \eta)_{\ast} (\eta^{\ast} \omega_{\mathcal I})^N
	= \Phi_{\ast} \omega_{\mathcal I}^N.
	\]
		Here the last equality holds since $\eta$ induces
		a diffeomorphism between Zariski open dense subsets, 
		and a Zariski closed subset is measure zero. 
	By Theorem~\ref{thm_GK_DH_measure}, we have that
	$\Phi_{\ast} \omega_{\mathcal I}^N/N! = L_{\ast} m_{C(\mathbf i, \mathbf a)}$,
	so the result follows.
\end{proof}
\begin{example}\label{example_TC_projection}
	Let $G = \SL(4)$, $\mathcal I= (\{1,2\}, \{3\})$ and $\mathbf i =(1,2,1,3)$.
	The projection map $L = (dA)^{\ast} \colon \mathbb{R}^4 \to \mathbb{R}^3$ is
	given by the integer matrix
	\[
	\begin{bmatrix}
	1 & 0 & 1 & 0 \\
	0 & 1 & 0 & 0 \\
	0 & 0 & 0 & 1
	\end{bmatrix}
	\]
	as in Example~\ref{example_projection_A}.
	In Figure~\ref{figure_projections} we draw figures for four different pairs of 	weights $(\lambda_1,\lambda_2) = (2 \varpi_1 + 4\varpi_2, 2\varpi_3),
	(\varpi_1 + 4 \varpi_2, 2 \varpi_3),
	(2\varpi_1 + 4 \varpi_2, \varpi_3),
	(2\varpi_1 + 3 \varpi_2, 2\varpi_3)$
	which determine line bundles 
	$\mathcal{L}_{\mathcal I, \lambda_1,\lambda_2}$.
	The polytope in Figure~\ref{figure_projections}-(1) has eight facets. When
	we change an integer vector $(\lambda_1,\lambda_2)$ a little bit,  some facets move as one can
	see in the figure. In Figure~\ref{figure_projections}-(2), (3), (4) the red dots
	represent vertices of the projection for the corresponding integer vector,
	and the blue dots represent vertices of the projection for 
	$(\lambda_1,\lambda_2) = (2\varpi_1 + 4 \varpi_2, 2\varpi_3)$.
	For pairs $(2 \varpi_1 + 4\varpi_2, 2\varpi_3)$,
	$(\varpi_1 + 4 \varpi_2, 2 \varpi_3)$, and
	$(2\varpi_1 + 4 \varpi_2, \varpi_3)$, the projections are honest polytopes
	while the projection for $(2\varpi_1 +3\varpi_2, 2\varpi_3)$ is not.
	\begin{figure}[t]
		\tabcolsep=10pt\renewcommand*{\arraystretch}{2}
		\begin{subfigure}{0.45\linewidth}
			\begin{tikzpicture}[x = {(4mm, 1.8mm)}, y={(-4mm,1.8mm)},z={(0cm,2cm)}, scale=0.8]
			\definecolor{mycolor1}{rgb}{ 0 ,   0.4470  ,  0.7410};
			\definecolor{mycolor2}{rgb}{ 0.8500 ,   0.3250 ,   0.0980};
			\definecolor{mycolor3}{rgb}{ 0.9290 ,   0.6940  ,  0.1250};
			\definecolor{mycolor4}{rgb}{0.4940 ,   0.1840  ,  0.5560};	
			\definecolor{mycolor5}{rgb}{  0  ,  0.4470  ,  0.7410};
			\definecolor{mycolor6}{rgb}{0.4660  ,  0.6740  ,  0.1880};
			\definecolor{mycolor7}{rgb}{0.3010 ,   0.7450  ,  0.9330};
			\definecolor{mycolor8}{rgb}{ 0.6350 ,   0.0780  ,  0.1840};

			\coordinate (1) at (-6,0,0);
			\coordinate (2) at (-4,0,-2);
			\coordinate (3) at (-6,2,0);
			\coordinate (4) at (-4,4,-2);
			\coordinate (5) at (0,0,0);
			\coordinate (6) at (0,0,-2);
			\coordinate (7) at (-2,2,0);
			\coordinate (8) at (-4,4,-2);
			\coordinate (9) at (-2,-4,0);
			\coordinate (10) at (-4,0,-2);
			\coordinate (11) at (0,-4,0);
			\coordinate (12) at (0,0,-2);
			
			\begin{scope}[color=gray!50, thin]
			\foreach \xi in {0,-2,-4,-6} { \draw (\xi, 4,0) -- (\xi, 4,-2) -- (\xi, -4, -2); }%
			\foreach \yi in {-4,-2,0,2,4} {\draw (-6,\yi,-2) -- (0,\yi,-2) -- (0,\yi,0);}%
			\foreach \zi in {0,-0.5,-1,-1.5,-2} {\draw (-6,4,\zi) -- (0,4,\zi) -- (0,-4,\zi);}%
			\end{scope}
			
			\draw[thin] (-6,4,0) -- (-6,4,-2) -- (-6,-4,-2) -- (0,-4,-2);
			\foreach \xi in {0,-2,-4,-6} {\draw (\xi,-4,-2) -- (\xi, -4.1, -2) node[anchor=west] {\tiny{\xi}};}
			\foreach \yi in {-4,-2,0,2,4} {\draw (-6,\yi,-2) -- (-6.1,\yi, -2) node[anchor=north] {\tiny{\yi}};}
			\foreach \zi in {0,-0.5,-1,-1.5,-2} {\draw (-6,4,\zi) -- (-6,4.1,\zi) node[anchor=east] {\tiny{\zi}};}
			
			\node at (-3,-5,-2) {$x$};
			\node at (-8,-1,-2) {$y$};
			\node at (-6,6,-1) {$z$};

			\draw[thick, mycolor1] (1) -- (9) -- (11) -- (5) -- (7) -- (3) -- cycle;
			\draw[thick, mycolor2] (5) -- (6) -- (8) -- (7) --cycle;
			\draw[thick, mycolor3] (5) -- (6) -- (12) -- (11) -- cycle;
			\draw[thick, mycolor4] (1) -- (3)-- (4) -- (2) -- cycle;
			\draw[thick, mycolor5] (2) -- (10) -- (12) -- (6) -- (8) -- (4) -- cycle;
			\draw[thick, mycolor6] (3)-- (7)-- (8)--(4) --cycle;
			\draw[thick, mycolor7] (9) -- (10) --(12)--(11)--cycle;
			\draw[thick, mycolor8] (9)--(10)--(2)--(1)--cycle;
			
			\foreach \x in {1,...,12}{
				\node[circle,fill=blue,inner sep=1pt] at (\x) {};
			}
			
			\end{tikzpicture}
			\caption{$(\lambda_1,\lambda_2) = (2\varpi_1+4\varpi_2,2\varpi_3)$.}
		\end{subfigure}
		~~~~
		\begin{subfigure}{0.45\linewidth}
			\begin{tikzpicture}[x = {(4mm, 1.8mm)}, y={(-4mm,1.8mm)},z={(0cm,2cm)}, scale=0.8]
			\definecolor{mycolor1}{rgb}{ 0 ,   0.4470  ,  0.7410};
			\definecolor{mycolor2}{rgb}{ 0.8500 ,   0.3250 ,   0.0980};
			\definecolor{mycolor3}{rgb}{ 0.9290 ,   0.6940  ,  0.1250};
			\definecolor{mycolor4}{rgb}{0.4940 ,   0.1840  ,  0.5560};	
			\definecolor{mycolor5}{rgb}{  0  ,  0.4470  ,  0.7410};
			\definecolor{mycolor6}{rgb}{0.4660  ,  0.6740  ,  0.1880};
			\definecolor{mycolor7}{rgb}{0.3010 ,   0.7450  ,  0.9330};
			\definecolor{mycolor8}{rgb}{ 0.6350 ,   0.0780  ,  0.1840};

			\coordinate (1) at (-6,0,0);
			\coordinate (2) at (-4,0,-2);
			\coordinate (3) at (-6,2,0);
			\coordinate (4) at (-4,4,-2);
			\coordinate (5) at (0,0,0);
			\coordinate (6) at (0,0,-2);
			\coordinate (7) at (-2,2,0);
			\coordinate (8) at (-4,4,-2);
			\coordinate (9) at (-2,-4,0);
			\coordinate (10) at (-4,0,-2);
			\coordinate (11) at (0,-4,0);
			\coordinate (12) at (0,0,-2);
			
			\coordinate (n1) at (-5,0,0);
			\coordinate (n2) at (-3,0,-2);
			\coordinate (n3) at (-5,1,0);
			\coordinate (n4) at (-3,3,-2);
			\coordinate (n5) at (0,0,0);
			\coordinate (n6) at (0,0,-2);
			\coordinate (n7) at (-1,1,0);
			\coordinate (n8) at (-3,3,-2);
			\coordinate (n9) at (-1,-4,0);
			\coordinate (n10) at (-3,0,-2);
			\coordinate (n11) at (0,-4,0);
			\coordinate (n12) at (0,0,-2);
			
			\begin{scope}[color=gray!50, thin]
			\foreach \xi in {0,-2,-4,-6} { \draw (\xi, 4,0) -- (\xi, 4,-2) -- (\xi, -4, -2); }%
			\foreach \yi in {-4,-2,0,2,4} {\draw (-6,\yi,-2) -- (0,\yi,-2) -- (0,\yi,0);}%
			\foreach \zi in {0,-0.5,-1,-1.5,-2} {\draw (-6,4,\zi) -- (0,4,\zi) -- (0,-4,\zi);}%
			\end{scope}
			
			\draw[thin] (-6,4,0) -- (-6,4,-2) -- (-6,-4,-2) -- (0,-4,-2);
			\foreach \xi in {0,-2,-4,-6} {\draw (\xi,-4,-2) -- (\xi, -4.1, -2) node[anchor=west] {\tiny{\xi}};}
			\foreach \yi in {-4,-2,0,2,4} {\draw (-6,\yi,-2) -- (-6.1,\yi, -2) node[anchor=north] {\tiny{\yi}};}
			\foreach \zi in {0,-0.5,-1,-1.5,-2} {\draw (-6,4,\zi) -- (-6,4.1,\zi) node[anchor=east] {\tiny{\zi}};}
			
			\node at (-3,-5,-2) {$x$};
			\node at (-8,-1,-2) {$y$};
			\node at (-6,6,-1) {$z$};

			\draw[dashed, mycolor1] (1) -- (9) -- (11) -- (5) -- (7) -- (3) -- cycle;
			\draw[dashed,mycolor2] (5) -- (6) -- (8) -- (7) --cycle;
			\draw[dashed,mycolor3] (5) -- (6) -- (12) -- (11) -- cycle;
			\draw[dashed,mycolor4] (1) -- (3)-- (4) -- (2) -- cycle;
			\draw[dashed,mycolor5] (2) -- (10) -- (12) -- (6) -- (8) -- (4) -- cycle;
			\draw[dashed,mycolor6] (3)-- (7)-- (8)--(4) --cycle;
			\draw[dashed,mycolor7] (9) -- (10) --(12)--(11)--cycle;
			\draw[dashed,mycolor8] (9)--(10)--(2)--(1)--cycle;
			
			\draw[thick, mycolor1] (n1) -- (n9) -- (n11) -- (n5) -- (n7) -- (n3) -- cycle;
			\draw[thick, mycolor2] (n5) -- (n6) -- (n8) -- (n7) --cycle;
			\draw[thick, mycolor3] (n5) -- (n6) -- (n12) -- (n11) -- cycle;
			\draw[thick, mycolor4] (n1) -- (n3)-- (n4) -- (n2) -- cycle;
			\draw[thick, mycolor5] (n2) -- (n10) -- (n12) -- (n6) -- (n8) -- (n4) -- cycle;
			\draw[thick, mycolor6] (n3)-- (n7)-- (n8)--(n4) --cycle;
			\draw[thick, mycolor7] (n9) -- (n10) --(n12)--(n11)--cycle;
			\draw[thick, mycolor8] (n9)--(n10)--(n2)--(n1)--cycle;
			
			\foreach \x in {1,...,12}{
				\node[circle,fill=blue,inner sep=1pt] at (\x) {};
				\node[circle,fill=red,inner sep=1pt] at (n\x) {};
			}
			\end{tikzpicture}
			\caption{$(\lambda_1,\lambda_2) = (\varpi_1+4\varpi_2,2\varpi_3)$.}
		\end{subfigure}
		\par\bigskip\par\bigskip
		\begin{subfigure}{0.45\linewidth}
			\begin{tikzpicture}[x = {(4mm, 1.8mm)}, y={(-4mm,1.8mm)},z={(0cm,2cm)}, scale=0.8]
			\definecolor{mycolor1}{rgb}{ 0 ,   0.4470  ,  0.7410};
			\definecolor{mycolor2}{rgb}{ 0.8500 ,   0.3250 ,   0.0980};
			\definecolor{mycolor3}{rgb}{ 0.9290 ,   0.6940  ,  0.1250};
			\definecolor{mycolor4}{rgb}{0.4940 ,   0.1840  ,  0.5560};	
			\definecolor{mycolor5}{rgb}{  0  ,  0.4470  ,  0.7410};
			\definecolor{mycolor6}{rgb}{0.4660  ,  0.6740  ,  0.1880};
			\definecolor{mycolor7}{rgb}{0.3010 ,   0.7450  ,  0.9330};
			\definecolor{mycolor8}{rgb}{ 0.6350 ,   0.0780  ,  0.1840};

			\coordinate (1) at (-6,0,0);
			\coordinate (2) at (-4,0,-2);
			\coordinate (3) at (-6,2,0);
			\coordinate (4) at (-4,4,-2);
			\coordinate (5) at (0,0,0);
			\coordinate (6) at (0,0,-2);
			\coordinate (7) at (-2,2,0);
			\coordinate (8) at (-4,4,-2);
			\coordinate (9) at (-2,-4,0);
			\coordinate (10) at (-4,0,-2);
			\coordinate (11) at (0,-4,0);
			\coordinate (12) at (0,0,-2);
			
			\coordinate (n1) at (-6,0,0);
			\coordinate (n2) at (-5,0,-1);
			\coordinate (n3) at (-6,2,0);
			\coordinate (n4) at (-5,3,-1);
			\coordinate (n5) at (0,0,0);
			\coordinate (n6) at (0,0,-1);
			\coordinate (n7) at (-2,2,0);
			\coordinate (n8) at (-3,3,-1);
			\coordinate (n9) at (-2,-4,0);
			\coordinate (n10) at (-3,-2,-1);
			\coordinate (n11) at (0,-4,0);
			\coordinate (n12) at (0,-2,-1);
			
			\begin{scope}[color=gray!50, thin]
			\foreach \xi in {0,-2,-4,-6} { \draw (\xi, 4,0) -- (\xi, 4,-2) -- (\xi, -4, -2); }%
			\foreach \yi in {-4,-2,0,2,4} {\draw (-6,\yi,-2) -- (0,\yi,-2) -- (0,\yi,0);}%
			\foreach \zi in {0,-0.5,-1,-1.5,-2} {\draw (-6,4,\zi) -- (0,4,\zi) -- (0,-4,\zi);}%
			\end{scope}
			
			\draw[thin] (-6,4,0) -- (-6,4,-2) -- (-6,-4,-2) -- (0,-4,-2);
			\foreach \xi in {0,-2,-4,-6} {\draw (\xi,-4,-2) -- (\xi, -4.1, -2) node[anchor=west] {\tiny{\xi}};}
			\foreach \yi in {-4,-2,0,2,4} {\draw (-6,\yi,-2) -- (-6.1,\yi, -2) node[anchor=north] {\tiny{\yi}};}
			\foreach \zi in {0,-0.5,-1,-1.5,-2} {\draw (-6,4,\zi) -- (-6,4.1,\zi) node[anchor=east] {\tiny{\zi}};}
			
			\node at (-3,-5,-2) {$x$};
			\node at (-8,-1,-2) {$y$};
			\node at (-6,6,-1) {$z$};

			\draw[dashed, mycolor1] (1) -- (9) -- (11) -- (5) -- (7) -- (3) -- cycle;
			\draw[dashed,mycolor2] (5) -- (6) -- (8) -- (7) --cycle;
			\draw[dashed,mycolor3] (5) -- (6) -- (12) -- (11) -- cycle;
			\draw[dashed,mycolor4] (1) -- (3)-- (4) -- (2) -- cycle;
			\draw[dashed,mycolor5] (2) -- (10) -- (12) -- (6) -- (8) -- (4) -- cycle;
			\draw[dashed,mycolor6] (3)-- (7)-- (8)--(4) --cycle;
			\draw[dashed,mycolor7] (9) -- (10) --(12)--(11)--cycle;
			\draw[dashed,mycolor8] (9)--(10)--(2)--(1)--cycle;
			
			\draw[thick, mycolor1] (n1) -- (n9) -- (n11) -- (n5) -- (n7) -- (n3) -- cycle;
			\draw[thick, mycolor2] (n5) -- (n6) -- (n8) -- (n7) --cycle;
			\draw[thick, mycolor3] (n5) -- (n6) -- (n12) -- (n11) -- cycle;
			\draw[thick, mycolor4] (n1) -- (n3)-- (n4) -- (n2) -- cycle;
			\draw[thick, mycolor5] (n2) -- (n10) -- (n12) -- (n6) -- (n8) -- (n4) -- cycle;
			\draw[thick, mycolor6] (n3)-- (n7)-- (n8)--(n4) --cycle;
			\draw[thick, mycolor7] (n9) -- (n10) --(n12)--(n11)--cycle;
			\draw[thick, mycolor8] (n9)--(n10)--(n2)--(n1)--cycle;
			
			\foreach \x in {1,...,12}{
				\node[circle,fill=blue,inner sep=1pt] at (\x) {};
				\node[circle,fill=red,inner sep=1pt] at (n\x) {};
			}
			\end{tikzpicture}
			\caption{$(\lambda_1,\lambda_2) = (2\varpi_1+4\varpi_2,\varpi_3)$.}
		\end{subfigure}
		~~~~
		\begin{subfigure}{0.45\linewidth}
			\begin{tikzpicture}[x = {(4mm, 1.8mm)}, y={(-4mm,1.8mm)},z={(0cm,2cm)}, scale=0.8]
			\definecolor{mycolor1}{rgb}{ 0 ,   0.4470  ,  0.7410};
			\definecolor{mycolor2}{rgb}{ 0.8500 ,   0.3250 ,   0.0980};
			\definecolor{mycolor3}{rgb}{ 0.9290 ,   0.6940  ,  0.1250};
			\definecolor{mycolor4}{rgb}{0.4940 ,   0.1840  ,  0.5560};	
			\definecolor{mycolor5}{rgb}{  0  ,  0.4470  ,  0.7410};
			\definecolor{mycolor6}{rgb}{0.4660  ,  0.6740  ,  0.1880};
			\definecolor{mycolor7}{rgb}{0.3010 ,   0.7450  ,  0.9330};
			\definecolor{mycolor8}{rgb}{ 0.6350 ,   0.0780  ,  0.1840};

			\coordinate (1) at (-6,0,0);
			\coordinate (2) at (-4,0,-2);
			\coordinate (3) at (-6,2,0);
			\coordinate (4) at (-4,4,-2);
			\coordinate (5) at (0,0,0);
			\coordinate (6) at (0,0,-2);
			\coordinate (7) at (-2,2,0);
			\coordinate (8) at (-4,4,-2);
			\coordinate (9) at (-2,-4,0);
			\coordinate (10) at (-4,0,-2);
			\coordinate (11) at (0,-4,0);
			\coordinate (12) at (0,0,-2);
			
			\coordinate (n1) at (-5,0,0);
			\coordinate (n2) at (-3,0,-2);
			\coordinate (n3) at (-5,2,0);
			\coordinate (n4) at (-3,4,-2);
			\coordinate (n5) at (0,0,0);
			\coordinate (n6) at (0,0,-2);
			\coordinate (n7) at (-2,2,0);
			\coordinate (n8) at (-4,4,-2);
			\coordinate (n9) at (-2,-3,0);
			\coordinate (n10) at (-4,1,-2);
			\coordinate (n11) at (0,-3,0);
			\coordinate (n12) at (0,1,-2);
			
			\begin{scope}[color=gray!50, thin]
			\foreach \xi in {0,-2,-4,-6} { \draw (\xi, 4,0) -- (\xi, 4,-2) -- (\xi, -4, -2); }%
			\foreach \yi in {-4,-2,0,2,4} {\draw (-6,\yi,-2) -- (0,\yi,-2) -- (0,\yi,0);}%
			\foreach \zi in {0,-0.5,-1,-1.5,-2} {\draw (-6,4,\zi) -- (0,4,\zi) -- (0,-4,\zi);}%
			\end{scope}
			
			\draw[thin] (-6,4,0) -- (-6,4,-2) -- (-6,-4,-2) -- (0,-4,-2);
			\foreach \xi in {0,-2,-4,-6} {\draw (\xi,-4,-2) -- (\xi, -4.1, -2) node[anchor=west] {\tiny{\xi}};}
			\foreach \yi in {-4,-2,0,2,4} {\draw (-6,\yi,-2) -- (-6.1,\yi, -2) node[anchor=north] {\tiny{\yi}};}
			\foreach \zi in {0,-0.5,-1,-1.5,-2} {\draw (-6,4,\zi) -- (-6,4.1,\zi) node[anchor=east] {\tiny{\zi}};}
			
			\node at (-3,-5,-2) {$x$};
			\node at (-8,-1,-2) {$y$};
			\node at (-6,6,-1) {$z$};

			\draw[dashed, mycolor1] (1) -- (9) -- (11) -- (5) -- (7) -- (3) -- cycle;
			\draw[dashed,mycolor2] (5) -- (6) -- (8) -- (7) --cycle;
			\draw[dashed,mycolor3] (5) -- (6) -- (12) -- (11) -- cycle;
			\draw[dashed,mycolor4] (1) -- (3)-- (4) -- (2) -- cycle;
			\draw[dashed,mycolor5] (2) -- (10) -- (12) -- (6) -- (8) -- (4) -- cycle;
			\draw[dashed,mycolor6] (3)-- (7)-- (8)--(4) --cycle;
			\draw[dashed,mycolor7] (9) -- (10) --(12)--(11)--cycle;
			\draw[dashed,mycolor8] (9)--(10)--(2)--(1)--cycle;
			
			\draw[thick, mycolor1] (n1) -- (n9) -- (n11) -- (n5) -- (n7) -- (n3) -- cycle;
			\draw[thick, mycolor2] (n5) -- (n6) -- (n8) -- (n7) --cycle;
			\draw[thick, mycolor3] (n5) -- (n6) -- (n12) -- (n11) -- cycle;
			\draw[thick, mycolor4] (n1) -- (n3)-- (n4) -- (n2) -- cycle;
			\draw[thick, mycolor5] (n2) -- (n10) -- (n12) -- (n6) -- (n8) -- (n4) -- cycle;
			\draw[thick, mycolor6] (n3)-- (n7)-- (n8)--(n4) --cycle;
			\draw[thick, mycolor7] (n9) -- (n10) --(n12)--(n11)--cycle;
			\draw[thick, mycolor8] (n9)--(n10)--(n2)--(n1)--cycle;
			
			\foreach \x in {1,...,12}{
				\node[circle,fill=blue,inner sep=1pt] at (\x) {};
				\node[circle,fill=red,inner sep=1pt] at (n\x) {};
			}
			
			\end{tikzpicture}
			\caption{$(\lambda_1,\lambda_2) = (2\varpi_1+3\varpi_2,2\varpi_3)$.}
		\end{subfigure}
		\caption{The projection images of Grossberg--Karshon twisted cubes.}
		\label{figure_projections}
	\end{figure}
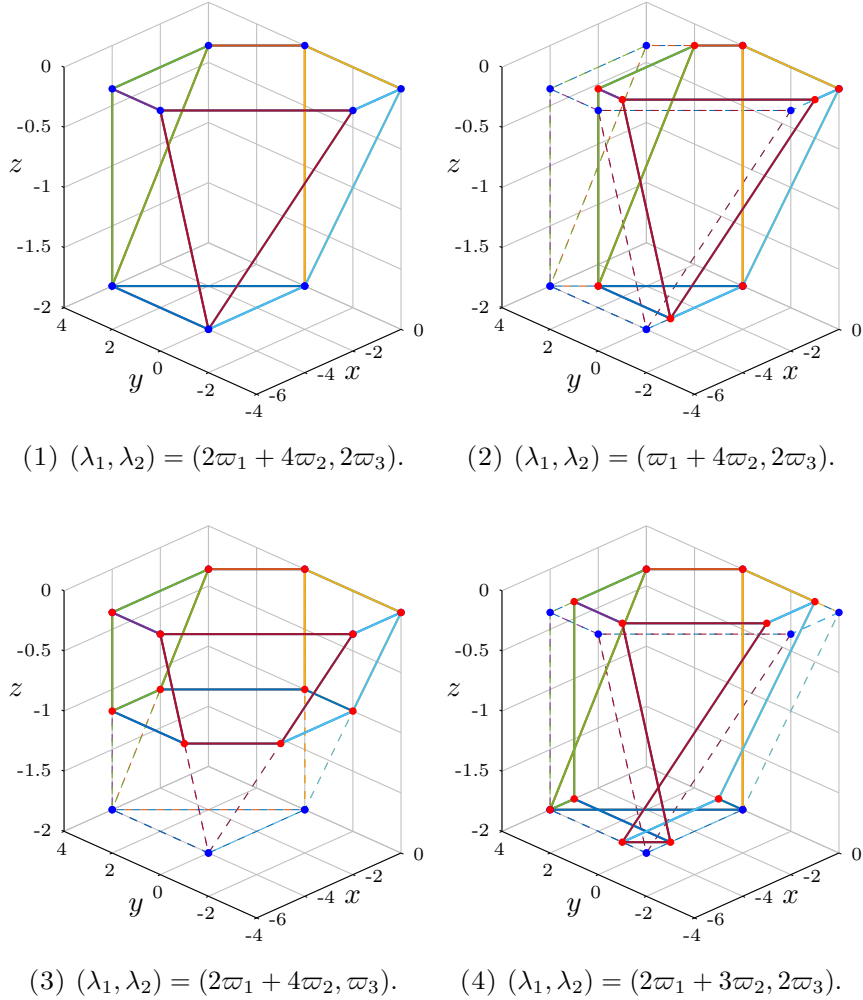
\end{example}
\begin{remark}
	Note that a Grossberg--Karshon twisted cube is neither closed not convex. When the Grossberg--Karshon twisted cube is a closed convex polytope, then we say it is \defi{untwisted}. In~\cite{Lee_GKcube}, an interpretation of untwistedness of Grossberg--Karshon twisted cubes $C(\mathbf i, \mathbf a)$ using combinatorics of $\mathbf i$ and $\mathbf a$ is provided. (Also, see~\cite{HaYa_GKcube, HaLee_GKcube}.) Using the result~\cite[Theorem~1]{Lee_GKcube}, Grossberg--Karshon twisted cubes appearing in Example~\ref{example_TC_projection} are all \textit{twisted}. However, their projections can be honest polytopes as we saw in Figure~\ref{figure_projections}. Determining whether the projection of a Grossberg--Karshon twisted cube is an honest polytope is a still open problem.
\end{remark}



\end{document}